\documentclass{amsart}

\usepackage{epsfig}
\usepackage{amsmath}
\usepackage{amssymb}
\usepackage[all]{xy}

\usepackage{tabularx}

\newcolumntype{C}{>{\centering\arraybackslash}X}

\def\comment#1{{\sf{[#1]}}}

\def\Z{{\mathbb Z}}
\def\Q{{\mathbb Q}}

\def\R{{\mathbb R}}
\def\C{{\mathbb C}}
\def\P{{\mathbb P}}

\def\bA{{\mathbb A}}
\def\bB{{\mathbb B}}
\def\D{{\mathbb D}}
\def\bE{{\mathbb E}}
\def\bG{{\mathbb G}}
\def\H{{\mathbb H}}
\def\L{{\mathbb L}}

\def\bU{{\mathbb U}}
\def\V{{\mathbb V}}

\def\A{{\mathcal A}}
\def\cD{{\mathcal D}}
\def\E{{\mathcal E}}
\def\F{{\mathcal F}}
\def\M{{\mathcal M}}
\def\cN{{\mathcal N}}
\def\O{{\mathcal O}}
\def\U{{\mathcal U}}
\def\cC{{\mathcal C}}
\def\cG{{\mathcal G}}
\def\cH{{\mathcal H}}

\def\K{{\mathcal K}}
\def\cL{{\mathcal L}}
\def\cP{{\mathcal P}}
\def\cR{{\mathcal R}}
\def\cS{{\mathcal S}}
\def\cV{{\mathcal V}}
\def\X{{\mathcal X}}

\def\a{{\mathfrak a}}

\def\f{{\mathfrak f}}
\def\g{{\mathfrak g}}
\def\h{{\mathfrak h}}
\def\k{{\mathfrak k}}
\def\m{{\mathfrak m}}
\def\n{{\mathfrak n}}
\def\p{{\mathfrak p}}
\def\r{{\mathfrak r}}
\def\s{{\mathfrak s}}
\def\u{{\mathfrak u}}

\def\B{\mathfrak B}
\def\fM{\mathfrak M}

\def\d{{\delta}}
\def\e{{\epsilon}}
\def\w{{\omega}}

\def\G{{\Gamma}}

\def\rhotilde{{\tilde{\rho}}}

\def\ctilde{{\tilde{c}}}
\def\kbar{{\overline{k}}}
\def\xbar{{\overline{x}}}

\def\Ebar{{\overline{E}}}
\def\Hbar{{\overline{\cH}}}

\def\Mbar{\overline{\M}}
\def\cEbar{{\overline{\E}}}

\def\nablabar{{\overline{\nabla}}}

\def\eop{\check{\e}}
\def\eeop{\check{\mathbf{e}}}

\def\vv{{\vec{\mathsf v}}}
\def\ww{{\vec{\mathsf w}}}
\def\uu{{\vec{1}}}
\def\tate{{\vec{\mathsf{t}}}}

\def\bp{{\boldsymbol{\p}}}

\def\Pol{{\mathbf{Pol}}}

\def\aa{{\mathbf a}}
\def\bb{{\mathbf b}}
\def\ee{{\mathbf e}}

\def\zz{{\mathbf z}}
\def\br{{\mathbf r}}

\def\bv{\mathbf{v}}
\def\bw{\mathbf{w}}
\def\v{\mathbf{v}}

\def\bmu{{\boldsymbol{\mu}}}
\def\bphi{{\boldsymbol{\varphi}}}
\def\be{{\boldsymbol{\epsilon}}}
\def\ssigma{{\boldsymbol{\sigma}}}

\def\adual{{\check{\aa}}}
\def\bdual{{\check{\bb}}}

\def\Hdual{{\check{H}}}

\def\Qbar{{\overline{\Q}}}

\def\Ql{{\Q_\ell}}

\def\Zl{{\Z_\ell}}

\def\Vec{{\mathsf{Vec}}}
\def\Rep{{\mathsf{Rep}}}
\def\MHS{{\mathsf{MHS}}}

\def\MTM{{\mathsf{MTM}}}
\def\MEM{{\mathsf{MEM}}}
\def\MM{{\mathsf{MM}}}

\def\sl{\mathfrak{sl}}

\def\gl{\mathfrak{gl}}

\def\SL{{\mathrm{SL}}}
\def\GL{{\mathrm{GL}}}
\def\Gm{{\mathbb{G}_m}}
\def\Ga{{\mathbb{G}_a}}

\def\SLtilde{\widetilde{\SL}}

\def\cGhat{\widehat{\cG}}

\def\grt{{\mathfrak{grt}}}
\def\GRT{{\mathrm{GRT}}}

\def\what{{\mathbf{w}}}

\def\DR{{\mathrm{DR}}}
\def\et{\mathrm{\acute{e}t}}

\def\cris{\mathrm{cris}}
\def\wtd{\mathrm{wtd}}
\def\rel{\mathrm{rel}}
\def\cusp{\mathrm{cusp}}
\def\eis{{\mathrm{eis}}}
\def\geom{\mathrm{geom}}
\def\mot{\mathrm{mot}}
\def\prol{{(\ell)}}
\def\top{\mathrm{top}}
\def\un{{\mathrm{un}}}
\def\op{\mathrm{op}}

\def\ss{\mathrm{ss}}
\def\Mod{\mathrm{Mod}}
\def\elliptic{\mathrm{ell}}
\def\class{\mathrm{class}}
\def\res{\mathrm{res}}

\def\Fr{\mathcal{F}}
\def\Frbar{{\overline{\Fr}}}

\def\an{\mathrm{an}}

\def\cts{\mathrm{cts}}
\def\fte{\mathrm{fte}}
\def\sm{\mathrm{sm}}

\def\reg{{\sf{reg}}}
\def\real{{\sf{real}}}

\def\l{{\ell}}
\def\Tate{{\mathrm{Tate}}}
\def\sr{{\mathsf r}}

\def\dot{{\bullet}}
\def\bs{\backslash}
\def\bbs{{\bs\negthickspace \bs}}

\def\ll{\langle\langle}
\def\rr{\rangle\rangle}

\def\Pminus{{\P^1-\{0,1,\infty\}}}

\def\Sdual{{\check{S}}}

\def\Het{H_{\et}}

\def\Hfte{H_\sm}

\def\blank{\phantom{x}}
\def\round#1{{(#1)}}

\newcommand\im{\operatorname{im}} 
\newcommand\id{\operatorname{id}}
\newcommand\ad{\operatorname{ad}}

\newcommand\Span{\operatorname{span}}
\newcommand\Spec{\operatorname{Spec}}
\newcommand\Gal{\operatorname{Gal}}
\newcommand\Hom{\operatorname{Hom}}

\newcommand\Ext{\operatorname{Ext}}
\newcommand\End{\operatorname{End}}
\newcommand\Aut{\operatorname{Aut}}
\newcommand\Diff{\operatorname{Diff}}

\newcommand\Der{\operatorname{Der}}
\newcommand\Gr{\operatorname{Gr}}

\newcommand\Isom{\operatorname{Isom}}
\newcommand\Res{\operatorname{Res}}

\newcommand\Inn{\operatorname{Inn}}
\newcommand\OutDer{\operatorname{OutDer}}
\newcommand\diag{\operatorname{diag}}
\newcommand\Lie{\operatorname{Lie}}

\newcommand\Sym{\operatorname{Sym}}

\renewcommand\Im{\operatorname{Im}}

\numberwithin{equation}{section}

\newtheorem{theorem}{Theorem}[section]
\newtheorem{lemma}[theorem]{Lemma}
\newtheorem{proposition}[theorem]{Proposition}
\newtheorem{corollary}[theorem]{Corollary}

\theoremstyle{definition}
\newtheorem{definition}[theorem]{Definition}
\newtheorem{example}[theorem]{Example}
\newtheorem{problem}{Problem} 

\theoremstyle{remark}
\newtheorem{remark}[theorem]{Remark}
\newtheorem{notation}[theorem]{Notation}
\newtheorem{question}[theorem]{Question}

\newtheorem{convention}[theorem]{Convention}
\newtheorem{conjecture}[theorem]{Conjecture}

\begin{document}

\title{Universal Mixed Elliptic Motives}

\author{Richard Hain}
\address{Department of Mathematics\\ Duke University\\
Durham, NC 27708-0320}
\email{hain@math.duke.edu}


\author{Makoto Matsumoto}
\address{Graduate School of Sciences\\
Hiroshima University\\
Hiroshima 739-8526, Japan}
\email{m-mat@math.sci.hiroshima-u.ac.jp}

\thanks{Supported in part by the Scientific Grants-in-Aid 19204002, 23244002,
15K13460, JST-CREST 151001, and by the Core-to-Core grant 18005 from the Japan
Society for the Promotion of Science; and grants DMS-0706955, DMS-1005675 and
DMS-1406420 from the National Science Foundation.}

\date{\today}



\maketitle

\tableofcontents

\section{Introduction}

Among the principal goals of the theory of motives is to construct a $\Q$-linear
tannakian category $\MM(T)$ of mixed motivic sheaves over a smooth base $T$
whose ext groups are the motivic cohomology groups of $T$:
$$
\Ext_{\MM(T)}^\dot(\Q,M) \cong H^\dot_\mot(T,M),
$$
where $M$ is a motive over $T$, such as $\Q(n)$. This goal was partially
achieved by Levine \cite{levine}, Hanamura \cite{hanamura} and Voevodsky \cite{voevodsky}, each of whom constructed a triangulated tensor category of mixed motives with the correct ext groups. Many obstructions remain to constructing tannakian categories of mixed motives including, most notably, the problem of establishing Beilinson--Soul\'e vanishing. One case where this goal has been achieved is that of mixed Tate motives over a ring of $S$ integers in a number field. This was established by Levine \cite{levine:mtm} and Deligne--Goncharov \cite{deligne-goncharov} using the work of Borel \cite{borel}, Beilinson \cite{beilinson:regulators} and the existence of a derived category of mixed motives. The theory of mixed Tate motives can be regarded as the story of motives associated to genus $0$ curves and their moduli spaces, \cite{brown:genus0}.

In this paper we take a step towards extending this story to genus $1$ curves
and their moduli spaces. We construct a $\Q$-linear tannakian category $\MEM_1$
of {\em universal mixed elliptic motives} over $\M_{1,1/\Z}$, the moduli stack
of elliptic curves, which contains $\MTM$, the category of mixed Tate motives
unramified over $\Z$. The ring of functions on its tannakian fundamental group
is a Hopf algebra in the category of ind-objects of $\MTM$ which encodes relations between the periods of iterated integrals of Eisenstein series and periods classical modular forms of level 1.

Each object $\V$ of $\MEM_1$ consists of an object $V$ of $\MTM$ whose Betti
realization is endowed with an action of $\SL_2(\Z)$. This action is required to
determine a compatible set of local systems (Betti, $\Q$-de~Rham, Hodge,
$\ell$-adic) over $\M_{1,1}$ whose fibers over the canonical tangent vector
$\tate:=\partial/\partial q$ \label{def:tate} of $\Mbar_{1,1}$ at the cusp are
the various  realizations of $V$. An object of $\MEM_1$ can be specialized to an
elliptic curve $E/T$ by pulling back its associated local systems along the the
structure map $T \to \M_{1,1}$ to obtain a family of local systems over $T$.

This hybrid approach to constructing a tannakian category of mixed motives over
$\M_{1,1}$ is necessitated by the lack of understanding of motives associated to
modular forms. In particular, it circumvents the problem of proving
Beilinson--Soul\'e vanishing for the category of mixed motives generated by
motives of classical modular forms. We expect the local systems (Betti, de~Rham,
Hodge, $\ell$-adic) associated to an object of $\MEM_1$ to be the set of
compatible realizations of a mixed motivic sheaf over $\M_{1,1/\Q}$ in the sense
of Ayoub \cite{ayoub} or Arapura \cite{arapura}. If this is the case, then their
pullbacks to $T$ should be the set of realizations of a mixed elliptic motive
over $T$ in the sense of Goncharov \cite{goncharov:mem}. We have not pursued
this as our interests lie in the tannakian aspects of the theory, especially in
the determination of the fundamental group of $\MEM_1$ and its relation to
classical modular forms and mixed Tate motives.

The most basic object of $\MEM_1$ is the local system $\H := R^1\pi_\ast\Q$
associated to the universal elliptic curve $\pi: \E_{/\Z} \to \M_{1,1/\Z}$. The
simple objects of $\MEM_1$ are the Tate twists $S^m\H(r)$ of the symmetric
powers of $\H$. The elliptic polylogarithms of Beilinson and Levin
\cite{beilinson-levin} produce basic objects of $\MEM_1$. These provide, for
each $n > 1$, a non-trivial extension
\begin{equation}
\label{eqn:polylog}
0 \to S^{2n-2}\H(2n-1) \to \bE \to \Q(0) \to 0,
\end{equation}
which corresponds to the Eisenstein series of weight $2n$. 

The tannakian fundamental group of $\MEM_1$ is an extension of  $\GL_2$ by a
prounipotent group $\U^\MEM_1$. The main result of this paper
(Theorem~\ref{thm:motivic}) is a partial presentation of the Lie algebra
$\u_1^\MEM$ of $\U^\MEM_1$. General results imply that this has a ``minimal presentation'' whose generators project to a basis of $H_1(\u_1^\MEM)$ and where a minimal set of relations projects to a basis of $H_2(\u_1^\MEM)$. Existing results on modular forms and mixed Tate motives are used to compute $H^1(\u_1^\MEM)$ and to give a lower bound on the size of $H^2(\u_1^\MEM)$. Specifically, the extensions (\ref{eqn:polylog}) comprise a basis of the ``geometric part'' of $H^1(\u^\MEM_1)$. A complete basis is obtained by adding the extensions that correspond to the (motivic) odd zeta values $\zeta^\m(2m+1)$, ($m>0$). Using work of Brown \cite{brown:mmv} and Pollack \cite{pollack}, we construct a (conjecturally complete) set of minimal relations between generators dual to these extensions and determine their leading quadratic terms. These relations are dual to a linearly independent set of elements of $H^2(\u_1^\MEM)$. Each minimal relation corresponds to a Hecke eigenform, and each Hecke eigenform determines a countable set of minimal relations. If one assumes that the ``regulator mapping''
$$
\Ext^2_{\MEM_1}(\Q,S^{2n}\H(r))\otimes \R \to
H^2_\cD(\M_{1,1/\R},S^{2n}\H_\R(r)) 
$$
is injective\footnote{See Conjecture~\ref{conj:h2}(\ref{item:hodge}) for a more
precise statement.}, an analogue of Beilinson's conjecture
\cite[Conj.~8.4.1]{beilinson:hodge_coho}, then these relations generate all
relations in $\u^\MEM_1$. This partial presentation is discussed in more detail
later in the introduction.

One goal of this work is to illuminate the relationship between cusp forms of
$\SL_2(\Z)$ and the depth filtration of the fundamental group of $\MTM$. In this
vein, in Section~\ref{sec:ihara-takao}, we show that the relations in
$\u^\MEM_1$ coming from cusp forms imply the congruences between the generators
of the infinitesimal Galois action on the unipotent fundamental group of
$\Pminus$ in depth 2 that were discovered by Ihara and Takao \cite{ihara} and
made explicit by Goncharov \cite{goncharov:dihedral} and Schneps \cite{schneps}.
Our proof gives a conceptual explanation of how and why cusp forms impose
relations on the depth graded quotients of mixed Tate motives.

\bigskip

We now give a more detailed, but still informal, discussion of universal mixed
elliptic motives. The full definition is given in Section~\ref{sec:mem}. Suppose
that $r$ and $n$ are non-negative integers with $r+n > 0$. Denote the moduli
stack over $\Spec\Z$ of smooth projective curves of genus 1 with $n$ marked points and $r$
non-zero tangent vectors by $\M_{1,n+\vec{r}/\Z}$. So $\M_{1,1}$ is the moduli stack of elliptic curves, and $\M_{1,2}$ is the moduli stack of elliptic curves $(E,0)$ together with an additional point $x\neq 0$. For $\ast\in \{1,\uu,2\}$ we
construct a category $\MEM_\ast$ of universal mixed elliptic motives over
$\M_{1,\ast/\Z}$. Objects of $\MEM_\ast$ will be called {\em universal mixed
elliptic motives of type $\ast$}.

The Tate curve $\E_\Tate\to\Spec\Z[[q]]$ defines a tangential base point
$$
\tate : \Spec\Z((q)) \to \M_{1,1/\Z}.
$$
The normalization of the fiber $\Ebar_0$ over $q=0$ is isomorphic to $\P^1$ and
has a natural coordinate $w$, unique up to the involution $w\leftrightarrow
w^{-1}$, that takes the value $1$ at the identity and defines an isomorphism of
the smooth points $E_0$ of $\Ebar_0$ with $\Gm_{/\Z}$. There is therefore a map
$$
\Spec\Z((t)) \to \Spec\Z((q,v)) \to \M_{1,2/\Z},\quad q \mapsto t,\ v \mapsto t,
$$
where $v=w-1$, which determines a tangential base point of $\M_{1,2/\Z}$ and
the tangential base point
$$
\Spec\Z((t)) \to \Spec\Z((q,v)) \to \M_{1,\uu/\Z}
$$
of $\M_{1,\uu/\Z}$. We shall denote any of these distinguished base points,
including $\tate$, by $\vv_o$.

Informally, a {\em mixed elliptic motive of type $\ast \in \{1,\uu,2\}$ over
$\Z$} is a ``motivic local system'' $\V$ of $\Q$-vector spaces over
$\M_{1,\ast/\Z}$ with a weight filtration $W_\dot$ by sub local systems that
satisfies:
\begin{enumerate}

\item each weight graded quotient of $\V$ is a sum of the simple local systems
$S^n \H(m)$, where $\H = R^1 \pi_\ast \Q(0)$ and $\pi : \E \to \M_{1,\ast}$ is
the universal elliptic curve;

\item the fiber $V_o$ of $\V$ over $\vv_o$ is an object of $\MTM$; its weight
filtration $M_\dot$ is the relative weight filtration associated to the weight
filtration $W_\dot$ and the monodromy operator about the nodal cubic.

\end{enumerate}
The fiber of $\H$ over $\vv_o$ will be denoted by $H$. It is isomorphic to
$\Q(0) \oplus \Q(-1)$ as an object of $\MTM$.

Objects of $\MEM_\ast$ have compatible Betti, Hodge, $\ell$-adic and
$\Q$-de~Rham realizations. The Hodge realization of a mixed elliptic motive $\V$
is an admissible variation of mixed Hodge structure over $\M_{1,\ast}^\an$ whose
fiber over $\vv_o$ is the associated limit mixed Hodge structure. The
$\ell$-adic realization of $\V$ is a lisse sheaf $\V_\ell$ over
$\M_{1,\ast/\Z[1/\ell]}$ whose fiber over $\vv_o$ is the $\ell$-adic realization
of $V_o$. The $\Q$-de~Rham realization of $\V$ is a bifiltered bundle
$(\cV,F^\dot,W_\dot)$ over $\Mbar_{1,\ast/\Q}$ with a regular connection that
underlies Deligne's canonical extension of $\V$ to $\Mbar^\an_{1,1}$ and its
Hodge and weight filtrations. There are natural functors
$$
\MTM\to \MEM_1 \to \MEM_2 \to \MEM_\uu \to \MTM
$$
whose composite is the identity.

Along with the pure motives $S^n\H(m)$, the simplest objects of $\MEM_\ast$
include the ``geometrically constant'' motives. These are the objects of $\MTM$
that are pulled back along the structure map $\M_{1,\ast/\Z} \to \Spec\Z$. Their
Hodge realizations are constant variations of MHS over $\M_{1,\ast}^\an$. More
interesting objects of $\MEM_\ast$ can be constructed from the unipotent
fundamental group of punctured elliptic curves. Denote by $\bp$ the local system
over $\M_{1,2}$ whose fiber over $(E;0,x)$ is the Lie algebra $\p(E',x)$ of the
unipotent completion of $\pi_1(E',x)$, where $E'$ denotes $E-\{0\}$. The local
system $\bp$ is a pro-object of $\MEM_2$. This is an elliptic analogue of the
result of Deligne and Goncharov \cite{deligne-goncharov} that the Lie algebra of
the unipotent fundamental group of $\Pminus$ is a pro-object of
$\MTM$.\label{def:E'}

The category $\MEM_\ast$, being a $\Q$-linear neutral tannakian category, is the
category of representations of an affine group scheme, unique up to inner
automorphism, that we shall denote by $\pi_1(\MEM_\ast)$. One goal of this paper
is to determine the structure of $\pi_1(\MEM_\ast)$. While we do not find an
explicit presentation of its Lie algebra, we are able to give a good idea of its
``shape''. The first step in this direction is taken in
Section~\ref{sec:structure} where we show that $\pi_1(\MEM_\ast)$ is an
extension of $\GL(H)$ by a prounipotent group and prove that there are natural
isomorphisms
$$
\pi_1(\MEM_\uu) \cong \pi_1(\MEM_1)\ltimes \Q(1)
\text{ and }
\pi_1(\MEM_2) \cong \pi_1(\MEM_1)\ltimes \pi^\un_1(E_\tate',\vv_o),
$$
where $E_\tate$ is the fiber of the universal elliptic curve over $\tate$. This
reduces the problem of finding presentations of $\pi_1(\MEM_\ast)$ to the case
$\ast = 1$ and to the problem of determining how it acts on the Lie algebra of
$\pi^\un_1(E_\tate',\vv_o)$.

When trying to understand $\pi_1(\MEM_\ast)$, it is convenient to define the
{geometric fundamental group} $\pi_1^\geom(\MEM_\ast)$ of $\MEM_\ast$ to be the
kernel of the natural surjection $\pi_1(\MEM_\ast) \to \pi_1(\MTM)$. Taking the
fiber over the tangent vector $\vv_o$ corresponding to the Tate curve gives a
splitting $s_\tate : \pi_1(\MTM) \to \pi_1(\MEM_\ast)$ of this homomorphism, so
that one has an isomorphism
$$
\pi_1(\MEM_\ast) \cong \pi_1(\MTM)\ltimes \pi_1^\geom(\MEM_\ast).
$$

A key technical point (established in Section~\ref{sec:splitting_DR}), which
simplifies the problem of finding presentations of the Lie algebras of the
$\pi_1(\MEM_\ast)$,  is that the de~Rham realization of every object of
$\MEM_\ast$ has a canonical (i.e., unique natural) bigrading that splits the
Hodge filtration and both weight filtrations. This generalizes the canonical
grading of the de~Rham realization of a mixed Tate motive that splits it weight
and Hodge filtrations. This bigrading allows us to canonically identify the
de~Rham realization $V^\DR$ of an object $\V$ of $\MEM_\ast$ with its associated
bigraded module
$$
\Gr V^\DR = \bigoplus_{m,n} \Gr^M_{2m}\Gr^W_n V^\DR
$$
and reduces the problem of finding a presentation of the Lie algebra of
$\pi_1(\MEM_\ast)$ to the problem of finding a presentation of its associated
bigraded Lie algebra.

For this reason, we now specify the fiber functor to be the de~Rham fiber
functor. This allows us to identify the Lie algebra $\g_\ast^\MEM$ of
$\pi_1(\MEM_\ast)$ with the completion of its associated bigraded Lie algebra.
Denote its pronilpotent radical by $\u_\ast^\MEM$. The canonical splitting of
the weight filtration $W_\dot$ gives a canonical splitting of the extension
$$
0 \to \u^\MEM_\ast \to \g^\MEM_\ast \to \gl(H) \to 0.
$$
So we can regard $\u^\MEM_\ast$ as a Lie algebra in the category of
$\gl(H)$-modules and canonically identify $\g^\MEM_\ast$ with $\gl(H)\ltimes
\u^\MEM_\ast$.

The computation of the simple extensions in Part~2 is used in
Section~\ref{sec:gens} to find a set of generators $\ee_{2n}$ of the Lie algebra
$\u_1^\geom$ of the prounipotent radical of $\pi_1^\geom(\MEM_\ast)$. The
generator $\ee_{2n}$ is dual to the extension (\ref{eqn:polylog}) associated to
the Eisenstein series of weight $2n$. It is a highest weight vector in a copy of
$S^{2n-2}H(2n-1)$ in $\u^\geom_1$.

A fundamental result \cite{deligne-goncharov} of Deligne and Goncharov asserts
that $\pi_1(\MTM)$ is an extension of $\Gm$ by a prounipotent group whose Lie
algebra is freely generated (non canonically) by a set $\{\ssigma_{2m+1}:m\ge
1\}$. In Part~\ref{part:presentation}, we show that $\u^\MEM_1$ has a bigraded
presentation with generators
$$
\zz_{2m+1} \in \Gr^W_{-4m-2}\Gr^M_{-4m-2}\u^\MEM_1 \text{ and }
\ee_{2n} \in \Gr^W_{-2n}\Gr^M_{-2}\u^\MEM_1, \quad m > 0,\ n > 1,
$$
as a Lie algebra in the category of $\sl(H)$-modules, where each $\zz_{2m+1}$
spans a copy of the trivial representation of $\sl(H)$ and projects to the
corresponding generator $\ssigma_{2m+1}$ of the Lie algebra of the prounipotent
radical of $\pi_1(\MTM)$.  The Lie algebra $\u^\geom_1$ is generated
(topologically) by the set
$$
\{\ee_0^j\cdot\ee_{2n}:n>1,\ 0 \le j \le 2n-2\},
$$
where $\ee_0\in \sl(H)$ lowers $\sl(H)$-weights by $2$ and where $u\cdot v$
denotes the adjoint action of $u$ on $v$.

There are two kinds of relations in $\u^\MEM_1$: namely,
\begin{description}

\item [geometric relations] the relations between the geometric generators
$\{\ee_{2n}:n\ge 0\}$,

\item [arithmetic relations] which describe how the arithmetic generators
$\zz_{2m+1}$ act on the geometric generators $\ee_{2n}$.

\end{description}
In Section~\ref{sec:pollack_motivic}, we compute the quadratic leading terms of
a minimal generating set of relations of $\u^\MEM_1$. There is one such
geometric relation for each normalized Hecke eigen cusp form and each integer
$d\ge 2$, the degree of the relation as an expression in the $\ee_{2n}$, $n\ge
0$. Each Eisenstein series $G_{2m}$ determines a countable set of arithmetic
relations which express $[\zz_{2m-1},\ee_{2n}]$ in terms of the geometric
generators $\{\ee_k: k\ge 0\}$. We conjecture that regulator
$$
\Ext^1_\MM\big(\Q,H^1_\cusp(\M_{1,1},S^{2n}\H(2n+d))\big)\otimes_\Q\R
\to H^1_\cD\big(\M_{1,1}^\an,S^{2n}\H_\R(2n+d)\big)^{\Frbar_\infty}
$$
should be an isomorphism. (Cf.\ Conjecture~\ref{conj:h2}(\ref{item:hodge}).)
This is an analogue of Beilinson's Conjecture \cite{beilinson:hodge_coho}. If
true, it would imply that the relations we have found generate all relations in
$\u^\MEM_\ast$.

In Part~\ref{part:presentation}, we approach the problem of finding the
geometric relations by studying the monodromy representation $\u^\geom_1 \to
\Der \p$, where $\p$ denotes the Lie algebra of $\pi_1^\un(E_\tate,\vv_o)$.
Relations between the images of the $\ee_{2n}$ in $\Der \p$ give an upper bound
on the relations that hold between the $\ee_{2n}$ in $\u^\geom_1$. The problem
of determining these relations can be studied by passing to the associated
bigraded Lie algebras. Since $\p$ is free, it is isomorphic to the completion of
its associated bigraded, which is the free Lie algebra $\L(H)\cong\L(A,T)$. The
$\Q$-de~Rham story gives a formula for the images of the $\ee_{2n}$ in
$\Der\L(A,T)$ via the elliptic KZB equation \cite{kzb,levin-racinet,hain:kzb,rome}.
The image of the generator $\ee_{2n}$ of $\u^\geom_\uu$ is $2\e_{2n}/(2n-2)!$,
where $\e_{2n}$ is the unique derivation\footnote{These derivations were first
studied by Tsunogai \cite{tsunogai} and later used in this context by Calaque,
Enriquez and Etingof \cite{kzb}.} of $\L(H)$ satisfying
$$
\e_{2n}(T) = - \ad_T^{2n}(A) \text{ and } \e_{2n}([T,A])=0,\qquad n>1.
$$

Pollack, in his undergraduate thesis \cite{pollack}, found all quadratic
relations that hold between the $\{\e_{2n}:n\ge 0\}$ in $\Der \L(A,T)$ and
showed that they correspond to cuspidal cocycles of $\SL_2(\Z)$ invariant under
the ``real Frobenius" operator. He also found relations of each degree $d\ge 3$
that hold between the $\e_{2n}$ modulo ``depth 3'', one for each invariant
cuspidal cocycle of $\SL_2(\Z)$. His results are summarized in
Section~\ref{sec:pollack}.

A lower bound on the relations in $\u^\MEM_1$ is obtained from the computation
of the cup products of the classes of the extensions (\ref{eqn:polylog}) in the
real Deligne cohomology groups $H^\dot_\cD(\M_{1,1}^\an,S^{2n}\H(r))$. These are
deduced from the fundamental work of Francis Brown \cite{brown:mmv} who computed
periods of twice iterated integrals of Eisenstein series. His work is closely
related to unpublished computations of Tomohide Terasoma \cite{terasoma} who
computed cup products in Deligne cohomology of the classes in higher Chow groups
of Kuga--Sato varieties that correspond to Eisenstein series, generalizing to
higher weight Beilinson's computations \cite{beilinson:modular} in weight
2.\footnote{There is an inconsistency between the results of Brown and Terasoma.
Our computations are consistent with Brown's. Nonetheless, we believe that
Terasoma's results are basically correct.} It turns out that the upper and lower
bounds on the quadratic heads of the relations in $\u^\MEM_1$ coincide, as we
show in Section~\ref{sec:pollack_motivic}, which allows us to determine the
quadratic leading terms of all relations in $\u^\MEM_1$.

In related work, Enriquez \cite{enriquez} has defined an elliptic generalization
of Drinfeld's braided monoidal categories \cite{drinfeld} and defined an affine
$\Q$-group $\GRT_\elliptic$ for which the scheme of elliptic associators is a
torsor.\footnote{Implicit in \cite{enriquez} is that this group depends on the
choice of a free Lie algebra of rank 2 over $\Q$ together with an ordered pair of
generators. Two natural choices are the Betti and de~Rham realizations of the
Lie algebra of $\pi_1^\un(E_\tate',\vv_o)$ with a pair of generators which
descends to a symplectic basis of its abelianization. Such a choice will
correspond to the choice of a fiber functor.} It is a split extension
$$
1 \to R_\elliptic \to \GRT_\elliptic \to \GRT \to 1,
$$
where $\GRT$ is the affine $\Q$-group, defined by Drinfeld \cite{drinfeld}, and
where the projection to $\GRT$ is induced by the map (\ref{eqn:inclusion}). One
expects there to be a homomorphism $\pi_1(\MEM_\uu) \to \GRT_\elliptic$ with an
appropriate choice of fiber functors,  which projects the extension above to
$$
1 \to \pi_1^\geom(\MEM_\uu) \to \pi_1(\MEM_\uu) \to \pi_1(\MTM) \to 1.
$$
Assuming that it exists, this homomorphism will be surjective if and only if
Enriquez's conjecture \cite[Conj.~9.1]{enriquez} holds and the  standard
homomorphism $\pi_1(\MTM) \to \GRT$ is surjective. It will be injective if and
only if $\u^\MEM \to \Der\p$ is injective.\footnote{Brown's result
\cite{brown:mtm} implies that $\pi_1(\MTM) \to \GRT$ is injective.} The
existence of this homomorphism will imply that the relations in $\u^\MEM_\uu$
will also hold in the Lie algebra $\grt_\elliptic$ of $\GRT_\elliptic$.

\bigskip

This paper is in four parts. In the first part we define universal mixed
elliptic motives. Before doing this, we give a detailed description of the local
system $\H$ over $\M_{1,1}$ and its fiber $H$ over $\tate$ in all of its
manifestations. We also consider the basic structure of $\pi_1(\MEM_\ast)$ that
one gets from formal considerations.

The second part is devoted to proving that the elliptic polylogarithmic
extensions (\ref{eqn:polylog}) lie in $\MEM_\ast$ and that they exhaust all
extensions of $\Q$ by a simple object $S^m\H(r)$ of $\MEM_1$ when $m>0$. For
this we use the Hodge theory of the relative completion of $\SL_2(\Z)$, which is
described in detail in \cite{hain:modular}. Results from \cite{hain:db_coho}
allow us to relate extensions of admissible variations of mixed Hodge structure
over $\M_{1,\ast}^\an$ to the Deligne--Beilinson cohomology groups
$H_\cD^\dot(\M_{1,\ast}^\an,S^m\H(r))$. We also use weighted and crystalline
completion to study the $\ell$-adic aspects. This part concludes with a
discussion of the relationship between the groups
$\Ext^2_{\MEM_\ast}(\Q,S^m\H(r))$ and certain of the conjectures of Beilinson
and Bloch--Kato on regulators associated to modular curves.

In Part 3 we consider the problem of finding a presentation of the Lie algebra
$\u^\MEM_\ast$. We first show that the $\ee_{2n}$ and the $\zz_{2m+1}$ generate
$\u^\MEM_1$. We then recall Pollack's relations and use Brown's period
computations to show that they lift to relations in $\u^\MEM_\ast$.

In Part 4 we use the results of Part 3 to begin the study of the relationship
between universal mixed elliptic motives and mixed Tate motives. This is
achieved by specialization to the nodal cubic. As an application, we deduce the
Ihara--Takao congruences from Pollack's relations. We also work out the precise
relationship between the  depth filtration on $\pi_1(\MTM)$ and the pull back
of the ``elliptic weight filtration'' $W_\dot$ of $\Der \p$.

\subsection{History}

This paper evolved slowly over the past 8 years and owes much to its
antecedents. These include the paper \cite{beilinson-levin} of Beilinson and
Levin on the elliptic polylogarithm, the work of Deligne \cite{deligne:p1} on
the fundamental group of $\Pminus$, and the paper of Deligne and Goncharov
\cite{deligne-goncharov} on mixed Tate motives. The works of
Calaque--Enriquez--Etingof \cite{kzb} and Levin--Racinet \cite{levin-racinet} on
the elliptic KZB connection proved to be useful, both in understanding the
$\Q$-de~Rham picture and in establishing formulas that play a key role in
Part~\ref{part:genus0}. Finally, Nakamura's computation \cite{nakamura} of the
action of the absolute Galois group $G_\Q$ on the fundamental group of
$E'_\tate$ was also a useful reference point and helped us understand the
$\ell$-adic picture.

This work originated in informal computations done in the Spring of 2007 in
which we found (under optimistic assumptions) the basic form (described  in
\S\ref{sec:relns}) of a presentation of the $\ell$-adic crystalline completion
of $\pi_1(\M_{1,1/\Z[\ell^{-1}]},\tate)$. Standard conjectures on the ranks of
Selmer groups and an optimistic assumption on the size of the image of cup
products in Galois cohomology---which now appears likely to be true---led us to
predict that each cusp form of $\SL_2(\Z)$ should determine a sequence of
relations in the unipotent radical of the crystalline completion, and thus
between the derivations $\e_{2n}$ discussed above. The first author gave the
problem of finding some of these relations to Aaron Pollack, who was then an
undergraduate. He found all the quadratic relations between the $\e_{2n}$ and
also relations of each degree $\ge 3$ that held modulo depth $\ge 3$. This
suggested that our optimistic assumptions were indeed true. However, the proof
that Pollack's relations lifted to $\u^\MEM_1$ had to wait until the period
computations of Brown \cite{brown:mmv}.\ There is no $\ell$-adic proof of this.
It would  be very interesting to have one.
\bigskip

\noindent{\em Acknowledgments:} The authors gratefully acknowledge the support
of MSRI in 2008 where some of the initial work was done. The first author also
gratefully acknowledges the generous support of the Friends of the Institute for Advanced Study and Duke University which allowed him to spend his sabbatical year 2014--15 at the IAS where part of this paper was written. The first author would like to thank Francis Brown for numerous stimulating discussion during which the material in Parts 3 and 4 evolved. The authors would also thank Tomohide Terasoma for his correspondence related to the computation of the cup product in Deligne cohomology of Kuga--Sato varieties. They would also like to thank Francis Brown and the referee for their detailed and constructive comments on the manuscript, and Aaron Pollack for his contributions to this project during the early stages of this work.

\section{Notation and Conventions}

All moduli spaces in this paper will be regarded as stacks and their underlying
analytic spaces will be regarded as orbifolds (i.e., stacks in the category of
complex analytic varieties). The dual $\Hom_F(V,F)$ of a vector space $V$ over a field $F$ will be denoted by $V^\vee$.

\subsection{Path multiplication} We use the topologists' convention. The product
$\alpha\beta$ of two paths $[0,1]\to X$ in a topological space is defined when
$\alpha(1)=\beta(0)$. The product $\alpha\beta$ traverses $\alpha$ first, then
$\beta$. This is the opposite of the algebraists' convention, where paths are
multiplied in the ``functional order''.

\subsection{The topological fundamental group} Suppose that $X$ is an algebraic
variety  over a subring $R$ of $\C$. Denote the corresponding analytic variety
by $X^\an$. For $x\in X(\C)$, we  denote the topological fundamental group of
$X^\an$ by $\pi_1^\top(X,x)$.

If $X$ is a DM stack over $R$, the associated analytic space $X^\an$ is an
orbifold (i.e., a stack in the category of topological spaces). For each
suitable base point $x$ of $X^\an$, we denote the orbifold fundamental group of
$X^\an$ by $\pi_1^\top(X,x)$. Fundamental groups of stacks are defined in
\cite{noohi}, for example. Typically in this paper, $X$ will be the
stack/orbifold quotient of a smooth variety by a finite group. In this case the
orbifold fundamental group is easy to describe directly. See,
\cite[Sec.~3]{hain:elliptic}, for example.

Denote the orbifold quotient of a topological space $\X$ by a discrete group
$\G$ by $\G\bbs\X$. If $\X$ is simply connected, then for each choice of a base
point $x_o\in \X$, there is a natural isomorphism
$$
\pi_1^\top(\G\bbs \X,x_o) \cong \G.
$$
In addition, we can regard the projection $p : \X \to \G\bbs \X$ as a base
point. In this case $\G$ is naturally isomorphic to $\pi_1^\top(\G\bbs \X,p)$.
For example
$$
\pi_1^\top(\SL_2(\Z) \bbs \h,p) \cong \SL_2(\Z).
$$

\subsection{The \'etale fundamental group}

Denote the \'etale fundamental group of a scheme $X$ over a ring $R$ by
$\pi_1(X,\xbar)$, where $\xbar$ is a geometric point of $X$. The \'etale
fundamental group of a DM stack $X$ can also be defined --- see \cite{noohi}. We
will also denote it by $\pi_1(X,\xbar)$. Suppose that $R$ is a field $k$ with
algebraic closure $\kbar$. The fundamental group of $\Spec k$ with respect to
the geometric point $\Spec \kbar \to \Spec k$ is simply $\Gal(\kbar/k)$. We
denote it by $G_k$. The structure morphism $X \to \Spec k$ induces a
homomorphism $\pi_1(X,\xbar) \to G_k$. One has the canonical short exact
sequence
\begin{equation}
\label{eqn:ses}
1 \to \pi_1(X\times_k \kbar,\xbar) \to \pi_1(X,\xbar) \to G_k \to 1.
\end{equation}

\subsection{Topological versus \'etale fundamental groups}
\label{sec:opposite}

When $X$ is a DM-stack over $\C$ and $x\in X(\C)$, there is a natural
isomorphism\footnote{Strictly speaking, our path multiplication convention
implies that this is an anti-isomorphism. If one prefers, one can replace the
\'etale fundamental group by its opposite group.}
$$
\pi_1(X,x) \cong \pi_1^\top(X,x)^\wedge
$$
between the \'etale fundamental group of $X$ and the profinite completion of the
topological fundamental group of the associated analytic variety. When $k$ is a
subfield of $\C$, the exact sequence (\ref{eqn:ses}) becomes
$$
1 \to \pi_1^\top(X,x)^\wedge \to \pi_1(X,x) \to G_k \to 1.
$$

\subsection{Tannakian fundamental groups}

Suppose that $F$ is a field of characteristic zero. We will denote the tannakian
fundamental group of an $F$-linear neutral tannakian category $\cC$ with respect
to a fiber functor $\w : \cC \to \Vec_F$ by $\pi_1(\cC,\w)$. It is an affine
group scheme that is an inverse limit of affine algebraic groups. When the
context is clear, we will omit $\w$ from the notation.

We will denote the category of linear representations of a group $G$ over a
field $F$ by $\Rep_F(G)$ and the sub-category of finite dimensional
representations by $\Rep^\fte_F(G)$. If $G$ is an affine $F$-group (scheme), we
will write $\Rep(G)$ instead of $\Rep_F(G)$. If $F$ has characteristic zero,
then $\Rep(G)$ is equivalent to the category of ind-objects of $\Rep^\fte(G)$.

Subsequently, the term ``algebraic group'' will mean ``affine algebraic group''.
In particular, we will not refer to elliptic curves as algebraic groups.

\subsection{Cohomology of affine group schemes}

The cohomology of an affine group scheme $G$ over a field $F$ of characteristic
zero with coefficients in a $G$-module $V$ is defined by
$$
H^\dot(G,V) := \Ext^\dot_{\Rep(G)}(F,V)
$$ 
where $\Rep(G)$ denotes the category of $G$-modules. The fact that every
$G$-module is a direct limit of its finite dimensional submodules implies that
if $V$ is finite dimensional over $F$, then
$$
H^\dot(G,V) = \Ext^\dot_{\Rep^\fte(G)}(F,V)
$$
where $\Rep^\fte(G)$ denotes the category of finite dimensional $G$-modules. If
$G$ is an extension of a reductive group $R$ by a prounipotent group $U$, and
if $U$ has Lie algebra $\u$, then there are natural isomorphisms
$$
H^\dot(G,V) \cong H^\dot(U,V)^R \cong H^\dot(\u,V)^R.
$$
For background and more details, see \cite[Chapts.~3,4]{jantzen} and
\cite{hain:db_coho}.

\subsection{Free Lie algebras}

The free Lie algebra generated by a vector space $V$ will be denoted by $\L(V)$.
\label{def:LV} The free Lie algebra generated by a set $\X$ will be denoted by
$\L(\X)$. It is isomorphic to the free Lie algebra generated by the vector space
with basis $\X$.

\part{Introduction to Universal Mixed Elliptic Motives}

\section{Mixed Tate Motives}

Deligne and Goncharov \cite{deligne-goncharov}, using work of Voevodsky
\cite{voevodsky} and Levine \cite{levine,levine:mtm}, constructed a
$\Q$-tannakian category of {\em mixed Tate motives} $\MTM(\O_{K,S})$ over a
number field $K$, unramified outside a subset $S$ of primes of $\O_K$. This
category has simple objects $\Q(n)$ and has the property that
$$
\Ext_{\MTM(\O_{K,S})}^j\big(\Q,\Q(n)\big) \cong 
\begin{cases}
\Q & j = n = 0, \cr
K_{2n-1}(\O_{K,S})\otimes \Q & j = 1,\ n > 0, \cr
0 & \text{ otherwise.}
\end{cases}
$$
In this paper, we need only the case where $\O_{K,S}=\Z$, so for simplicity we restrict discussion to that case. Denote $\MTM(\Z)$ by $\MTM$.\label{def:MTM} Each object $V$ of $\MTM$ has a weight filtration, which we denote by $M_\dot$ and a Hodge filtration $F^\dot$. There are realization functors on $\MTM$ that take the object $V$ to its Betti realization $(V^B,M_\dot)$, its de~Rham realization $(V^\DR,M_\dot,F^\dot)$, and for each prime number $\ell$, its $\ell$-adic realization $(V_\ell,M_\dot)$, which is a filtered $G_\Q := \Gal(\Qbar/\Q)$ module, unramified outside $\ell$ and crystalline at $\ell$. There are comparison isomorphisms
$$
((V^\DR,M_\dot)\otimes_\Q \C \cong V^B,M_\dot)\otimes_\Q \C \text{ and }
(V^B,M_\dot)\otimes_\Q \Ql \cong (V_\ell,M_\dot).
$$
The Betti and de~Rham realizations combine to give the Hodge realization of $V$, which is a mixed Hodge structure whose weight graded quotients are of Tate type. One important property of $\MTM$ is that the functor that takes $\MTM$ to the category $\MHS$ of $\Q$-MHS is fully faithful.

Every simple object of $\MTM$ is isomorphic to $\Q(n)$ for some $n\in \Z$. For later use, we recall same basic facts. The weight filtration of $V=\Q(n)$ is
$$
0 = M_{-2n-1}V \subset M_{-2n}V = V
$$
and that the realizations of $\Q(n)$ are
$$
V^B = \Q,\ V^\DR = \Q,\ V_\ell = \Ql.
$$
The Galois action on $V_\ell$ is given by the $n$th power of the $\ell$-adic
cyclotomic character $\chi_\ell : G_K \to \Zl^\times$. The Hodge filtration of
$V^\DR$ is
$$
V^\DR = F^{-n} V^\DR \supset F^{-n+1}V^\DR = 0.
$$
The comparison isomorphism $V^\DR\otimes\C \to V^B\otimes \C$ is multiplication
by $(2\pi i)^{-n}$. The Hodge realization is the one dimensional Hodge structure of type $(-n,-n)$.

The category $\MTM^\ss$ of semi-simple mixed Tate motives (i.e., those that are
direct sum of copies of $\Q(m)$'s) is neutral tannakian and has fundamental
group isomorphic to $\Gm$; the mixed Tate motive $\Q(m)$ corresponds to the
$m$th power of the standard character of $\Gm$. The functor $\Gr^M_\dot : \MTM
\to \MTM^\ss$ is a retraction of $\MTM$ onto $\MTM^\ss$.

Let
$$
\w^B : \MTM \to \Vec_\Q,\ \w^\DR : \MTM \to \Vec_\Q,\
\w_\ell : \MTM \to \Vec_\Ql
$$
be the fiber functors that take $V$ to $V^B$, $V^\DR$ and $V_\ell$, respectively. For each of these, the inclusion $\MTM^\ss \to \MTM$ induces a
homomorphism $\pi_1(\MTM,\w) \to \Gm$. Denote its kernel by $\K$ (or $\K_\w$
when we want to identify the fiber functor).\label{def:K} It is prounipotent.
Since for every mixed Tate motive $V$, there is a natural isomorphism
$$
V^\DR \cong \bigoplus_n F^nV^\DR \cap M_{2n} V^\DR
$$
of $\Q$-vector spaces, the functor $\Gr^M_\dot$ induces a splitting of the
homomorphism $\pi_1(\MTM,\w^\DR) \to \pi_1(\MTM^\ss,\w^\DR) = \Gm$. The Lie
algebra $\k$ of $\K$ is thus a $\Gm$-module. It is (non-canonically) isomorphic
to the completion of  the free Lie algebra
$$
\L(\ssigma_3,\ssigma_5,\ssigma_7,\ssigma_9, \dots)
$$
where $\Gm$ acts on $\ssigma_{2n+1}$ via the $(2n+1)$st power of the standard
character.

\section{Moduli Spaces of Elliptic Curves}
\label{sec:moduli}

Suppose that $r$ and $n$ are non-negative integers satisfying $r+n>0$. Denote
the moduli stack over $\Spec \Z$ of smooth projective curves of genus 1 with $n$ marked
points and $r$ non-zero tangent vectors by $\M_{1,n+\vec{r}}$. Here the $n$
marked points and the anchor points of the $r$ tangent vectors are distinct.
Taking each tangent vector to its anchor point defines a morphism
$\M_{1,n+\vec{r}} \to \M_{1,n+r}$ that is a principal ${\mathbb G}_m^r$-bundle.
The Deligne-Mumford compactification \cite{knudsen} of $\M_{1,n}$ will be
denoted by $\Mbar_{1,n}$. It is also defined over $\Spec\Z$. In this paper, we
are primarily concerned with the cases $(n,r) = (1,0), (0,1)$ and $(2,0)$. Then $\M_{1,1}$ is the moduli stack of elliptic curves $(E,0)$, and $\M_{1,2}$ is the moduli stack of elliptic curves $(E,0)$ with an additional point $x\neq 0$.

\subsection{The moduli stacks $\M_{1,1}$, $\M_{1,\uu}$ and $\M_{1,2}$}

For a $\Z$-algebra $A$, we denote by $\M_{1,n+\vec{r}/A}$ the stack
$\M_{1,n+\vec{r}}\times_{\Spec\Z}\Spec A$. Similarly, the pullback of
$\Mbar_{1,n}$ to $\Spec A$ will be denoted by $\Mbar_{1,n/A}$.
The universal elliptic curve over $\M_{1,1}$ will be denoted by $\E$. Note that
$\M_{1,2}$ is $\E$ with its identity section removed. The extension
$\cEbar\to\Mbar_{1,1}$ of the universal elliptic curve to $\Mbar_{1,1}$ is
obtained by glueing in the Tate curve $\E_\Tate \to \Spec\Z[[q]]$ (cf.\
\cite[Ch.~V]{silverman:2}). It is simply $\Mbar_{1,2}$.

The unique cusp of $\Mbar_{1,1}$ is the moduli point of the nodal cubic. Denote
it by $e_o$.\label{def:cusp} The standard line bundle $\cL$ over $\Mbar_{1,1}$
is the conormal bundle of the zero section of $\cEbar \to \Mbar_{1,1}$. Sections
of $\cL^{\otimes n}$ over $\Mbar_{1,1}$ are modular forms of weight $n$, and
those that vanish at the cusp $e_o$ are cusp forms of weight $n$.

The restriction of the relative tangent bundle of $\cEbar$ to its identity 
section is the dual $\check{\cL}$ of $\cL$.  The moduli space $\M_{1,\uu}$ is
$\check{\cL}'$, the restriction of $\check\cL$ to $\M_{1,1}$ with its
zero-section removed. This is isomorphic to the complement $\cL'$ of the
zero-section of $\cL$.

We will also identify $\Mbar_{1,1}$ with the identity section of $\cEbar$. With
this convention, $e_o$ also denotes the identity of the nodal cubic in $\cEbar$
and also the corresponding point of the partial compactification $\check\cL$ of
$\M_{1,\uu}$.

An explicit description of $\cL'$ over $\Mbar_{1,1}$ can be deduced from the
discussion \cite[Chapt.~2]{katz-mazur} and the formulas in
\cite[Appendix~A]{silverman:1}; namely $\Mbar_{1,1}$ is the quotient stack
$\Gm\bbs\cL'$. When $2$ and $3$ are invertible in $A$, $\cL_A'$ is the scheme
$$
\cL_A' = \bA^2_A - \{0\} = \Spec A[u,v] -\{0\}.
$$
The point $(u,v)$ corresponds to the plane cubic $y^2 = 4x^3-ux-v$ and the
abelian differential $dx/y$. The $\Gm$-action is $\lambda : (u,v) \mapsto
(\lambda^{-4} u,\lambda^{-6} v)$. In this case, $\Mbar_{1,1}$ and $\M_{1,1}$ are
the quotient stacks
$$
\Mbar_{1,1/A} = \Gm\bbs \cL'_A \text{ and }
\M_{1,1/A} = \Gm\bbs(\bA^2_A - D^{-1}(0)),
$$
where $D=u^3 - 27v^2$ is (up to a factor of 4) the discriminant of the cubic.

Since some 2-pointed genus 1 curves have a non-trivial automorphism, $\Mbar_{1,2}$ is a stack, but not a scheme. When $2$ and $3$ are invertible in $A$, the stack $\M_{1,2/A}$ is the quotient of the scheme
$$
\{(u,v,x,y) \in \bA_A^2\times \bA_A^2: y^2 = 4x^3 - ux -v,\ (u,v) \neq 0\}
$$
by the $\Gm$-action
$$
\lambda : (u,v,x,y) \mapsto
(\lambda^{-4}u,\lambda^{-6}v,\lambda^{-2}x,\lambda^{-3}y).
$$
The point $(u,v,x,y)$ corresponds to the point $(x,y)$ on the cubic
$y^2=4x^3-ux-v$ and the abelian differential $dx/y$. The action of $\lambda$
multiplies $dx/y$ by $\lambda$.

\subsection{Tangent vectors}
\label{sec:tangent}

We use Deligne's tangential base points \cite[\S15]{deligne:p1}. The Tate curve $E_{\mathrm{Tate}} \to \Spec \Z[[q]]$ corresponds to a morphism 
$\Spec\Z[[q]] \to \Mbar_{1,1}$. The parameter $q$ is a formal parameter of
$\Mbar_{1,1}$ at the cusp $e_o$.  The tangent vector
$$
\partial/\partial q = \Spec \Qbar((q^{1/n}:n\ge 1))
$$
of $\Mbar_{1,1}$ at the moduli point of the nodal cubic is integral and its
reduction mod $p$ is non-zero for all prime numbers $p$. Denote it by $\tate$.
Identify $\Mbar_{1,1}$ with the identity section of the completion $\Mbar_{1,2}$
of the universal elliptic curve over $\Mbar_{1,1}$. With this identification,
$e_o$ is the identity of the nodal cubic.

The fiber $\Ebar_0$ of the Tate curve over $q=0$ is an integral model of the
nodal cubic. Its normalization is $\P^1_\Z$. Let $E_0$ be $\Ebar_0$ with its
double point removed. It is isomorphic to $\Gm_{/\Z}$.  Let $w$ be a parameter
on $\Ebar_0$ whose pullback to the normalization of $\Ebar_0$ takes the values
$0$ and $\infty$ on the inverse image of the double point, and the value 1 at
the identity. It is unique up to $w\mapsto w^{-1}$. It determines the tangent
vector $\ww_o := \partial/\partial w$ of $\Ebar_0$ at the identity $e_o$ of
$E_0$. It is integrally defined and non-zero at all primes.

We thus have the tangent vector $\vv_o := \tate+\ww_o$ of $\M_{1,\uu}$ (and thus
of $\M_{1,2}$ and $\E$ as well) at the identity $e_o$ of $E_0$ which is non-zero
at all primes.

\subsection{Moduli spaces as complex orbifolds}

To fix notation and conventions, we give a quick review of the construction of
$\M_{1,1}^\an$, $\E^\an$ and their Deligne-Mumford compactifications. All will
be regarded as complex analytic orbifolds. This material is classical and very
well-known. A detailed discussion of these constructions and an explanation of
the notation can be found in \cite{hain:elliptic}.

\subsubsection{The orbifolds $\M_{1,1}^\an$ and $\Mbar_{1,1}^\an$}
\label{sec:orbifolds}

The moduli space $\M_{1,1}^\an$ is the orbifold quotient $\M_{1,1}^\an =
\SL_2(\Z)\bbs\h$ of the upper half plane $\h$ by the standard $\SL_2(\Z)$
action. For $\tau\in\h$, set $\Lambda_\tau := \Z \oplus \Z\tau$. The point
$\tau\in \h$ corresponds to the elliptic curve
$$
(E_\tau,0) := (\C/\Lambda_\tau,0),
$$
together with the symplectic basis $\aa,\bb$ of $H_1(E_\tau,\Z)$ that
corresponds to the generators $1,\tau$ of $\Lambda_\tau$ via the canonical
isomorphism $\Lambda_\tau \cong H_1(E_\tau,\Z)$. The element
$$
\gamma = \begin{pmatrix} a & b \cr c & d \end{pmatrix}
$$
of $\SL_2(\Z)$ takes the basis $\aa,\bb$ of $H_1(E_\tau,\Z)$ to the basis
\begin{equation}
\begin{pmatrix} a & b \cr c & d \end{pmatrix}
\begin{pmatrix}\bb \cr \aa\end{pmatrix}
\end{equation}
of $H_1(E_{\gamma\tau},\Z)$.

The orbifold $\Mbar_{1,1}^\an$ underlying the Deligne-Mumford compactification
$\Mbar_{1,1}$ of $\M_{1,1}$ is obtained by glueing in the quotient $C_2\bbs\D$
of a disk $\D$ of radius $e^{-2\pi}$ by a trivial action of the cyclic group
$C_2:=\{\pm 1\}$. These are glued together by the diagram
$$
\xymatrix{
C_2 \bbs \D & \ar[l]_(0.75)\approx
{\begin{pmatrix} \pm 1 & \Z \cr 0 & \pm 1 \end{pmatrix}}
\bbs
\{\tau\in \h:\Im(\tau)>1\} \ar[r] & \SL_2(\Z)\bbs \h
}
$$
where the left-hand map takes $\tau$ to $q:=\exp(2\pi i\tau)$ and where $C_2$ is included in {\scriptsize$\begin{pmatrix} \pm 1 & \Z \cr 0 & \pm \end{pmatrix}$} as the scalar matrices.

\subsubsection{The line bundle $\cL^\an$}

The restriction of $\cL^\an$ to $\M_{1,1}^\an$ is the orbifold quotient of the
trivial line bundle $\C\times\h \to \h$ by the $\SL_2(\Z)$ action 
$$
\begin{pmatrix}a & b\cr c & d\end{pmatrix} : (z,\tau) \mapsto
\big((c\tau+d)z,(a\tau+b)/(c\tau+d)\big)
$$
It is an orbifold line bundle over $\M_{1,1}^\an$. Its restriction (i.e., pullback) to the punctured $q$-disk
$$
\D^\ast := \begin{pmatrix} 1 & \Z \cr 0 & 1 \end{pmatrix} \bs \h
$$
is naturally isomorphic to the trivial bundle
$\C\times\D^\ast \to \D^\ast$, and therefore extends naturally to a (trivial)
line bundle over $\D$. The line bundle $\cL^\an$ over $\Mbar_{1,1}$ is the
unique extension of the line bundle above to $\Mbar_{1,1}^\an$ that restricts to
this trivial bundle over the $q$-disk.

\subsubsection{Eisenstein series}

To fix notation and normalizations we recall some basic facts from the
analytic theory of elliptic curves.

Suppose that $k\ge 1$. The (normalized) Eisenstein series of weight $2k$
is defined by the series
$$
G_{2k}(\tau) = \frac{1}{2}\frac{(2k-1)!}{(2\pi i)^{2k}}
\sum_{\substack{\lambda \in \Z\oplus \Z\tau\cr\lambda \neq 0}}
\frac{1}{\lambda^{2k}}
= -\frac{B_{2k}}{4k} + \sum_{n=1}^\infty \sigma_{2k-1}(n)q^n,
$$
where $\sigma_k(n)=\sum_{d|n}d^k$. When $k>1$ it converges absolutely to a
modular form of weight $2k$. When properly summed, it also converges when $k=1$.
In this case the logarithmic 1-form
\begin{equation}
\label{eqn:G2}
\frac{d\xi}{\xi} - 2\cdot 2\pi i\, G_2(\tau)\, d\tau
\end{equation}
is $\SL_2(\Z)$-invariant, where $\SL_2(\Z)$ acts on $\C^\ast\times \h$ by
$\gamma : (\xi,\tau) \mapsto \big(\xi/(c\tau+d),\gamma\tau\big)$.

The ring of modular forms of $\SL_2(\Z)$ is the polynomial ring $\C[G_4,G_6]$.

\subsection{The Weierstrass $\wp$-function}

For a lattice $\Lambda\subset\C$, the Weierstrass $\wp$-function is defined by
$$
\wp_\Lambda(z) := \frac{1}{z^2} +
\sum_{\substack{\lambda \in \Lambda\cr \lambda \neq 0}}
\bigg[\frac{1}{(z-\lambda)^2} - \frac{1}{\lambda^2} \bigg].
$$
For $\tau \in \h$, set $\wp_\tau(z) = \wp_{\Lambda_\tau}(z)$. One has the
expansion
$$
\wp_\tau(z) = (2\pi i)^2\bigg(\frac{1}{(2\pi i z)^2} +
\sum_{m=1}^\infty \frac{2}{(2m)!}G_{2m+2}(\tau) (2\pi i z)^{2m}\bigg).
$$
The function $\C \to \P^2(\C)$ defined by
$
z \mapsto [(2\pi i)^{-2}\wp_\tau(z),(2\pi i)^{-3}\wp_\tau'(z),1]
$
induces an embedding of $\C/\Lambda_\tau$ into $\P^2(\C)$. The image has affine
equation
$$
y^2 = 4x^3 - g_2(\tau)x - g_3(\tau),
$$
where
$$
g_2(\tau) = 20 G_4(\tau) \text{ and } g_3(\tau) = \frac{7}{3} G_6(\tau).
$$
This has discriminant the normalized cusp form of weight 12:
$$
\Delta(\tau) := g_2(\tau)^3 - 27 g_3(\tau)^2 = q \prod_{n\ge 1}(1-q^n)^{24}.
$$

The following statement is easily verified.  For a proof of the last statement,
see \cite[Prop.~19.1]{hain:kzb}.

\begin{proposition}
\label{prop:differentials}
The abelian differential $dx/y$ corresponds to  $d(2\pi i dz)$ and the differential
$xdx/y$ of the second kind corresponds to $(2\pi i)^{-2}\wp_\tau(z)d(2\pi i z)$.
These differentials form a symplectic basis of $H^1_\DR(E_\tau)$ as
$$
\int_{E_\tau} \frac{dx}{y}\smile \frac{xdx}{y} = 2 \pi i.
$$
\end{proposition}

\subsection{The analytic space $\M_{1,\uu}^\an$}

Recall that $D(u,v) = u^3 - 27v^2$. The map
$$
\C^\ast \times \h \to \C^2 - D^{-1}(0),\quad (\xi,\tau) \mapsto
\big(\xi^{-4}g_2(\tau),\xi^{-6}g_3(\tau)\big)
$$
induces a biholomorphism ${\cL^\an}' \to \M_{1,\uu}^\an = \C^2 - D^{-1}(0)$.
The point $(\xi,\tau)$ of $\C^\ast\times \h$ corresponds to the point
$$
\big(\C/\xi\Lambda_\tau,2\pi i dz\big) \cong
\big(\C/\Lambda_\tau,2\pi i\xi\,dz\big)
$$
of $\M_{1,\uu}^\an$, which also corresponds to the curve $y^2 = 4x^3 -
g_2(\tau)x - g_3(\tau)$ with the abelian differential $\xi dx/y$.

\subsubsection{The universal elliptic curve $\E^\an$}

Define $\G$ to be the subgroup of $\GL_3(\Z)$ that consists of the matrices
$$
\gamma = \begin{pmatrix} a & b & 0 \cr c & d & 0 \cr m & n & 1\end{pmatrix}
$$
where $a,b,c,d,m,n\in \Z$ and $ad-bc = 1$. It is isomorphic to the semi-direct
product of $\SL_2(\Z)$ and $\Z^2$ and acts on $X := \C\times \h$ on the left via
the formula $\gamma : (z,\tau)\mapsto (z',\tau')$, where 
$$
\begin{pmatrix}
\tau' \cr 1 \cr z' 
\end{pmatrix}
=
(c\tau+d)^{-1}
\begin{pmatrix}
a & b & 0 \cr c & d & 0 \cr m & n & 1
\end{pmatrix}
\begin{pmatrix}
\tau \cr 1 \cr z 
\end{pmatrix}.
$$
The universal elliptic curve $\pi^\an : \E^\an \to \M_{1,1}^\an$ is the orbifold
quotient of the projection $\C\times \h \to \h$ by $\G$, which acts on $\h$ via
the quotient map $\G\to \SL_2(\Z)$. The fiber of $\pi^\an$ over the orbit of
$\tau$ is $E_\tau$.

It can also be regarded as the orbifold quotient $\SL_2(\Z)\bbs \E_\h$ of the
universal framed elliptic curve
$$
\E_\h := \Z^2 \bs \ (\C\times \h),
$$
where $(m,n) : (z,\tau) \mapsto (z + m\tau + n,\tau)$, by the natural action
$$
\gamma : (z,\tau) \mapsto \big((c\tau+d)^{-1}z,\gamma \tau)
$$
of $\SL_2(\Z)$.

The universal elliptic curve $\cEbar^\an \to \Mbar_{1,1}^\an$ is obtained by
glueing in the Tate curve as described in \cite[\S5]{hain:elliptic}. The
restriction to the $q$-disk $\D$ of $\cEbar^\an$ minus the double point of
$\Ebar_0$ is the quotient of $\C^\ast\times\D$ by the $\Z$-action
$$
n : (w,q) \mapsto
\begin{cases}
(q^n w, q) & q \neq 0 \cr
(w,0) & q = 0.
\end{cases}
$$

To relate this to the algebraic construction, note that $\M_{1,2+\uu}^\an$ is
the analytic variety $\G \bs \big(\C^\ast\times \C \times \h)$, where
$\gamma(\xi,z,\tau) = \big((c\tau+d)\xi,\gamma(z,\tau)\big)$. The function
$$
(\xi,z,\tau) \mapsto \big(\xi^{-4}g_2(\tau),\xi^{-6}g_3(\tau),
[(2\pi i)^{-2}\wp_\tau(z),(2\pi i)^{-3}\wp_\tau'(z),1]\big)
$$
from $\C^\ast\times \C \times \h$ to $\C^2-D^{-1}(0)\times \P^2$ induces a
biholomorphism
$$
\G\bs \big(\C^\ast\times \C \times \h) \to \M_{1,2+\uu}^\an.
$$
It is invariant with respect to the-$\C^\ast$ action $\lambda\cdot(\xi,z,\tau) =
(\lambda\xi,z,\tau)$ on $\G\bs(\C^\ast\times \C \times \h)$  and the
$\C^\ast$-action on $\M_{1,2+\uu}$ that multiplies the abelian differential by
$\lambda$.

\subsubsection{Orbifold fundamental groups}

Since $\M_{1,1}^\an$ is the orbifold quotient of $\h$ by $\SL_2(\Z)$, there
is a natural isomorphism
$$
\pi_1^\top(\M_{1,1},p) \cong \Aut(\h \to \SL_2(\Z)\bbs\h)
\cong \SL_2(\Z),
$$
where $p$ denotes the projection $\h \to \M_{1,1}^\an$. The inclusion
of the imaginary axis in $\h$ induces a canonical isomorphism
\begin{equation}
\label{eqn:canon_isom}
\pi_1^\top(\M_{1,1},\tate) \cong \pi_1^\top(\M_{1,1},p) \cong \SL_2(\Z).
\end{equation}

The orbifold fundamental group of $\M_{1,\uu}^\an$ is a central extension of
$\SL_2(\Z)$ by $\Z$. There are natural isomorphisms
$$
\pi_1^\top(\M_{1,\uu},\vv_o) \cong B_3 \cong \SLtilde_2(\Z),
$$
where $B_3$ denotes the braid group on 3 strings and $\SLtilde_2(\Z)$ denotes
the inverse image of $\SL_2(\Z)$ in the universal covering group
$\SLtilde_2(\R)$ of $\SL_2(\R)$. Again, this is well-known; details can be found
in \cite[\S8]{hain:elliptic}.

Similarly, there are natural isomorphisms
$$
\pi_1^\top(\E,\vv_o) \cong \pi_1^\top(\E,p')
\cong \Aut(\C\times\h\to \E^\an) \cong \G \cong \SL_2(\Z)\ltimes \Z^2,
$$
where $p' : \C\times\h \to \E^\an$ is the projection.

\section{The Local System $\H$}
\label{sec:localsys_H}

Roughly speaking, the local system $\H$ over $\M_{1,1}$ is the ``motivic local
system'' $R^1\pi_\ast\Q$ associated to the universal elliptic curve $\pi : \E\to
\M_{1,1}$. While we could work in Voevodsky's category of motivic sheaves
\cite{voevodsky}, we will instead define $\H$ to be a set of compatible
realizations: Betti, $\Q$-de Rham, Hodge, and $\ell$-adic \'etale. These are
described below in detail.

In subsequent sections, we will abuse notation and denote the pull back of $\H$
to $\M_{1,n+\vec{r}}$ by $\H$ for all $r,n\ge 0$ with $r+n>0$.

\subsection{Betti realization}

The Betti realization of $\H$ is the orbifold local system
$$
\H^B := R^1\pi_\ast^\an\Q
$$
over $\M_{1,1}^\an$ associated to the universal elliptic curve $\pi^\an : \E^\an
\to \M_{1,1}^\an$. Since Poincar\'e duality induces an isomorphism $H^1(E) \cong
H_1(E)$ for all elliptic curves $E$, $\H^B$ is also isomorphic to the local
system whose fiber over $[E]\in \M_{1,1}$ is $H_1(E,\Q)$.

Since $\h$ is contractible, there is a natural isomorphism
$$
H_1(\E_\h,\Z) \cong H_1(E_\tau,\Z) \cong \Z\aa \oplus \Z\bb
$$
for each $\tau\in \h$. The sections $\aa$ and $\bb$ thus trivialize the pullback
of $\H^B$ to $\h$. The induced action of $\SL_2(\Z)$ on $H_1(\E_\h,\Z)$ is given
by left multiplication:
$$
\gamma : 
\begin{pmatrix} \bb \cr \aa \end{pmatrix} \mapsto
\gamma \begin{pmatrix} \bb \cr \aa \end{pmatrix}.
$$
It corresponds to a right action of $\SL_2(\Z)$ on $\Q\aa\oplus \Q\bb$.

Denote the dual basis of $H^1(\E_\h,\Z)$ by $\adual,\bdual$. The action of
$\SL_2(\Z)$ on this frame is given by $\gamma : \begin{pmatrix} \bdual & -\adual
\end{pmatrix} \mapsto \begin{pmatrix} \bdual & -\adual \end{pmatrix}\gamma$. The
basis $\adual,\bdual$ corresponds to the basis $-\bb,\aa$ of $H_1(E_\tau)$ under
Poncar\'e duality. The corresponding action of $\SL_2(\Z)$ on this frame is:
\begin{equation}
\label{eqn:action}
\gamma
: \begin{pmatrix} \aa & -\bb \end{pmatrix} \mapsto
\begin{pmatrix} \aa & -\bb \end{pmatrix}\gamma.
\end{equation}
This defines a left action of $\SL_2(\Z)$ on $H^B\times \h$, where $H^B = \Q\aa
\oplus \Q\bb$ should be thought of as the fiber of $\H^B$ over $p$, and also
over $\tate$.

\subsection{The flat vector bundle $\cH^\an$ and its canonical extension}
\label{sec:connection}
Denote the flat connection on the holomorphic vector bundle
$$
\cH^\an = \H^B\otimes \O_{\M_{1,1}^\an}
$$
by $\nabla_0$. The pullback $\cH^\an_\h$ of $\cH^\an$ to $\h$ is the vector
bundle $\O_{\M_{1,1}^\an}\aa \oplus \O_{\M_{1,1}^\an}\bb$. The sections $\aa$
and $\bb$ are flat.

Define a holomorphic section $\bw$ of $\H^\an_\h$ by
$$
\bw(\tau) = 2\pi i\, \w_\tau \in H^1(E_\tau,\C),
$$
where $\w_\tau$ denotes the class of the holomorphic 1-form $dz$ in
$H^1(E_\tau)$.

This bundle has a Hodge filtration
$$
\cH^\an = F^0 \cH^\an  \supset F^1 \cH^\an \supset F^2 \cH^\an = 0.
$$
The section $\bw$ trivializes $F^1\cH^\an_\h$.
 
The sections $\aa$ and $\bw$ descend to give a framing
$$
\cH^\an_{\D^\ast} = \O_{\D^\ast}\aa \oplus \O_{\D^\ast}\bw
$$
of 
the pullback of $\cH^\an$ to the punctured $q$-disk. Since $\log q = 2\pi i
\tau$,
$$
\nabla_0 \bw = \nabla_0 (-2\pi i \bb + \log q\, \aa) = \aa \frac{dq}{q}.
$$
Since $\nabla_0 \aa = 0$, the connection on $\cH_\h$ is given by
\begin{equation}
\label{eqn:connection}
\nabla_0 = d + \aa\frac{\partial}{\partial \bw} \frac{dq}{q}.
\end{equation}

To define an extension $\Hbar^\an$ of $\cH^\an$ to a vector bundle over
$\Mbar_{1,1}^\an$, it suffices to extend $\cH^\an_{\D^\ast}$ to a vector
bundle over $\D$. We do this by defining
$$
\Hbar_\D^\an = \O_{\D}\aa \oplus \O_{\D}\bw.
$$
The formula (\ref{eqn:connection}) implies that the connection $\nabla_0$
extends to a meromorphic connection on $\Hbar^\an$ with a regular singular point
at the cusp $e_o$. The residue of the connection at the cusp is the nilpotent
operator $\aa\partial/\partial \bw$. This implies:

\begin{proposition}
The flat vector bundle $(\Hbar^\an,\nabla_0)$ is Deligne's canonical extension
of $\cH^\an$ to $\Mbar_{1,1}^\an$. The holomorphic sub-bundle $F^1\cH^\an$
extends to the holomorphic sub-bundle $F^1 \Hbar^\an$ that is spanned (locally)
by $\bw$.
\end{proposition}

\subsection{Algebraic de~Rham realization}
\label{sec:DR}

A vector bundle on $\M_{1,1/\Q}$ with a connection is a vector bundle on
$\bA^2_\Q - \{0\}$ with a $\Gm$-action, endowed with a $\Gm$-invariant
connection that is trivial on each $\Gm$-orbit. Define
$$
\Hbar^\DR = \O_{\bA^2_\Q} S \oplus \O_{\bA^2_\Q}T.
$$
Extend the action $\lambda : (u,v) \mapsto (\lambda^{-4}u,\lambda^{-6}v)$ of
$\Gm$ on $\bA^2$ to this bundle by defining
$$
\lambda \cdot S = \lambda^{-1} S \text{ and } \lambda \cdot T = \lambda T.
$$
Define a connection $\nabla_0$ on $\Hbar^\DR$ by
$$
\nabla_0 = d +
\Big(
-\frac{1}{12}\frac{dD}{D}\otimes T
+ \frac{3}{2}\frac{\alpha}{D}\otimes S
\Big)\frac{\partial}{\partial T}
+
\Big(
-\frac{u}{8}\frac{\alpha}{D}\otimes T
+ \frac{1}{12}\frac{dD}{D}\otimes S
\Big)\frac{\partial}{\partial S},
$$
where $\alpha = 2udv-3vdu$ and $D = u^3 - 27 v^2$.

The connection is $\Gm$-invariant, trivial on each orbit, defined over $\Q$, has regular singularities
along the discriminant divisor $D = u^3 - 27v^2 = 0$ and is holomorphic on its
complement. It therefore descends to a rational connection over
$\Mbar_{1,1/\Q}$, with a regular singular point at the cusp $e_o$.

Define a Hodge filtration
$$
\Hbar^\DR = F^0\Hbar^\DR \supset F^1\Hbar^\DR \supset F^2\Hbar^\DR = 0
$$
on $\Hbar^\DR$ by setting $F^1\Hbar^\DR = \O_{\bA^2_\Q} T$. This is a
$\Gm$-invariant sub-bundle, and thus defined over $\M_{1,1/\Q}$.

The following statement follows from \cite[Prop.~19.6]{hain:kzb},
\cite[Prop.~19.7]{hain:kzb}, and Proposition~\ref{prop:differentials}
above.\footnote{The normalizations in \cite{hain:kzb} differ from those here.
They are related by $T = 2\pi i\hat{T}$, $S=\hat{S}/2\pi i$, and $A=2\pi i
\aa$.}

\begin{proposition}
\label{prop:nabla_0}
There is a natural isomorphism
$$
(\Hbar^\an,\nabla_0) \cong
(\Hbar^\DR,\nabla_0)\otimes_{\O_{\Mbar_{1,1/\Q}}}\O_{\Mbar_{1,1}^\an}
$$
that respects the Hodge filtration. When pulled back to
$\M_{1,\uu/\Q}=\bA^2-\{0\}$, the section $T$ corresponds to $dx/y$ and the
section $S$ to $xdx/y$. After pulling back to the $q$-disk along the slice
$q\mapsto \big(g_2(q),g_3(q)\big)\in \bA^2$, we have $T=\bw$ and $S=\aa - 2
G_2(q) \bw$.
\end{proposition}

\subsection{Hodge realization}

The Hodge realization of $\H$ is the polarized variation of $\Q$-Hodge structure
(PVHS) of weight 1 over $\M_{1,1}^\an$ whose underlying local system is the
Betti incarnation $\H^B$ of $\H$ and whose associated flat vector bundle, with
its Hodge filtration, is the flat vector bundle
$(\Hbar^\DR,\nabla_0,F^\dot)\otimes\O_{\M_{1,1}^\an}$.  It is naturally
polarized by the cup product. Since it is of weight 1, its weight filtration is
$$
0 = W_0\H \subseteq W_1\H = \H.
$$

We compute the limit mixed Hodge structure on the fiber $H := H_\tate$ of $\H$
over $\tate=\partial/\partial q$.\label{def:H}

The fiber $H^\DR$ of $\Hbar^\DR$ over the cusp $e_o$ has $\Q$-basis $S$ and $T$.
Proposition~\ref{prop:nabla_0} implies that $T=\bw$ and that, when $q=0$,
$$
S=\aa - G_2(0)\bw/2 = \aa + \bw/12.
$$
Consequently,
$$
H^\DR = \Q\aa \oplus \Q\bw.
$$
The residue of the connection at $e_o$ is the nilpotent operator
$$
N := \aa\frac{\partial}{\partial \bw} \in \End H^\DR.
$$
The associated monodromy weight filtration of $H^\DR$ (centered at the weight,
1, of $\H$) is
$$
0 = M_{-1} H^\DR \subset M_0 H^\DR = M_1 H^\DR \subset M_2 H^\DR = H^\DR.
$$

It remains to determine the Betti structure $H^B$ on $H$ associated to the
tangent vector $\tate =\partial/\partial q$. This is computed according to
Schmid's prescription \cite{schmid}. The complex vector space underlying $H$ is
$H^\DR\otimes\C$. Its integral lattice is
$$
H^B_\Z := H^0(\h,\H^B) = \Z\aa \oplus \Z\bb.
$$
The comparison isomorphism $H^B \to H^\DR\otimes \C$ associated to $\tate$ takes $\bv \in H^B$ to
$$
\lim_{q\to 0}q^N\bv
= \lim_{q\to 0}\big(\id + \log q\, \aa\partial/\partial \bw\big)\bv.
$$
Since $\bb = \big((\log q)\aa - \bw\big)/(2\pi i)$, the comparison isomorphism takes
$\aa \in H^B$ to $\aa \in H^\DR$ and $\bb \in H^B$ to
$$
\lim_{q\to 0}\big(\id + \log q\, \aa\partial/\partial \bw\big)\bb
= -(2\pi i)^{-1}\bw.
$$
Since $\bw$ spans $F^1 H^\DR$, it follows that the limit MHS is isomorphic
to $\Z(0) \oplus \Z(-1)$. The copy of $\Z(0)$ is spanned by $\aa$ and
the copy of $\Z(-1)$ is spanned by $\bb = -(2\pi i)^{-1}\bw$.

The image of the positive generator of $\pi_1^\top(\D^\ast)$ in
$\pi_1^\top(\M_{1,1},\tate) = \SL_2(\Z)$
is
$$
\sigma_o := \begin{pmatrix} 1 & 1 \cr 0 & 1\end{pmatrix}.
$$
Formula (\ref{eqn:action}) implies that in $\End H$
$$
\log \sigma_o = \aa\frac{\partial}{\partial \bb}
= -{2\pi i}\,\aa\frac{\partial}{\partial \bw}.
$$

\begin{remark}
For the uninitiated, it may seem strange that the pure Hodge structure on
$H^1(E_\tau)$ becomes a mixed Hodge structure in the limit. One way to come to
terms with this is to think of the limit MHS $H=H_\tate$ as being the mixed
Hodge structure on $H^1(E_\tate)$, where $E_\tate$ denotes a smoothing of the
nodal cubic $E_0$ in the direction of $\tate$.\footnote{The notion of the fiber
$E_\vv$ of $\E$ over a non-zero tangent vector $\vv$ of the origin of the
$q$-disk can be made precise. See, for example, see Appendix C of
\cite{hain:kzb}.} There is a continuous retraction $E_\tate \to E_0$ whose
composition with the inclusion $E_0^\ast \hookrightarrow E_\tate$ is the
inclusion $E_0^\ast \hookrightarrow E_0$,  where $E_0^\ast$ denotes the
non-singular locus of $E_0$, which is isomorphic to $\C^\ast$. This sequence
induces an exact sequence of MHS
$$
0 \to H^1(E_0) \to H^1(E_\tate) \to H^1(E_0^\ast) \to 0,
$$
which exhibits $H^1(E_\tate)$ as an extension of $H^1(\C^\ast) =\Z(-1)$ by
$H^1(E_0) = \Z(0)$. This MHS is not always split --- if one replaces $\tate$ by
$\lambda\tate$, then the MHS on $H^1(E_{\lambda\tate})$ is the extension of
$\Z(-1)$ by $\Z(0)$ corresponding to $\lambda \in \C^\ast \cong
\Ext^1(\Z(-1),\Z(0))$. To prove this, one replaces $q$ by $q/\lambda$ in the
above computations of the limit.
\end{remark}

\subsection{\'Etale realization}
\label{sec:etale}

The $\ell$-adic realization of $\H$ is the lisse sheaf
$$
\H_\ell := R^1\pi_\ast\Ql
$$
over the stack $\M_{1,1/\Q}$. Its fiber over the moduli point of an elliptic
curve $(E,0)$ over a number field $K$ in $\Qbar$ is $\pi_1(E,0)\otimes\Ql(-1)$,
endowed with the natural action of $G_K := \Gal(\Qbar/K)$. As in the Hodge case,
its fiber $ H_\ell = H^1_\et(E_{\tate/\Qbar},\Ql) $ over $\tate = \Spec
\Qbar((q^{1/n}:n\ge 1))$ is a split extension of $\Ql(-1)$ by $\Ql(0)$. Details
can be found in Nakamura \cite{nakamura}.

The local system $\H_\ell$ is determined by its monodromy representation
$$
\rho_{H,\ell} : \pi_1(\M_{1,1/\Q},\tate) \to \Aut H_\ell.
$$
Its restriction
$$
\pi_1(\M_{1,1/\Qbar},\tate) \to \Aut H_\ell = \Aut(\Ql(0)\oplus \Ql(-1))
$$
to the geometric fundamental group is $G_\Q$-equivariant and corresponds to the
action of $\pi_1^\top(\M_{1,1},\tate)$ on $H^B\otimes \Ql$ under the comparison
isomorphisms
$$
H_\ell \cong H^B\otimes\Ql \text{ and }
\pi_1(\M_{1,1/\Qbar},\tate)^\op \cong \pi_1^\top(\M_{1,1},\tate)^\wedge,
$$
where $(\blank)^\wedge$ denotes profinite completion and $\op$ denotes opposite
group. (Cf.\ Section \ref{sec:opposite}.)

\subsection{Summary}
\label{sec:summary}
The key point of the previous discussion is that the fibers over $\tate$ of the local system $\H^B$, $\H_\ell$, $\Hbar^\DR$ are naturally isomorphic to the realizations of the object $H=\Q(0)\oplus \Q(-1)$ of $\MTM$. In this sense, the fibers over $\tate$ can be lifted canonically to an object $H$ of $\MTM$. The local systems $\H^B$, $\Hbar^\DR$ and $\H_\ell$ are determined by the mixed Tate motive $H$ and the action of $\SL_2(\Z)$ on its Betti realization. More precisely:
\begin{enumerate}

\item There is an object $H = \Q(0)\oplus \Q(-1)$ of $\MTM$, endowed with
two weight filtrations: $M_\dot$, its weight filtration in $\MTM$, and the
second weight filtration
$$
0 = W_0 H \subset W_1 H = H.
$$
Its Betti realization $H^B = \Q\aa \oplus \Q\bb$ is the fiber of $\H^B$ over $\tate$; its de~Rham realization is the fiber $H^\DR = \Q\aa \oplus \Q\bw$ of $\Hbar^\DR$ over the cusp. The comparison isomorphism takes $\aa$ to $\aa$ and $\bb$ to $-(2\pi i)^{-1}\bw$.

\item There is an action of $\pi_1^\top(\M_{1,1},\tate)\cong \SL_2(\Z)$ on $H^B$, where
$\gamma \in \SL_2(\Z)$ acts via the formula (\ref{eqn:action}). This determines
the $\Q$-local system $\H^B$ over $\M_{1,1}^\an$.

\item There is a filtered bundle with connection $(\Hbar^\DR, F^\dot,\nabla_0)$
over $\Mbar_{1,1/\Q}$ whose complexification is isomorphic to the canonical
extension of $\cH := \H^B\otimes \O_{\M_{1,1}^\an}$ to $\Mbar_{1,1}^\an$ and
where $F^\dot$ is the usual Hodge filtration. Together these give $\H^B$ the
structure of a polarized variation of Hodge structure over $\M_{1,1}^\an$.

\item The fiber of $\Hbar^\DR$ over the cusp $e_o$ of $\Mbar_{1,1/\Q}$ is
naturally isomorphic to $H^\DR$. The filtration $M_\dot$ is the monodromy weight
filtration of the residue $N = \aa\partial/\partial \bw \in \End H^\DR$ of
$\nabla$ at $e_o$. The limit MHS on the fiber of the polarized variation $\H^B$
over $\tate$ is the Hodge realization of $H$.

\item For each prime number $\ell$, the action $\pi_1(\M_{1,1/\Qbar},\tate) \to
\Aut H_\ell$ induced by the action of $\SL_2(\Z)$ on $H^B$ via the comparison
isomorphism is $G_\Q$-equivariant.

\end{enumerate}

This is the basic universal elliptic motive. Other universal elliptic motives
will include $S^n \H(r) := (S^n\H)\otimes\Q(r)$.

\section{Universal Mixed Elliptic Motives}
\label{sec:mem}

Suppose that $\ast \in \{1,\uu,2\}$. Denote the natural tangential base point of
$\M_{1,\ast/\Q}$ constructed in Paragraph~\ref{sec:tangent} by $\vv_o$. It
induces a section $\vv_{o\ast} : G_\Q \to \pi_1(\M_{1,\ast/\Q},\vv_o)$ of the
natural homomorphism $\pi_1(\M_{1,\ast/\Q},\vv_o) \to G_\Q$ and therefore
an action of $G_\Q$ on $\pi_1(\M_{1,\ast/\Qbar},\vv_o)$.

Let $\H$ be the structure described in the summary in Section~\ref{sec:summary}, pulled
back to $\M_{1,\ast}$.

\begin{definition}[mixed elliptic motives]
 A {\em universal mixed elliptic motive} $\V$ over $\Z$ of type $\ast$ 
consists of:
\begin{enumerate}

\item an object $V$ of $\MTM$ (which is called the {\em fiber of $\V$ over
$\vv_o$}) whose weight filtration is denoted $M_\dot$;

\item an increasing filtration $W_\dot$ of $V$ in $\MTM$ that
satisfies
$$
V = \bigcup_m W_m V \text{ and } \bigcap_m W_m V = 0;
$$

\item a bifiltered vector bundle $(\cV,W_\dot,F^\dot)$ over
$\Mbar_{1,\ast/\Q}$ whose fiber over $e_o$ is the $\Q$-de Rham realization of
$(V,W_\dot)$;

\item an integrable flat connection
$$
\nabla : \cV \to \cV \otimes \Omega^1_{\Mbar_{1,\ast/\Q}}(\log \Delta)
$$
defined over $\Q$ with nilpotent residue along each component of the boundary
divisor $\Delta$ which preserves $W_\dot$ and satisfies Griffiths
transversality:
$$
\nabla :
F^p\cV \to F^{p-1}\cV \otimes \Omega^1_{\Mbar_{1,\ast/\Q}}(\log \Delta);
$$

\item a homomorphism $\rho_V : \pi_1^\top(\M_{1,\ast}^\an,\vv_o) \to
\Aut(V^B,W_\dot)$.

\end{enumerate}
Denote by $(\V^B,W_\dot)$ the filtered $\Q$-local system over
$\M_{1,\ast}^\an$ whose fiber over $\vv_o$ is $(V^B,W_\dot)$ and whose monodromy
representation is $\rho_V$;

\begin{enumerate}
\setcounter{enumi}{5}

\item an isomorphism $(\cV,W_\dot,\nabla)\otimes_\O \O_{\Mbar_{1,\ast}^\an}$
with the canonical extension of the filtered flat bundle $(\V^B,W_\dot)\otimes
\O_{\M_{1,\ast}^\an}$ to $\Mbar_{1,\ast}^\an$ that induces the comparison
isomorphism of $(V^\DR,W_\dot)\otimes\C$ with $(V^B,W_\dot)\otimes \C$ on the fiber over $\vv_o$.

\end{enumerate}
These are required to satisfy:
\begin{enumerate}

\item[(a)] for each prime number $\ell$, the homomorphism
$$
\rho_{V,\ell} : \pi_1(\M_{1,\ast/\Qbar},\vv_o) \to \Aut V_\ell
$$
induced by $\rho_V$ via the comparison isomorphism $V_\ell \cong V^B\otimes \Ql$
is $G_\Q$-equivariant;

\item[(b)] each weight graded quotient $\Gr^W_m \V$ of $\V$ is isomorphic to a direct
sum of copies (with multiplicities) of $S^{m+2r}\H(r)$.

\end{enumerate}
\end{definition}

\begin{remark}
The last condition implies that $M_\dot$ is the relative weight filtration
associated to the nilpotent endomorphism $\log \sigma_o$ (the positive generator
of the fundamental group of the punctured $q$-disk) of the filtered vector space
$(V^B,W_\dot)$. The isomorphism of the canonical extension of $\V^B\otimes
\O_{\M_{1,\ast}^\an}$ with the complexification of  $(\V,W_\dot,F^\dot)$ defines
an admissible variation of MHS over $\M_{1,\ast}^\an$ whose limit MHS over
$\vv_o$ is naturally isomorphic to the Hodge realization of $(V,M_\dot,W_\dot)$.

In addition, since $V$ is an object of $\MTM$, the associated Galois
representation $G_\Q \to \Aut(V\otimes\Ql)$ is unramified at all primes $p\neq
\ell$ and is crystalline at $\ell$.
\end{remark}

\begin{remark}
Lemma~4.3 of \cite{hain:db_coho} implies that the filtration $W_\dot$ of a
universal mixed elliptic motive $V$ is uniquely determined by the filtration
$M_\dot$ and the $\pi_1^\top(\M_{1,\ast},\vv_o)$ action on $V^B$.
\end{remark}

\begin{definition}
A morphism $\phi : \V \to \bU$ of universal mixed elliptic motives consists of a
morphism $\phi^\MTM : V \to U$ in $\MTM$ and a morphism $\phi^\DR :
(\cV,F^\dot,\nabla) \to (\U,F^\dot,\nabla)$ of their de~Rham realizations. These
are required to be compatible with all additional structures in the sense that
\begin{enumerate}

\item the morphisms $V^\DR \to U^\DR$ induced by $\phi^\MTM$ and $\phi^\DR$ are
equal;

\item for all prime numbers $\ell$, the diagram
$$
\xymatrix@R=.7pc@C=.8pc{
\pi_1^\top(\M_{1,\ast},\vv_o)\ar[rr]^(0.55){\rho_U}
\ar[dr]_{\rho_V}\ar[dd]
&& \Aut U^B \ar[dd] \cr
 & \Aut V^B\ar[ur]\ar[dd]\cr
\pi_1(\M_{1,\ast/\Q},\vv_o) \ar'[r]^(.8){\rho_{U,\l}}[rr]
\ar[dr]_{\rho_{V,\l}} && \Aut U_{\l}\cr
 & \Aut V_{\l}\ar[ur]\cr
G_\Q \ar[dr]\ar[uu]^{\vv_o} \ar'[r][rr]  && \pi_1(\MTM,\w_\ell)(\Ql) \ar[uu] \cr
& \pi_1(\MTM,\w_\ell)(\Ql) \ar@{=}[ur]\ar[uu]
}
$$
commutes, where the homomorphisms $\Aut V^B \to \Aut V_\ell$ and $\Aut U^B \to
\Aut U_\ell$ are induced by the comparison maps;

\item The induced map $\V^B \to \bU^B$ of local systems induces a morphism of
variations of MHS over $\M_{1,\ast}^\an$.

\end{enumerate}
\end{definition}

Denote the category of mixed elliptic motives of type $\ast \in \{1,\uu,2\}$ by
$\MEM_\ast$.\label{def:MEM} There is a functor $\vv_o^\ast : \MEM_\ast \to
\MTM$ that takes a mixed elliptic motive $\V$ to its fiber $(V,M_\dot)$ over the
base point $\vv_o$.

\begin{example}[Geometrically constant mixed elliptic motives]
Suppose that $\ast \in \{1,\uu,2\}$. Objects $\V$ of $\MEM_\ast$ for which the
representation $\rho_V$ is trivial will be called {\em geometrically constant}.
These have the property that the two weight filtrations $W_\dot$ and $M_\dot$
coincide on their fiber over $\vv_o$ and are characterized by this property when
$\ast \neq \uu$. (Cf.\ Proposition~\ref{prop:m=w}.) One can think of the
geometrically constant objects of $\MEM_\ast$ as pullbacks of objects of $\MTM$
along the structure morphism $\M_{1,\ast} \to \Spec\Z$.
\end{example}

\begin{example}[Simple universal mixed elliptic motives]
These are the Tate twists $S^m\H(r)$ of symmetric powers of $\H$. That they are
universal elliptic motives follows from the discussion in the previous 
section that was summarized in Section~\ref{sec:summary}. The
mixed Tate motive underlying $S^m\H(r)$ is
$$
S^m H(r) = \Q(r) \oplus \Q(r-1) \oplus \dots \oplus \Q(r-m).
$$
\end{example}

\begin{definition}
An object $V$ of $\MEM_\ast$ is {\em $W$-pure of weight $r$} if $\Gr^W_j V = 0$ when $j\neq r$. It is {\em $M$-pure of weight $m$} if $\Gr^M_j V = 0$ when $j\neq m$.
\end{definition}

The simple object $S^m \H(r)$ of $\MEM_\ast$ is $W$-pure of weight $m-2r$. It is $M$-pure if and only if $m=0$. Not all
objects of $\MEM_\ast$ are $W$-pure or geometrically constant. The simplest
non-trivial examples are extensions of $\Q$ by $S^{2n}\H(2n+1)$ over $\M_{1,2}$
for each $n\ge 1$. These are the elliptic polylogarithms of Beilinson and Levin
\cite{beilinson-levin}.

\begin{example}
An important and non-trivial example of a pro-object of $\MEM_\ast$ is provided
by the local system over $\M_{1,2}$ whose fiber over $[E,x]$ is the Lie algebra
$\p(E',x)$ of the unipotent completion\footnote{Unipotent completion is briefly
reviewed in Section~\ref{sec:unipt_comp}.} of $\pi_1(E',x)$. Its restriction to
$\M_{1,\uu}$ is the local system whose fiber over $[E,\vv]$ is the Lie algebra
of the unipotent completion of $\pi_1(E',\vv)$.

Choose a parameter $w$ on the fiber $\Ebar_0$ over $q=0$ of the Tate curve
that takes the value $1$ at the identity and defines an isomorphism of the
smooth points $E_0$ of $\Ebar_0$ with $\Gm$. It is unique up to the involution
$w\mapsto 1/w$ of $\Ebar_0$. The tangent vector
$$
\ww_o : \Spec\Z((w)) \to E_0
$$
is integrally defined and is non-zero mod $p$ for all prime numbers $p$. It will
be used as a base point for both $E_0'$ and also for $E_\tate$ via the inclusion
$E_0\to E_\tate$. Note that $E_0'$ is isomorphic to $\Pminus$.

\begin{theorem}[\cite{hain:nodal}]
The Lie algebra $\p(E'_\tate,\ww_o)$ is a pro-object of $\MTM$. The inclusion
$E_0 \to E_\tate$ induces a morphism
$$
\p(\Pminus,\ww_o) \cong \p(E'_0,\ww_o) \to \p(E'_\tate,\ww_o).
$$
\end{theorem}

\begin{corollary}
\label{cor:elliptic_polylog}
There is an object $\bp$ of $\MEM_2$ whose fiber over $\vv_o$ is
$\p(E'_\tate,\ww_o)$.
\end{corollary}
\end{example}

\subsection{Variants: higher levels and more decorations}

When $r+n>0$, one can make a similar definition of mixed elliptic motives over the moduli stack $\M_{1,n+\vec{r}/\O_N}[N]$ of decorated genus 1 curves with a level $N\ge 1$ structure, where $\O_N = \Z[\bmu_N,1/N]$. One first has to choose an integrally defined tangent vector $\vv_o$ at a cusp with everywhere good reduction. (That is, a tangent vector defined over $\O_N$ that is non-zero mod $\wp$ for all prime ideals $\wp$ of $\O_N$.) When $N>1$ or $r+n>2$, there are several possible choices of tangent vector $\vv_0$. All give equivalent categories of mixed elliptic motives. This will be proved in \cite{hain:nodal}. The coordinate ring of the tannakian fundamental group of level $N$ universal mixed elliptic motives should be an ind-object of $\MTM(\O_N)$ and there might be an interesting connection to Deligne's work \cite{deligne:N}.

\subsection{A ``theorem of the fixed part''}
We conclude this section by proving a mild generalization of the Theorem of the Fixed Part. For this, we need the following result, which follows from GAGA.

\begin{lemma}
\label{lem:gaga}
Suppose that $k$ is a subfield of $\C$ and that $X$ is a smooth variety defined
over $k$ with an action by a $k$-algebraic group $G$. If $\cV$ is a $G$-
invariant vector bundle over $X$ with a rational, $G$-invariant connection
$\nabla : \cV \to \cV \otimes \Omega^1_X\otimes_{\O_X}k(X)$, then for all $P\in
X(k)$ the square
$$
\xymatrix{
H^0(X,(\cV,\nabla))^G \ar[r]  \ar[d] & V_P \ar[d] \cr
H^0(X^\an,(\cV^\an,\nabla))^G \ar[r] & V_P\otimes_k \C
}
$$
is cartesian. Here $V_P$ denotes the fiber of $\cV$ over $P$. \qed
\end{lemma}

We also need the stack version. When the connection is trivial on $G$-orbits, there is an isomorphism
$$
H^0(X,(\cV,\nabla))^G  \cong H^0(G\bbs X,(\cV,\nabla)).
$$
This implies that the square
$$
\xymatrix{
H^0(G\bbs X,(\cV,\nabla)) \ar[r]  \ar[d] & V_P \ar[d] \cr
H^0(G\bbs X^\an,(\cV^\an,\nabla))^G \ar[r] & V_P\otimes_k \C
}
$$
is also cartesian.

The next result is a version of the Theorem of the Fixed Part for universal mixed elliptic motives.

\begin{proposition}
\label{prop:fixed_part}
For all objects $\V$ of $\MEM_\ast$ there is an object $H^0(\M_{1,\ast},\V)$ of $\MTM$ whose Hodge realization is the canonical mixed Hodge structure on the invariants $H^0(\M_{1,\ast}^\an,\V^B)$ of $\V^B$. It is a sub object of the fiber $V$ of $\V$ over $\vv_o$.
\end{proposition}

\begin{proof}
The Theorem of the Fixed Part implies that $H^0(\M_{1,\ast},\V^B)$ has a natural MHS and that this is a sub-MHS of the Hodge realization of $V$ with $M_\dot=W_\dot$. Since the Hodge realization functor on $\MTM$ is fully faithful, $H^0(\M_{1,\ast},\V^B)$ is the Hodge realization of an object of $\MTM$. Denote it by $H^0(\M_{1,\ast},\V)$. Compatibility with the $\Q$-de Rham realization follows from Lemma~\ref{lem:gaga}.
\end{proof}

Lemma~\ref{lem:gaga} and Proposition~\ref{prop:fixed_part} also imply that the geometrically constant object of $\MEM_\ast$ that restricts to $H^0(\M_{1,\ast},\V)$ over $\vv_o$ is a sub-MEM of $\V$.

\begin{corollary}
\label{cor:purity}
Every $W$-pure object of $\MEM_\ast$ is semi-simple.
\end{corollary}

\begin{proof}
By standard arguments, it suffices to show that every extension
$$
0 \to \V \to \bE \to \Q(0) \to 0
$$
in $\MEM_\ast$, where $\bE$ is $W$-pure of weight 0, splits. The definition of universal mixed elliptic motives implies that, since $\bE$ is $W$-pure, as an $\SL_2(\Z)$-module, $E^B$ is isomorphic to a direct sum of copies of symmetric powers of $H^B$. It follows that
$$
H^0(\M_{1,\ast},\bE^B) \to H^0(M_{1,\ast},\Q)
$$
is a surjective morphism of objects of $\MTM$ of $W$-weight 0. Since these are trivial $\SL_2(\Z)$-modules, the two weight filtrations $M$ and $W$ are equal, so they are both also $M$-pure of weight 0. The morphism therefore splits. The choice of a splitting of $E \to \Q(0)$ in $\MTM$ induces a splitting in $\MEM_\ast$. 
\end{proof}

\section{Tannakian Considerations}
\label{sec:tannaka}

The category $\MEM_\ast$, $\ast \in \{1,\uu,2\}$, is a rigid abelian tensor
category over $\Q$. The functor that takes an object of $\MEM_\ast$ to its fiber
$(V,M_\dot)$ over $\vv_o$ defines a functor $\vv_o^\ast : \MEM_\ast \to \MTM$.
It is faithful, exact and preserves tensor products and unit objects.
Consequently, each fiber functor $\w : \MTM \to \Vec_F$ gives rise to a fiber
functor $\w\circ \vv_o^\ast : \MEM_\ast \to \Vec_F$ that we shall also denote by
$\w$.

\begin{proposition}
The category $\MEM_\ast$ is a neutral tannakian category over $\Q$. \qed
\end{proposition}

The standard fiber functors for both $\MTM$ and $\MEM_\ast$ are:
$$
\w^B : \MEM_\ast \to \Vec_\Q,\
\w^\DR : \MEM_\ast \to \Vec_\Q,\text{ and }
\w_\ell : \MEM_\ast\otimes\Ql \to \Vec_\Ql.
$$
Other useful fiber functors include $\Gr^W_\dot\w^B$, $\Gr^W_\dot\w^\DR$ and
$\Gr^W_\dot\w_\ell$.

Tannaka duality implies that, when $\w = \w^B$ and $\w = \w^\DR$, the category
$\MEM_\ast$ is equivalent to the category $\Rep(\pi_1(\MEM_\ast,\w))$ of finite
dimensional representations of $\pi_1(\MEM_\ast,\w)$.

The morphisms
$$
\xymatrix{
\M_{1,\uu} \ar[r]\ar[dr] & \M_{1,1}\ar[d] & \M_{1,2} \ar[l]\ar[dl] \cr
& \Spec \Z
}
$$
induce functors
$$
\xymatrix{
\MEM_\uu & \MEM_1 \ar[l] \ar[r] & \MEM_2 \cr
& \MTM(\Z) \ar[ul]\ar[u]\ar[ur]
}
$$
In Section~\ref{sec:restn} we show that there is a ``restriction functor''
$\MEM_2 \to \MEM_\uu$ and that $\MEM_1 \to\MEM_\uu$ factors through it: $\MEM_1
\to \MEM_2 \to \MEM_\uu$.

\subsection{Generalities}

The functor $c: \MTM \to \MEM_\ast$ that takes a mixed Tate motive
$(V,M_\dot)$ to the geometrically constant mixed elliptic motive with fiber
$(V,M_\dot)$ over $\vv_o$ is fully faithful and induces a surjective
homomorphism
$$
\pi_1(\MEM_\ast,\w) \to \pi_1(\MTM,\w)
$$
for each fiber functor $\w : \MTM \to \Vec_F$. It is split by the homomorphism
$\vv_o^\ast$.

Set \label{def:pigeom}
$$
\pi_1^\geom(\MEM_\ast,\w) = \ker\{\pi_1(\MEM_\ast,\w) \to \pi_1(\MTM,\w)\}.
$$
One thus has a split extension
$$
1 \to \pi_1^\geom(\MEM_\ast,\w) \to \pi_1(\MEM_\ast,\w)
\to \pi_1(\MTM,\w) \to 1
$$
which is split by $\vv_o^\ast$.

One also has the group $\pi_1(\MEM_\ast,\vv_o^\ast)$. It is an affine group
scheme in $\MTM$. More precisely:

\begin{proposition}
The coordinate ring of $\pi_1(\MEM_\ast,\vv_o^\ast)$ is a Hopf algebra in the
category of ind-objects of $\MTM$. The Betti, de~Rham and $\ell$-adic
realizations of $\O\big(\pi_1(\MEM_\ast,\vv_o^\ast)\big)$ are the coordinate
rings of
$$
\pi_1^\geom(\MEM_\ast,\w^B),\ \pi_1^\geom(\MEM_\ast,\w^\DR)
\text{ and }
\pi_1^\geom(\MEM_\ast\otimes\Ql,\w_\ell),
$$
respectively. \qed
\end{proposition}

A universal mixed elliptic motive $V$ is {\em semi-simple} if it is a
direct sum of simple universal elliptic motives:\label{def:MEMss}
$$
V = \bigoplus_{j=1}^N S^{m_j}\H(r_j).
$$
The category $\MEM_\ast^\ss$ of semi-simple universal elliptic motives is the
full tannakian subcategory of $\MEM_\ast$ generated by the $S^m\H(r)$. For each
of the standard fiber functors $\w$, we have
$$
\pi_1(\MEM_\ast^\ss,\w) \cong \GL(H_\w)
$$
where $V_\w$ denotes $\w(V)$ for all $V\in\MEM_\ast$. The simple object $\H$
corresponds to the defining representation of $\GL(H_\w)$, and the simple object
$\Q(-1)$ is $\Lambda^2 \H$, and thus corresponds to the 1-dimensional
representation $\det : \GL(H_\w) \to \Gm$. Combining these, we see that $S^m
\H(r)$ corresponds to the $m$th symmetric power of the defining representation,
twisted by $\det^{\otimes(-r)}$. Note that if $\V$ is a $W$-pure object $V$ of
$\MEM_\ast$ of weight $m$, then the scalar $\lambda\id_H$ matrix in $\GL(H_\w)$
acts on $V_\w$ as $\lambda^m\id_V$.

\begin{proposition}
\label{prop:W-splitting}
For each of the standard fiber functors, the surjection
\begin{equation}
\label{eqn:surj}
\pi_1(\MEM_\ast,\w) \to \GL(H_\w)
\end{equation}
is split and has prounipotent kernel.
\end{proposition}

\begin{proof}
Corollary~\ref{cor:purity} implies that the $W$-graded quotients $\Gr^W_m \V$ of an MEM $\V$ are semi-simple. That is, we have a functor $\Gr^W_\dot :\MEM_\ast \to \MEM_\ast^\ss$. It is exact and thus induces a homomorphism
\begin{equation}
\label{eqn:splitting}
\GL(H_\w) \to \pi_1(\MEM_\ast,\Gr^W_\dot\w).
\end{equation}
Each choice of an $F$-rational point ($F=\Q,\Ql$) of the ``path torsor'' $\Isom^\otimes(\w,\Gr^W_\dot\w)$ gives an isomorphism $\pi_1(\MEM_\ast,\Gr^W_\dot\w) \to \pi_1(\MEM_\ast,\w)$. Composing it with (\ref{eqn:splitting}) gives a splitting of (\ref{eqn:surj}).

The fact that the kernel is prounipotent is a consequence of the general fact that if $\cC$ is a neutral tannakian category whose category of semi-simple objects is $\cC^\ss$, then the kernel of the homomorphism $\pi_1(\cC,\w) \to \pi_1(\cC^\ss,\w)$ induced by the inclusion $\cC^\ss \to \cC$ has unipotent kernel.
\end{proof}

\begin{remark}
Later we will need to know that the functor $\V \mapsto \Gr^W_\dot\Gr^M_\dot V$
that takes a mixed elliptic motive $\V$ to the associated bigraded of its fiber
over $\vv_o$ is exact. This is proved in Appendix~\ref{sec:splittings}.
\end{remark}

\subsection{The restriction functor $\MEM_2 \to \MEM_\uu$}
\label{sec:restn}

Set $\cN = \cL^{-1}$. Recall that $\M_{1,2} = \E'$.
Since $\cN^\an$ is a covering space of $\E^\an$, there are neighbourhoods $U$
of the zero section of $\cN^\an$ and $V$ of the zero section of $\E^\an$ such
that $V$ gets mapped biholomorphically onto $U$ by the projection $\cN \to
\E$ and such that $U$ is a deformation retract of $\cN$. Denote the
complement of the zero section in $U$ and $V$ by $U'$ and $V'$. Then the
sequence of maps
$$
\xymatrix{
\M_{1,\uu}^\an & \ar@{_{(}->}[l]_(0.4){\simeq} U' \ar[r]^\approx &
V' \ar@{^{(}->}[r] & \M_{1,2}^\an		
}
$$
induces a homomorphism $\pi_1^\top(\M_{1,\uu},\vv_o)\to
\pi_1^\top(\M_{1,2},\vv_o)$. This induces a $G_\Q$-equivariant homomorphism
$\pi_1(\M_{1,\uu/\Qbar},\vv_o) \to \pi_1(\M_{1,2/\Qbar},\vv_o)$.

\begin{proposition}
\label{prop:restriction}
There is a natural restriction functor $\MEM_2 \to \MEM_\uu$ that is
the identity on the fiber over the base point $\vv_o$ and takes the
object $\V$ with monodromy representation $\rho_V : \pi_1^\top(\M_{1,2},\vv_o)
\to \Aut V^B$ to an object of $\MEM_\uu$ with monodromy representation the
composite
$$
\pi_1^\top(\M_{1,\uu},\vv_o) \to \pi_1^\top(\M_{1,2},\vv_o) \to \Aut V^B.
$$
\end{proposition}

\begin{proof}
The only task is to prove that if $\V$ is an object of $\MEM_2$, then its de
Rham realization $(\cV^\DR,\nabla)$, a vector bundle with connection over
$\M_{1,2/\Q}$, pulls back to a vector bundle with connection over
$\M_{1,\uu/\Q}$ satisfying the required compatibilities with the other
realizations.

Denote the zero section of $\E_{/\Q}$ by $Z$ and its ideal sheaf in $\O_\E$ by
$\m_Z$. Then
$$
\O_\cN = \Gr^\dot\O_\E := \bigoplus_{n\ge 0} \m_Z^n/\m_Z^{n+1}.
$$

There is a natural isomorphism $\O_\E \cong \O_Z\oplus\m_Z$ that is induced by
the inclusion $Z \hookrightarrow \E$ and the projection $\E \to \M_{1,1}\cong
Z$. This induces the eigenspace decomposition
$$
\O_\E/\m_Z^2 = \O_Z \oplus \m_Z/\m_Z^2 = \O_Z \oplus \cN^\vee
$$
under the natural $\Gm$-action, where $(\blank)^\vee$ denotes dual. Note that $\O_\cN$ is the graded $\O_Z$ algebra
$$
\O_\cN =
\mathrm{Sym}_{\O_Z} \cN = \O_Z \oplus \cN^\vee \oplus S^2 \cN^\vee \oplus
\cdots
$$
The pullback of $\cV^\DR$ from $\E$ to $\cN$ is defined to be the graded
$\O_\cN$-module
$$
\Gr^\dot \cV^\DR :=
\big(\cV^\DR\otimes_{\O_\E} {\O_\E/\m_Z^2}\big) \otimes_{\O_\E/\m_Z^2}\O_\cN.
$$
The connection $\nabla$ on $\cV^\DR$ induces a connection\footnote{Examples
of how this works in practice can be found in \cite[\S13]{hain:kzb}.}
$$
\nablabar : \Gr^\dot\cV^\DR \to (\Gr^\dot\cV^\DR)\otimes \Omega^1_\cN(\log Z)
$$
which is characterized by the properties:
\begin{enumerate}

\item it has regular singular points along $Z$,

\item it is invariant under the $\Gm$-action on $\cN$,

\item the residues of $\nabla$ and $\nablabar$ along $Z$, which lie in
$H^0(Z,\End\cV^\DR|_Z)$, are equal.

\end{enumerate}

To complete the proof, we sketch a proof of the compatibility of the monodromy
representations with the Betti realizations. More precisely, we explain why the
diagram
$$
\xymatrix{
\pi_1(\M_{1,\uu}^\an,\vv) \ar[dr]_{\rho_\nablabar}\ar[r] &
\pi_1(\M_{1,2}^\an,\vv) \ar[d]^{\rho_\nabla} \cr
& \Aut V_P
}
$$
commutes, where $P\in Z$, $\vv$ is a non-zero element of $T_P \M_{1,2}^\an$ (so
$\vv \in \M_{1,\uu}^\an$), and where $\rho_\nabla$ and $\rho_\nablabar$ are the
monodromy representations of $\nabla$ and $\nablabar$. It suffices to consider
the case where $\vv = \partial/\partial \xi$. Let $\tau_o\in \h$ be a point that
lies above $P\in Z$.

By standard ODE theory (see, for example, \cite{wasow}), since $\nabla$ has
nilpotent residue at each point along $Z$, there is a polynomial
$$
p(\tau,\xi,T) \in \Aut V_P \otimes_\C \O(\h\times \C)[T]
$$
whose regularized value $p(\tau_o,1,0)$ at $\vv := \partial/\partial \xi \in T_P
\M_{1,2}$ is the identity of $V_P$ and with the property that all flat sections
of $\cV$ are of the form $vp(\tau,\xi,\log\xi)$ for some $v\in V_P$. The
monodromy of $\nabla$ about an element $\gamma \in \pi_1(\M_{1,2}^\an,\vv)$ is
obtained by taking the analytic continuation $\tilde{p}$ of $p$ along $\gamma$
and then taking its regularized value at $\vv$:
$$
\rho_\nabla(\gamma) = \tilde{p}(\tau_o,1,0).
$$
On the other hand, ODE theory implies that the flat sections of $\nablabar$ over
$\M_{1,\uu}$ are of the form $p(\tau,1,\log \xi)$. The regularized monodromy
representation of $\nablabar$ on $\gamma \in \pi_1(\M_{1,\uu}^\an,\vv)$ is
computed from $\tilde{p}$. Since we may assume that $\gamma$ lies in the
neighbourhood $U'$ of $Z$, we have the claimed compatibility of monodromy
representations.
\end{proof}

\section{Hodge Theoretic Considerations}
\label{sec:hodge}

For $F=\Q$ or $\R$, define $\MHS_F(\M_{1,\ast},\H)$\label{def:MHS_H} to be the
category of admissible variations of $F$-MHS over $\M_{1,\ast}^\an$ whose weight
graded quotients are sums of polarized variations of Hodge structure of the form
$S^m\H \otimes A$, where $A$ is a Hodge structure.\footnote{It is important to
note that $A$ is not necessarily of type $(p,p)$.} It is neutral tannakian.
Denote the forgetful functor that takes a variation to the $F$-vector space
underlying its fiber over $\vv_o$ by $\w_o$. When $F=\Q$ we will omit the
subscript.

\begin{theorem}
\label{thm:surjective}
The forgetful functor $\MEM_\ast \to \MHS(\M_{1,\ast},\H)$ that takes a
universal mixed elliptic motive $\V$ to the associated variation of MHS $\V^\an$
over $\M_{1,\ast}^\an$ is fully faithful. Consequently,
$$
\pi_1(\MHS(\M_{1,\ast},\H),\w_o) \to \pi_1(\MEM_\ast,\w^B)
$$
is surjective.
\end{theorem}

\begin{proof}
Since $\Hom_\cC(\bA,\bB)$ is naturally isomorphic to
$\Hom_\cC(\Q(0),\Hom_\Q(\bA,\bB))$ for all objects $\bA$ and $\bB$ of
$\MEM_\ast$ and $\cC = \MEM_\ast$ and $\MHS(\M_{1,\ast},\H)$, it suffices to
show that
$$
\Hom_{\MEM_\ast}(\Q(0),\V) \to \Hom_{\MHS(\M_{1,\ast},\H)}(\Q(0),\V^\an)
$$
is an isomorphism for all objects $\V$ of $\MEM_\ast$.

Injectivity is easily proved and is left to the reader. We will prove
surjectivity. Suppose that $\V = (V,\V^\DR,\rho_V)$ is an object of $\MEM_\ast$
and that $\phi^\an : \Q(0)^\an \to \V^\an$ is a morphism of variations of
MHS. It induces a morphism of MHS
$$
v : \Q(0) \to V^\MHS := (V^B,W_\dot,F^\dot)
$$
of limit MHS over $\vv_o$. Since the Hodge realization functor $\MTM \to \MHS$
is fully faithful \cite{deligne-goncharov}, this map is the Hodge realization of
a morphism $\Q(0) \to V$ in $\MTM$.

Identify the morphism $v : \Q(0) \to V$ with the set of the realizations ($v^B
\in V^B$, $v^\DR \in V^\DR$, $v_\ell \in V_\ell$) of the image of $1\in \Q(0)$.
For each prime $\ell$
$$
v_\ell \in H^0(\M_{1,\ast/\Q},V_\ell)
$$
as $v_\ell$ is fixed by both $G_\Q$ and $\pi_1(\M_{1,\ast/\Qbar},\vv_o)$.

To complete the construction of a morphism $\phi : \Q(0) \to \V$ in $\MEM_\ast$
that lifts $\phi^\an$, we need to construct a morphism
$\Q(0)_{\M_{1,\ast/\Q}}^\DR \to \V^\DR$. Such a map corresponds to an element of
$H^0(\M_{1,\ast/\Q},(\cV,\nabla))$, where $\V^\DR = (\cV,\nabla)$. The vector
$v^B\in V^B$ lies in the image of the restriction map
$$
H^0(\Mbar_{1,\ast}^\an,(\cV^\an,\nabla)) \to V^B\otimes\C,
$$
where $\cV^\an$ denotes Deligne's canonical extension of
$\V^B\otimes\O_{\M_{1,\ast}}^\an$ to $\Mbar_{1,\ast}^\an$. Lemma~\ref{lem:gaga}
implies that $v^\DR$ lies in the image of the corresponding homomorphism
$$
H^0(\Mbar_{1,\ast/\Q},(\cV,\nabla)) \to V^\DR.
$$
\end{proof}

\part{Simple Extensions in $\MEM_\ast$}
\label{part:exts}

In this section we compute the groups $\Ext^1_{\MEM_\ast}(\Q,S^m\H(r))$ for all $m$ and $r$. These computations are made possible by the fact that the Hodge realization is fully faithful, which means that an extension in $\MEM_\ast$ is non-trivial if and only its Hodge realization is non-trivial. We also prove the corresponding statement for $\ell$-adic Galois realizations. This part concludes with a discussion of $\Ext^2_{\MEM_\ast}(\Q,S^m\H(r))$ and its relation to standard conjectures in number theory. 

\section{Cohomology of $\M_{1,\ast}$}
\label{sec:coho}

Here we recall the basic facts we need. Full details in the Hodge case can be
found in \cite{hain:modular}. References in the $\ell$-adic case are given
below.

\subsection{The cohomology of $\M_{1,1}^\an$} The first basic observation is
that, since we are regarding $\M_{1,1}^\an$ as the orbifold $\SL_2(\Z)\bbs\h$,
there is a natural isomorphism
$$
H^\dot(\M_{1,1}^\an,\V) \cong H^\dot(\SL_2(\Z),V)
$$
for all $\SL_2(\Z)$-modules $V$, where $\V$ denotes the local system over
$\M_{1,1}^\an$ that corresponds to $V$. Since $\SL_2(\Z)$ is virtually free,
this implies that $H^j(\M_{1,1}^\an,\V)$ vanishes for all $j\ge 2$ whenever $\V$
is a local system of $\Q$-modules. Since $-I\in \SL_2(\Z)$ acts as $(-1)^m$ on
$H^\dot(\SL_2(\Z),S^m H)$, it follows that $H^\dot(\M_{1,1}^\an,S^m\H)$ vanishes
when $m$ is odd. When $m=0$, $H^1(\M_{1,1}^\an,\Q) = 0$ as $H_1(\SL_2(\Z);\Z) =
\Z/12$.

Well-known results of Shimura, Manin, Drinfeld and Zucker (cf.\
\cite{lang,zucker}) imply that the cohomology groups
$H^\dot(\M_{1,1}^\an,S^n\H)$ and their mixed Hodge structures can be expressed
in terms of modular forms. Denote the space of (holomorphic) modular forms of
$\SL_2(\Z)$ of weight $w$ by $\fM_w$ and the subspace of cusp forms by
$\fM_w^o$. Recall that these are trivial when $w$ is odd.

In this section we regard $\H$ as a polarized variation of Hodge structure over
$\M_{1,1}^\an$ of weight 1. As explained in Section~\ref{sec:moduli}, its fiber
over the tangent vector $\tate$ is
$$
H^B = \Q\aa \oplus \Q\bb \text{ and } H^\DR = F^0 H^\DR = \Q\aa \oplus \Q\bw,
$$
where the comparison isomorphism $H^B\otimes\C \to H^\DR\otimes\C$ takes $\aa$
to $\aa$ and identifies $\bw$ with $-2\pi i\bb$. The Hodge filtration $F^1H^\DR$
is spanned by $\bw$.

For $f \in \fM_{2n+2}$ define
$$
\w_f = 2\pi i f(\tau) \bw^{2n} d\tau
= (2\pi i)^{2n+1} f(\tau) (\bb - \tau\aa)^{2n} d\tau
\in E^1(\h)\otimes S^{2n}H.
$$
This is a holomorphic 1-form on $\h$ with values in $S^{2n}\cH_\h$. Since it is
$\SL_2(\Z)$ invariant and since the section $\bw$ spans $F^1\cH_\h$,
$$
\w_f \in H^0(\M_{1,1}^\an,\Omega^1\otimes F^{2n}S^{2n}\cH).
$$
Since this form is holomorphic, it defines a class in
$H^1(\M_{1,1}^\an,S^{2n}\H)$. When $f$ is the Eisenstein series $G_{2n}$, we
denote $\w_f$ by $\psi_{2n}$.\label{def:psi}

\begin{remark}
\label{rem:scalings}
The particular scalings have been chosen to have the property that, if $f \in
\fM_{2n+2}$ has rational Fourier coefficients, then $\w_f$ is an element of the
$\Q$ de~Rham cohomology $H_\DR^1(\M_{1,1/\Q},S^{2n}\cH)$. (Cf.
\cite[\S21]{hain:kzb}.) In particular, the class of $\psi_{2n+2}$ under the
residue map
$$
H^1(\M_{1,1}^\an,S^{2n}\H) \to \Q(-2n-1)
$$
is a rational generator of $\Q(-2n-1)^\DR$.
\end{remark}

The following statement is a composite of well-known results. (Cf.\
\cite{beilinson-levin}, \cite{zucker}.)

\begin{theorem}[Shimura, Zucker, Manin-Drinfeld]
\label{thm:eichler_shimura}
For each $n> 0$, the group $H^1(\M_{1,1}^\an,S^{2n}\H)$ has a natural MHS with
Hodge numbers $(2n+1,0)$, $(0,2n+1)$ and $(2n+1,2n+1)$. The map $\fM_{2n+2} \to
H^1(\M_{1,1}^\an,S^{2n}\H_\C)$ that takes $f$ to the class of $\w_f$ is
injective and has image $F^{2n+1}H^1(\M_{1,1}^\an,S^{2n}\H)$. The image of
$\fM_{2n+2}^o$ is
$$
H^{2n+1,0}(\M_{1,1}^\an,S^{2n}\H) := F^{2n+1}W_{2n+1}H^1(\M_{1,1}^\an,S^{2n}\H).
$$
The MHS $H^1(\M_{1,1}^\an,S^{2n}\H)$ splits over $\Q$. The copy of
$\Q(-2n-1)^\DR$ is spanned by $\psi_{2n+2}$.
\end{theorem}

\subsection{The Hodge structures associated to a Hecke Eigenform}

It is useful to decompose the real MHS on $H^1(\M_{1,1}^\an,S^{2n}\H)$ under the
action of the Hecke correspondences. Denote the set of normalized Hecke eigen
cusp forms of $\SL_2(\Z)$ of weight $w$ by $\B_w$.  It is a basis of $\fM_w^o$.

For $f\in \B_{2n+2}$, let $K_f$ be the subfield of $\C$ generated by the Fourier
coefficients of $f$. It is a totally real field. For each prime number $p$,
$T_p(f) = a_p f$ where $a_p$ is the $p$th Fourier coefficient of $f$. Let
$$
\label{def:V_f}
V_f := \bigcap_p \ker\big\{T_p - a_p\id : 
H^1(\M_{1,1}^\an,S^{2n}\H_\R) \to H^1(\M_{1,1}^\an,S^{2n}\H_\R)\big\}
$$
This is a 2-dimensional $\R$ Hodge substructure of
$H^1(\M_{1,1}^\an,S^{2n}\H)$. Its complexification $V_{f,\C}$ is spanned by
$\w_f$ and its complex conjugate. As a real Hodge structure, $\M_{1,1}^\an$
decomposes
\begin{equation}
\label{eqn:decomp}
H^1(\M_{1,1}^\an,S^m\H_\R) \cong \R(-2n-1)
\oplus \bigoplus_{f\in \B_{2n+2}} V_f.
\end{equation}
The copy of $\R(-2n-1)$ is spanned by the class $\psi_{2n+2}$.

There is a similar decomposition of $H^1(\M_{1,1}^\an,S^{2n}\H)$ as a $\Q$-MHS.
Suppose that $f\in \B_{2n+2}$. As above, let $a_p$ be the $p$th Fourier
coefficient of $f$. Let $m_p(x)$ be the minimal polynomial of $a_p$ over $\Q$.
Then
$$
\label{def:M_f}
M_p = \bigcap_p\ker\big\{m_p(T_p):
H^1(\M_{1,1}^\an,S^{2n}\H_\Q) \to H^1(\M_{1,1}^\an,S^{2n}\H_\Q)\big\}.
$$
This is a simple $\Q$ Hodge substructure $M_f$ of $H^1(\M_{1,1}^\an,S^{2n}\H_\Q)$ with the property that
$$
M_f\otimes\R = \bigoplus_{g} V_g,
$$
where $g$ ranges over the Hecke eigen cusp forms whose Fourier expansions are
Galois conjugate to that of $f$. It has dimension $2\dim_\Q K_f$. There is a
decomposition
$$
H^1(\M_{1,1}^\an,S^{2n}\H_\Q) \cong \Q(-2n-1) \oplus \bigoplus_{f} M_f.
$$
where $f$ ranges over the Galois equivalence classes of Hecke eigen cusp forms
of weight $2n+2$.

\subsection{Cohomology of $\M_{1,\uu}^\an$ and $\M_{1,2}^\an$}

These are easily deduced from the cohomology of $\M_{1,1}^\an$ using the Leray
spectral sequences of the projections $\M_{1,\uu}\to\M_{1,1}$ and
$\M_{1,2}\to\M_{1,1}$. We state the results and leave the proof to the reader.

\begin{proposition}
\label{prop:coho_vec}
There is an isomorphism
$$
H^1(\M_{1,\uu}^\an,\Q) \cong \Q(-1).
$$
This group is generated by the class of the form (\ref{eqn:G2}) corresponding to
the Eisenstein series $G_2$.\footnote{It is natural to extend the definition of
the $\psi_{2n}$ ($n>1$) to the case $n=1$ by defining $\psi_2 = 2\pi i
G_2(\tau)d\tau - (1/2)d\xi/\xi$. Then $\psi_2$ is a generator of
$H^1_\DR(\M_{1,\uu/\Q})$ and $H^1(\M_{1,\uu}^\an,\Q(1))$. Note also that
$\psi_2=-\frac{1}{24}\frac{dD}{D}$, where $D=u^3-27v^2$ denotes the
discriminant function on $\M_{1,\vec{1}}$.} If $m>0$, then the projection to
$\M_{1,1}^\an$ induces an isomorphism
$$
H^1(\M_{1,\uu}^\an,S^m\H) \cong H^1(\M_{1,1}^\an,S^m\H)
$$
and the cup product $H^1(\M_{1,\uu}^\an,S^m\H)\otimes H^1(\M_{1,\uu}^\an,\Q) \to
H^2(\M_{1,\uu}^\an,S^m\H)$ is an isomorphism of MHS, so that
$$
H^2(\M_{1,\uu}^\an,S^m\H) \cong H^1(\M_{1,1}^\an,S^m\H)(-1).
$$
\end{proposition}

In the case of $\M_{1,2}^\an$, we have:

\begin{proposition}
There are natural isomorphisms of MHS
$$
H^1(\M_{1,2}^\an,\H) \cong \Q(-1) \text{ and } H^2(\M_{1,2}^\an,\H)=0.
$$
If $m>1$, then the projection induces an isomorphism
$$
H^1(\M_{1,2}^\an,S^m\H) \cong H^1(\M_{1,1}^\an,S^m\H)
$$
in degree 1, and in degree 2, there is an isomorphism of MHS
$$
H^2(\M_{1,2}^\an,S^m\H) \cong
H^1(\M_{1,1}^\an,S^{m+1}\H)\oplus H^1(\M_{1,1}^\an,S^{m-1}\H)(-1).
$$
In particular, $H^2(\M_{1,2},S^m\H)$ is non-trivial only when $m$ is odd. \qed
\end{proposition}

This implies that
\begin{multline*}
H^2(\M_{1,2},\H\otimes S^{2n}\H) \cr
\cong
H^1(\M_{1,1}^\an,S^{2n+2}\H) \oplus H^1(\M_{1,1}^\an,S^{2n}\H)^{\oplus 2}(-1)
\oplus H^1(\M_{1,1}^\an,S^{2n-2}\H)(-2).
\end{multline*}
One can show that the cup product
$$
H^1(\M_{1,2}^\an,\H) \otimes H^1(\M_{1,2}^\an,S^{2n}\H)
\to H^2(\M_{1,2},\H\otimes S^{2n}\H)
$$
is injective and has image in the middle factor 
$H^1(\M_{1,1}^\an,S^{2n+2}\H)^{\oplus 2}(-1)$.

\section{Relative Unipotent Completion}
\label{sec:rel}

The next task is to describe the tannakian fundamental group of
$\MHS(\M_{1,\ast},\H)$. Computing it will allow us to bound the size of
$\pi_1(\MEM_\ast,\w^B)$. To do this, we need to review some basic facts about
relative completion of discrete groups and facts about the relative completion
of modular groups and their relation to extensions of variations of MHS.
References for this material include
\cite{hain:malcev,hain:modular,hain:db_coho}.

\subsection{Relative unipotent completion}

Suppose that $\G$ is a discrete group and that  $R$ is a reductive algebraic
group over a field $F$ of characteristic zero. Suppose that $\rho : \G \to R(F)$
is a Zariski dense representation. Consider the full subcategory $\cR(\G,\rho)$ of
$\Rep_F(\G)$ consisting of those $\G$-modules $M$ that admit a filtration
$$
0 = M_0 \subset M_1 \subset M_2 \subset \cdots \subset M_n = M
$$
in $\Rep_F(\G)$, where the action of $\G$ on each graded quotient $M_j/M_{j-1}$
factors through a rational representation $R \to \Aut(M_j/M_{j-1})$ via $\rho$.
This category is neutral tannakian over $F$. The completion of $\G$ relative to
$\rho$ is the affine group scheme $\pi_1(\cR(\G,\rho),\w)$ over $F$, where $\w$
takes a module $M$ to its underlying vector space. It is an extension
$$
1 \to \U(\G,\rho) \to \pi_1(\cR(\G,\rho),\w) \to R \to 1
$$
where $\U(\G,\rho)$ is prounipotent. There is a natural homomorphism
$$
\rhotilde : \G \to \pi_1(\cR(\G,\rho),\w)(F)
$$
whose composition with the canonical quotient map $\pi_1(\cR(\G,\rho),\w) \to R$
is $\rho$. It is Zariski dense. The coordinate ring of $\pi_1(\cR(\G,\rho),\w)$
is the Hopf algebra of matrix entries of the objects of $\cR(\G,\rho)$.

\subsection{Cohomological properties}

Set $\cG = \pi_1(\cR(\G,\rho),\w)$ and $\U = \U(\G,\rho)$. Denote the Lie
algebra of $\U$ by $\u$. It is pronilpotent. The action of $\cG$ on $\u$ induced
by the inner action of $\cG$ on $\U$ induces an action of $\cG$ on $H^\dot(\u)$.
Standard arguments imply that it factors through the quotient mapping $\cG\to
R$, so that $H^\dot(\u)$ is an $R$-module. The homomorphism $\G \to \cG(F)$
induces a homomorphism $H^\dot(\cG,V) \to H^\dot(\G,V)$. There is a natural
isomorphism
$$
H^\dot(\cG,V) \cong [H^\dot(\U,V)]^R.
$$
The following basic property of relative completion is easily proved by adapting
the proof of \cite[Thm.~4.6]{hain-matsumoto:weighted} or
\cite[Thm.~8.1]{hain-matsumoto:survey}. A proof in the de~Rham case is given in
\cite{hain:malcev}. 

\begin{proposition}
\label{prop:coho}
For all $R$-modules $V$, there are natural isomorphisms
$$
H^1(\cG,V) \cong [H^1(\u)\otimes V]^R \cong H^1(\G,V)
$$
and an injection
$$
H^2(\cG,V) \cong [H^2(\u)\otimes V]^R \hookrightarrow H^2(\G,V).
$$
\end{proposition}

Using the fact (cf.\ \cite{hain:db_coho}) that, for all $\cG$-modules $V$, there
are natural isomorphisms
$$
\Ext^\dot_{\Rep(\cG)}(A,B) \cong H^\dot(\cG,\Hom_F(A,B))
\cong H^\dot(\u,\Hom_F(A,B))^R,
$$
one can restate the previous result as follows. For all $\cG$-modules $A$ and
$B$, the natural map
$$
\Ext^j_{\Rep(\cG)}(A,B) \to \Ext^j_{\Rep_F(\G)}(A,B)
$$
is an isomorphism in degree 1 and injective in degree 2.

\subsection{Unipotent completion}
\label{sec:unipt_comp}

Unipotent completion is the special case where $R$ is the trivial group. In 
this case, $\cR(\G,\rho)$ is the category of unipotent representations of $\G$.
The unipotent completion of $\G$ over $F$ will be denoted by $\G_{/F}^\un$.

Suppose that $(E,0)$ is an elliptic curve over $\C$. Set $E'=E-\{0\}$. Let $b$ be a
base point of $E'$. (So $b$ can be a point of $E'$ or a non-zero tangent vector
$\vv \in T_0 E$.) Denote the unipotent completion of $\pi_1(E',b)$ by $\cP(E,b)$
and its Lie algebra by $\p(E',b)$. For all such $E$ and $b$, the coordinate ring
$\O(\cP(E,b))$ and Lie algebra $\p(E',b)$ of $\cP(E,b)$ have natural mixed Hodge
structures that are compatible with their algebraic structures. The weight
filtration $W_\dot$ of $\p(E',b)$ is (essentially) its lower central series:
$$
W_{-n}\p(E',b) = L^n \p(E',b).
$$
There is a canonical isomorphism $\Gr^W_\dot\p(E',b) \cong \L(H_1(E))$ of the
associated weight graded of $\p(E',b)$ with the free Lie algebra generated by
$H_1(E)$. These form a variation of MHS over $\M_{1,2}^\an$ when $b$ is a point
of $E'$, and over $\M_{1,\uu}^\an$ when $b$ is a tangential base point at $0$.

There is also a canonical MHS on the unipotent completion of
$\pi_1(E_\tate',\ww_o)$. It is the limit MHS associated to the tangent vector
$\vv_o$ of $\Mbar_{1,2}$ at the identity of the nodal cubic. It has a weight
filtration $W_\dot$ as above, and a relative weight filtration $M_\dot$. This
limit MHS is computed explicitly in \cite[\S18]{hain:kzb}.

Denote the commutator subalgebra of $\p = \p(E',b)$ by $\p'$ and its commutator
subalgebra by $\p''$. Then there is a short exact sequence
$$
0 \to H_1(\p') \to \p/\p'' \to H_1(E) \to 0
$$
in the category of Lie algebras with mixed Hodge structure. There is a canonical
isomorphism of mixed Hodge structures
$$
H_1(\p') \cong [\Sym H_1(E)](1) = \bigoplus_{m\ge 0} S^m H^1(E)(m+1).
$$
Consequently, these variations yield extensions
$$
0 \to S^m \H(m+1) \to \V_m \to \H(1) \to 0
$$
of variations of MHS over $\M_{1,2}^\an$ and $\M_{1,\uu}^\an$ for each $m\ge 0$.
But since the log of monodromy about the cusp of $E'$ lies in $\p'$ and spans
the copy of $\Q(1)$ in $H_1(\p')$, the monodromy of $\V_m$ about any loop in a
fiber of $\M_{1,\uu}^\an \to \M_{1,1}^\an$ is  trivial. Consequently, the
restriction of this extension to $\M_{1,\uu}^\an$ descends to a variation over
$\M_{1,1}^\an$. 

By standard manipulations, each of these variations is equivalent to a variation
\begin{equation}
\label{eqn:extn_p}
0 \to S^{m+1} \H(m+1) \oplus S^{m-1}\H(m) \to \bE_m \to \Q \to 0.
\end{equation}
Similar constructions can be made in the $\ell$-adic case.

The following result follows from Proposition~\ref{prop:coho_vec} when $\ast
=\uu$ and the vanishing of $H^1(\M_{1,\ast}^\an,\Q)$ when $\ast\neq \uu$.

\begin{proposition}
\label{prop:unipt_mcg}
The unipotent completion of $\pi_1^\un(\M_{1,\ast}^\an,x)$ is trivial when $\ast
\in \{1,2\}$ and is isomorphic to $\Ga$ when $\ast = \uu$.
\end{proposition}

\subsection{Relative completion of $\pi_1^\top(\M_{1,\ast},\vv_o)$}

Denote by $\cG_\ast^\rel(x)$ \label{def:Grel} the completion of
$\pi_1^\top(\M_{1,\ast},x)$ with respect to the natural homomorphism
$\pi_1^\top(\M_{1,\ast},x) \to \SL(H_x^B)$. It is an extension
$$
1 \to \U_\ast^\rel(x) \to \cG_\ast^\rel(x) \to \SL(H_x) \to 1,
$$
where $\U_\ast^\rel(x)$ is prounipotent and $\SL(H_x)$ denotes $\SL(H_x^B)$. 

\begin{theorem}[\cite{hain:malcev,hain:modular}]
For each choice of base point $x$ of $\M_{1,\ast}^\an$, the coordinate ring of
$\cG_\ast^\rel(x)$ has a natural mixed Hodge structure. The coordinate rings
$\{\O(\cG_\ast^\rel(x))\}_x$ form an ind-object of $\MHS(\M_{1,\ast},\H)$. The
MHS on the coordinate ring of $\cG_\ast^\rel(\vv_o)$ is the limit MHS of this
VMHS associated to the tangent vector $\vv_o$. In this case, the weight
filtration is the relative weight filtration $M_\dot$.
\end{theorem}

Denote the Lie algebra of $\cG_\ast^\rel(x)$ by $\g_\ast^\rel(x)$ and the Lie
algebra of $\U_\ast^\rel(x)$ by $\u_\ast^\rel(x)$. The Lie algebra
$\g_\ast^\rel(x)$ is dual to the cotangent space $\m/\m^2$, where $\m$ denotes
the ideal of $\O(\cG_\ast^\rel(x))$ that vanish at the identity. Its bracket is
induced by the skew symmetrized coproduct of $\O(\cG_\ast^\rel(x))$. The theorem
implies that $\g_\ast^\rel(x)$ is an inverse limit of Lie algebras with mixed
Hodge structure and that $\u_\ast^\rel(x)$ is a pronilpotent Lie algebra in the
category of mixed Hodge structures. These mixed Hodge structures have the
property that
$$
\g_\ast^\rel(x) = W_0\g_\ast^\rel(x) \text{ and }
\u_\ast^\rel(x) = W_{-1}\g_\ast^\rel(x).
$$

We now fix the base point of $\M_{1,\ast}^\an$ to be $\vv_o$ and write
$\cG_\ast^\rel$ for $\cG_\ast(\vv_o)$, $\U_\ast^\rel$ for $\U_\ast^\rel(\vv_o)$
and similarly for their Lie algebras. The coordinate ring of $\cG_\ast^\rel$ and
each of the Lie algebras has two weight filtrations, $W_\dot$ and $M_\dot$. Both
are compatible with the algebraic operations.

Since $\pi_1^\top(\M_{1,1},\tate) \cong \SL_2(\Z)$ has virtual cohomological
dimension 1, Proposition~\ref{prop:coho} implies that $H^2(\u_1) = 0$. This
implies that $\u_1$ is topologically freely generated by $H_1(\u)$. (See the
discussion in Section~\ref{sec:presentations}.) Proposition~\ref{prop:coho}
also implies that there are natural $\SL(H)$-module isomorphisms
$$
H^1(\u_1^\rel) \cong
\bigoplus_{n\ge 1} H^1(\M_{1,1}^\an,S^{2n}\H)\otimes (S^{2n}H)^\vee
\cong \bigoplus_{n\ge 1} H^1(\M_{1,1}^\an,S^{2n}\H)\otimes (S^{2n}H)(2n).
$$
Both are isomorphisms of MHS.

Recall from \cite[\S14]{hain:modular} that the exact sequence
$$
0 \to \Z \to \pi_1^\top(\M_{1,\uu},\vv_o) \to \pi_1^\top(\M_{1,1},\tate) \to 1
$$
induces an exact sequence
$$
0 \to \Ga \to \cG_\uu^\rel \to \cG_1^\rel \to 1
$$
and that this sequence splits, so that there is a natural isomorphism
$$
\cG_\uu^\rel \cong \cG_1^\rel \times \Ga.
$$
The projection onto $\Ga$ is induced by the Hurewicz homomorphism
$$
\pi_1^\top(\M_{1,\uu},\vv_o) \to H_1(\M_{1,\uu}^\an,\Q) \cong \Q,
$$
and is obtained by integrating the 1-form (\ref{eqn:G2}) along loops.

Applying completion to the extension
$$
1 \to \pi_1(E'_\tate,\ww_o) \to \pi_1^\top(\M_{1,2},\vv_o)
\to \pi_1^\top(\M_{1,1},\tate) \to 1
$$
gives an extension
$$
1 \to \cP \to \cG_2^\rel \to \cG_1^\rel \to 1
$$
where $\cP$ is the unipotent completion of $\pi_1(E'_\tate,\ww_o)$. The
corresponding sequence of Lie algebras
$$
0 \to \p \to \g_2^\rel \to \g_1^\rel \to 0
$$
is exact in the category of pro-objects of $\MHS_\Q$.

\begin{proposition}
There are natural $\SL(H)$-equivariant isomorphisms of MHS
$$
H^1(\u_1) \cong \bigoplus_{n\ge 1}
\Big(H^1_\cusp(\M_{1,1}^\an,S^{2n}\H)(2n)\oplus \Q(-1)\Big)\otimes S^{2n}H,
$$
$$
H^1(\u_\uu) \cong H^1(\u_1)\oplus \Q(-1)
\text{ and } H^1(\u_2) \cong H^1(\u_1) \oplus H. \qed
$$
\end{proposition}

\subsection{Variations of MHS over $\M_{1,\ast}^\an$}

The category $\MHS$ of $\Q$-MHS is neutral tannakian. Let $\w^B$ be the fiber
functor that takes a MHS to its underlying $\Q$-vector space. Set $\pi_1(\MHS) =
\pi_1(\MHS,\w^B)$.

Suppose that $\cG$ is an affine group scheme over $\Q$ whose coordinate ring is
a Hopf algebra in the category of ind-objects of $\MHS$. A {\em Hodge
representation} of $\cG$ is a $\cG$-module $V$, endowed with a MHS for which the
corresponding coaction
$$
V \to V \otimes \O(\cG)
$$ 
is a morphism of MHS. The action of $\pi_1(\MHS)$ on $\O(\cG)$ respects its Hopf
algebra structure. We can thus form the semi-direct product $\pi_1(\MHS)\ltimes
\cG$. The Hodge representations of $\cG$ are precisely the representations of
$\pi_1(\MHS)\ltimes\cG$.

Recall the definition of $\MHS(\M_{1,\ast},\H)$ from the beginning of Section~\ref{sec:hodge}.
A proof of the following result is sketched in \cite{hain:db_coho}. A more
detailed proof of a more general result will appear in a joint paper with
Gregory Pearlstein.

\begin{theorem}
Taking the fiber at the base point $x$ defines an equivalence of categories from
$\MHS(\M_{1,\ast},\H)$ to the category of Hodge representations of
$\cG_\ast^\rel(x)$. Equivalently, for all base points, there is a natural
isomorphism
$$
\pi_1(\MHS(\M_{1,\ast},\H),\w_x^B) \cong
\pi_1(\MHS)\ltimes \cG_\ast^\rel(x),
$$
where $\w_x^B$ denotes the fiber functor that takes a variation to the 
$\Q$-vector space underlying its fiber over $x$.
\end{theorem}

This result applies to tangential base points as well. This follows as the
weight filtration $W_\dot$ can be recovered from the relative weight filtration
$M_\dot$ and the monodromy logarithm $N$. (Cf.\ \cite[Lem.~4.3]{hain:db_coho}.)
It implies that a Hodge representation of $\cG_\ast^\rel$ must also respect the
weight filtration $W_\dot$.

\begin{proposition}
\label{prop:m=w}
If $\V$ is an admissible variation of mixed Hodge structure over
$\M_{1,\ast}^\an$ whose monodromy representation factors through the relative
completion of $\pi_1(\M_{1,\ast}^\an,\vv_o)$, then the weight filtration
$W_\dot$ and the relative weight filtration $M_\dot$ on the fiber over the
tangential base point $\vv_o$ are equal if and only if $\V$ has unipotent
monodromy. If $\ast \in \{1,2\}$, every admissible variation with unipotent
monodromy is constant. If $\ast = \uu$, every admissible unipotent variation is 
pulled back from an admissible unipotent variation of MHS over $\C^\ast$ along
the discriminant map $\M_{1,\uu}^\an \to \C^\ast$.
\end{proposition}

\begin{proof}
Suppose that $\V$ an admissible variation of MHS over $\M_{1,\ast}^\an$ whose
monodromy factors through the relative completion of
$\pi_1(\M_{1,\ast}^\an,\vv_o)$. Denote its fiber over $\vv_o$ by $V$. Let $N \in
\End V$ be the local monodromy logarithm. The definition of the relative weight
filtration implies that $M_\dot V = W_\dot V$ if and only if $N$ acts trivially
on $\Gr^W_\dot V$. But since each $W$ graded quotient of $V$ is a sum of $S^m
H(r)$, and since $N$ acts on $H$ as $\aa\partial/\partial\bb$, it follows that
$M_\dot = W_\dot$ if and only if $\Gr^W_\dot \V$ is a trivial local system ---
that is, if and only if $\V$ is a unipotent local system.

If $\V$ is a unipotent local system, then its monodromy representation factors
through the unipotent completion of $\pi_1(\M_{1,\ast}^\an,x)$. By
Proposition~\ref{prop:unipt_mcg}, this is trivial when $\ast = 1,2$. So $\V$ is
a trivial local in this case, and therefore trivial as a variation of MHS as
well. When $\ast = \uu$, the discriminant map $\M_{1,\uu}^\an \to \C^\ast$
induces an isomorphism on unipotent fundamental groups. So the last statement
follows from the classification of unipotent variations of MHS in
\cite{hain-zucker}.
\end{proof}

\subsection{Density results}

The forgetful functor
$$
\MHS(\M_{1,\ast},\H) \to 
\cR(\pi_1^\top(\M_{1,\ast},\vv_o),\rho)
$$
induces a homomorphism $\cG_\ast \to \pi_1(\MEM_\ast,\w^B)$. There is thus
a natural homomorphism
$$
\pi_1^\top(\M_{1,\ast},\vv_o) \to \pi_1(\MEM_\ast,\w^B)(\Q).
$$

\begin{theorem}
\label{thm:density}
The image of the natural homomorphism $\cG_\ast^\rel \to \pi_1(\MEM_\ast,\w^B)$
is $\pi_1^\geom(\MEM_\ast,\w^B)$. Consequently, $\pi_1^\geom(\MEM_\ast,\w^B)$ is
the Zariski closure of the image of the natural homomorphism
$\pi_1^\top(\M_{1,\ast},\vv_o) \to \pi_1(\MEM_\ast,\w^B)(\Q)$.
\end{theorem}

\begin{proof}
Since $\MEM_\ast \to \MHS(\M_{1,\ast},\H)$ takes geometrically constant MEMs to
constant variations of MHS, the surjection $\pi_1(\MHS(\M_{1,\ast},\H)) \to
\pi_1(\MEM_\ast,\w^B)$ commutes with their projections to $\pi_1(\MHS)$ and to
$\pi_1(\MTM,\w^B)$. So one has the commutative diagram
$$
\xymatrix{
1 \ar[r] & \cG_\ast^\rel \ar[r]\ar[d] &
\pi_1(\MHS)\ltimes\cG_\ast^\rel \ar[r]\ar[d] &
\pi_1(\MHS) \ar[r]\ar[d] & 1
\cr
1 \ar[r] & \pi_1^\geom(\MEM_\ast,\w^B) \ar[r] & \pi_1(\MEM_\ast,\w^B) \ar[r]
& \pi_1(\MTM,\w^B) \ar[r] & 1
}
$$
The surjectivity of the middle map, which follows from
Theorem~\ref{thm:surjective}, and the splitting $\pi_1(\MEM_\ast,\w^B) \cong
\pi_1(\MTM_\ast,\w^B)\ltimes \pi_1^\geom(\MEM_\ast,\w^B)$, imply the surjectivity of the
left hand map. The last assertion follows as the canonical homomorphism
$\pi_1^\top(\M_{1,\ast},\vv_o) \to \cG_\ast^\rel(\Q)$ is Zariski dense.
\end{proof}

Recall that $\K=\K^B$ is the kernel of the natural homomorphism
$\pi_1(\MTM,\w^B) \to \Gm$. It is prounipotent. Denote the kernel of
$\pi_1(\MEM_\ast,\w^B) \to \pi_1(\MTM,\w^B)$ by $\U_\ast^\MEM$. \label{def:umem}
Since the commutative diagram (in which all fiber functors are $\w^B$)
$$
\xymatrix@C=20pt@R=18pt{
& 1 \ar[d] & 1 \ar[d] & 1 \ar[d]
\cr
1 \ar[r] & \U^\geom_\ast \ar[d]\ar[r] & \U_\ast^\MEM \ar[r]\ar[d] &
\K \ar[d]\ar[r] & 1
\cr
1 \ar[r] & \pi_1^\geom(\MEM_\ast) \ar[d]\ar[r] & \pi_1(\MEM_\ast)
\ar[r]\ar[d] & \pi_1(\MTM) \ar[d]\ar[r] & 1
\cr
1 \ar[r] & \SL(H) \ar[r]\ar[d] & \GL(H) \ar[r]\ar[d] & \Gm \ar[r]\ar[d] & 1
\cr
& 1 & 1 & 1
}
$$
has exact rows and columns, we have:

\begin{corollary}
The kernel of the natural homomorphism $\pi_1(\MEM_\ast,\w) \to \GL(H_\w)$
is prounipotent. It is a split extension
$$
1 \to \U_\ast^\geom \to \U_\ast^\MEM \to \K \to 1
$$
of $\K$ by $\U_\ast^\geom := \U_\ast^\MEM\times_{\pi_1(\MEM_\ast,\w)} \pi_1^\geom(\MEM_\ast,\w)$.
\label{def:ugeom} \qed
\end{corollary}

The splitting is induced by the base point $\vv_o$.

\section{Extensions of Variations of Mixed Hodge Structure over
$\M_{1,\ast}^\an$}
\label{sec:vmhs}

In this section we give a brief summary of results from \cite{hain:db_coho} in
the case of $\M_{1,\ast}^\an$ and apply these to compute the extension groups in
$\MHS(\M_{1,\ast},\H)$. It will be natural to work with real variations of MHS
as well as $\Q$-variations. So in this section, $F$ will be $\Q$ or $\R$.

\subsection{Computation of the Ext groups}
By the results of the previous section, the category $\MHS_F(\M_{1,\ast},\H)$ is
the category of representations of \label{def:Ghat}
$$
\cGhat_\ast := \pi_1(\MHS_F)\ltimes \cG^\rel_\ast.
$$
Consequently, for an object $\V$ of $\MHS_F(\M_{1,\ast},\H)$ with fiber $V$ over
$\vv_o$, there is a natural isomorphism
$$
\Ext^\dot_{\MHS_F(\M_{1,\ast},\H)}(F,\V) \cong H^\dot(\cGhat_\ast,V)
$$
that is compatible with products. The conjugation action of $\cGhat_\ast$ on
$\cG_\ast^\rel$ induces an action of $\pi_1(\MHS_F)$ on
$H^\dot(\cG_\ast^\rel,V)$. So $H^\dot(\cG_\ast^\rel,V)$ has a natural $F$-MHS.
An explicit construction of this MHS is given in \cite{hain:db_coho}.

The spectral sequence of the extension
$$
1 \to \cG_\ast^\rel \to \cGhat_\ast \to \pi_1(\MHS_F) \to 1
$$
and the fact that $\Ext^j_{\MHS_F}(A,B)$ vanishes for all MHS $A$ and $B$ when
$j>1$, implies that there is an exact sequence
\begin{equation}
\label{eqn:ses_ext}
0 \to \Ext^1_\MHS(F,H^{j-1}(\cG_\ast^\rel,V)) \to
\Ext^j_{\MHS_F(\M_{1,\ast},\H)}(F,\V)
\to \G H^j(\cG_\ast^\rel,V) \to 0,
\end{equation}
where $\G$ denotes the functor $\Hom_\MHS(F,\blank)$. For all objects $\V$ of
$\MHS_F(\M_{1,\ast},\H)$, there is a natural homomorphism
\begin{equation}
\label{eqn:morphism}
\Ext^\dot_{\MHS_F(\M_{1,\ast},\H)}(F,\V) \to H^\dot_\cD(\M_{1,\ast}^\an,\V).
\end{equation}
In the case of modular curves, we get the best possible result.

\begin{theorem}[\cite{hain:db_coho}]
The natural homomorphism $H^\dot(\cG_\ast^\rel,V) \to H^\dot(\M_{1,\ast}^\an,\V)$
is an isomorphism of MHS. Consequently, the homomorphism (\ref{eqn:morphism})
is an isomorphism.
\end{theorem}

Plugging this into the exact sequence (\ref{eqn:ses_ext}) yields the following
useful computational tool.

\begin{corollary}
For all $j\ge 0$, there is an exact sequence
$$
0 \to \Ext^1_{\MHS_F}(F,H^{j-1}(\M_{1,\ast}^\an,\V)) \to
\Ext_{\MHS_F(\M_{1,\ast},\H)}^j(F,\V) \to \G H^j(\M_{1,\ast}^\an,\V) \to 0.
$$
\end{corollary}

This and the cohomology computations in Section~\ref{sec:coho} imply:

\begin{proposition}
\label{prop:vanishing}
The extension groups $\Ext^j_{\MHS_F(\M_{1,\ast},\H)}(F,\V)$ vanish when $j>2$
and $\ast \in \{1,\uu\}$, and when $j>3$ if $\ast = 2$.
\end{proposition}

First the 1-extensions. The cohomology computations in Section~\ref{sec:coho}
imply that:

\begin{theorem}
\label{thm:ext1}
When $\ast \in \{1,\uu,2\}$, there is an exact sequence
$$
0 \to \Ext^1_\MHS(F,F(r)) \to \Ext^1_{\MHS(\M_{1,\ast},\H)}(F,F(r))
\to \G H^1(\M_{1,\ast}^\an,F(r)) \to 0.
$$
The right-hand group is trivial except when $\ast = \uu$ and $r=1$. In this
case, it is spanned by the class $\psi_2$ associated to the Eisenstein series
$G_2$. When $\ast \in \{1,\uu\}$, and $m>0$, we have
$$
\Ext^1_{\MHS(\M_{1,\ast},\H)}(F,S^m\H(r)) =
\begin{cases}
F & m=2n > 0 \text{ and } r=2n+1,\cr
0 & \text{otherwise.}
\end{cases}
$$
When $\ast = 2$, we have
$$
\Ext^1_{\MHS(\M_{1,2},\H)}(F,S^m\H(r)) =
\begin{cases}
F & m = 1 \text{ and } r=1,\cr
F & m=2n > 0 \text{ and } r=2n+1,\cr
0 & \text{otherwise.}
\end{cases}
$$
When $m=2n> 0$, the extension corresponds to the Eisenstein series
$G_{2n+2}$.\footnote{These extensions correspond to the elliptic polylogarithms
of Beilinson--Levin \cite{beilinson-levin}. Their restriction to $\M_{1,1}^\an$
is described in \cite[\S13.3]{hain:modular}.} \qed
\end{theorem}

\begin{remark}
The generator of $\Ext^1_{\MHS(\M_{1,2},\H)}\big(\Q,\H(1)\big)$ is the extension whose fiber over $(E,0,x)$
$$
0 \to H_1(E,\Q) \to H_1(E,\{0,x\};\Q) \to \widetilde{H}_0(\{0,x\};\Q) \to 0.
$$

\end{remark}

In degree 2, we consider only extensions of real variations of MHS, and so, only
real Deligne--Beilinson cohomology. The reason for doing this should become
apparent in Section~\ref{sec:relns}. We will compute the answer only for $\ast
\in \{1,\uu\}$ as we shall not need the case $\ast = 2$.

\begin{theorem}
\label{thm:ext2}
For all integers $m\ge 0$ and $r$ we have
$$
\Ext^2_{\MHS(\M_{1,\ast},\H)}(\R,S^m\H(r)) =
\begin{cases}
\R & \ast = \uu,\ m = 0 \text{ and } r\ge 2,\cr
\R \oplus \bigoplus_{f\in \B_{2n+2}} V_f & m = 2n > 0,\ r\ge 2n+2,\cr
0 & \text{otherwise.}
\end{cases}
$$
\end{theorem}

\begin{proof}
Since $H^2(\M_{1,\ast}^\an,S^m\H)$ vanishes for all $m\ge 0$ when $\ast \in
\{1,\uu\}$, we have
$$
H^2_\cD(\M_{1,\ast},S^m\H_\R(r)) = \Ext^1_\MHS(\R,H^1(\M_{1,\ast},S^m\H(r))).
$$
This vanishes when $m$ is odd. So suppose that $m=2n$. The decomposition
(\ref{eqn:decomp}) implies that if $r\ge 2n+2$, where $n > 0$ when $\ast = 1$
and $n\ge 0$ when $\ast = \uu$, we have
\begin{align*}
H^2_\cD(\M_{1,\ast},S^{2n}\H_\R(r))
&= \Ext^1_\MHS(\R,H^1(\M_{1,\ast},S^{2n}\H(r)))
\cr
&\cong \Ext^1_\MHS(\R,\R(r-2n-1)) \oplus \bigoplus_{f\in \B_{2n+2}} V_f(r)
\end{align*}
and that this ext group vanishes in the remaining cases.

To complete the proof, recall that if $A$ is a real mixed Hodge structure, then
$$
\Ext^1_\MHS(\R,A(r)) \cong A_\C/(i^r A_\R + F^r A).
$$
Since $V_f$ has Hodge type $(2n+1,0)$ and $(0,2n+1)$, 
$$
V_{f,\C} = i^r V_f + F^r V_f
$$
when $r < 2n+2$ and $F^0V_f = 0$ when $r\ge 2n+2$. This implies
$\Ext^1_\MHS(\R,V_f(r))$ vanishes when $r<2n+2$ and is isomorphic to $V_f$
when $r\ge 2n+2$. Similarly, $\Ext^1_\MHS(\R,\R(r-2n-1))$ vanishes when
$r-2n-1\le 0$ and is isomorphic to $\R$ when $r\ge 2n-2$.
\end{proof}

\subsection{The $\Frbar_\infty$ action}

In this section, we take $\ast \in \{1,\uu\}$. Since $\M_{1,\ast}$  and the
universal elliptic curve over it are defined over $\Z$, complex conjugation (aka
``real Frobenius'') $\Fr_\infty \in \Gal(\C/\R)$ acts on $\E^\an \to
\M_{1,\ast}^\an$. This implies that $\Fr_\infty$ acts on $\H_\R$, and thus on
$H^\dot(\M_{1,\ast}^\an, S^{2n}\H_\R)$ as well. It extends to a $\C$-linear
involution of $H^\dot(\M_{1,\ast}^\an, S^{2n}\H_\C)$. As is well known, this can
be computed in terms of modular symbols. (Cf.\ \cite[Lem.~17.7]{hain:modular}.)

For each $f\in \B_{2n+2}$, $V_f$ is preserved by $\Fr_\infty$. Write
$$
V_f = V_f^+ \oplus V_f^-
$$
where $\Fr_\infty$ acts at $+1$ on $V_f^+$ and $-1$ on $V_f^-$. The action on
the class of the Eisenstein series is given by:

\begin{lemma}
\label{lem:real_frob}
The real Frobenius operator $\Fr_\infty$ multiplies the class of $\psi_{2n+2}$ 
in $H^1(\M_{1,\ast}^\an,S^{2n}\H_\C)$ by $-1$.
\end{lemma}

\begin{proof}
Since $\Fr_\infty$ commutes with the action of the Hecke operators, $\Fr_\infty$
multiplies the eigenform $\psi_{2n+2}$ by $\pm 1$. To determine this
multiple we use the residue map
$$
\Res : H^1(\M_{1,\ast}^\an,S^{2n}\H_\C) \to S^{2n}H/\im N \cong \R \bb^{2n},
$$
where $N$ denotes the monodromy logarithm $-\aa\partial/\partial \bb$. It 
anti-commutes with $\Fr_\infty$ as $\Fr_\infty$ reverses the orientation of the
$q$-disk. The result follows as $\bb^{2n}$ is invariant and
$\Res_{q=0}\psi_{2n+2} = (2\pi i \bb)^{2n}G_{2n+2}(0)$.
\end{proof}

Let
$$
\Frbar_\infty : H^\dot(\M_{1,\ast}^\an, S^{2n}\H_\C)
\to H^\dot(\M_{1,\ast}^\an, S^{2n}\H_\C)
$$
be the composition of $\Fr_\infty$ with complex conjugation on the coefficients.
This is $\C$-antilinear. It corresponds to complex conjugation on
$H^\dot_\DR(\M_{1,\ast/\R},S^{2n}\cH_\R)\otimes\C$ under the comparison
isomorphism. It therefore preserves the Hodge filtration and the real structure.
It lifts to an involution of the real Deligne--Beilinson cohomology
$$
H_\cD^\dot(\M_{1,\ast}^\an,S^{2n}\H(r)) \cong
\Ext^\dot_{\MHS(\M_{1,\ast},\H)}(\R,S^{2n}\H(r)).
$$

The (Betti) copy of $\Q(r-2n-1)$ in $H^1(\M_{1,\ast}^\an,S^{2n}\H_\Q(r))$ is
spanned by
$$
(2\pi i)^{r-2n-1}\psi_{2n+2}.
$$
(Cf.\ Remark~\ref{rem:scalings}.) Lemma~\ref{lem:real_frob} implies that
$\Frbar_\infty$ multiplies it by $(-1)^{r}$. These observations yield the
following refinement of Theorems~\ref{thm:ext1} and \ref{thm:ext2}.

\begin{proposition}
\label{prop:deligne_coho}
If $\ast \in \{1,\uu,2\}$, then
$$
H^1_\cD(\M_{1,\ast}^\an,\R(r))^{\Frbar_\infty} \cong
\begin{cases}
\R & \ast = \uu,\ r=1 \cr
\R & r\ge 3 \text{ odd,}\cr
0 & \text{otherwise.}
\end{cases}
$$
If $\ast \in \{1,\uu\}$, then for all integers $m > 0$ and $r$
$$
H^1_\cD(\M_{1,\ast}^\an,S^{m}\H_\R(r))
= H^1_\cD(\M_{1,\ast}^\an,S^{m}\H_\R(r))^{\Frbar_\infty}.
$$
It is non-zero if and only if $m=2n$, $r=2n+1$ and $m\ge 0$ when $\ast = \uu$
and $m>0$ when $\ast = 1$. In degree 2
$$
H^2_\cD(\M_{1,\ast}^\an,S^{m}\H_\R(r))^{\Frbar_\infty}
= 
\begin{cases}
\R & \ast = \uu,\ m=0 \text{ and } r \ge 2 \text{ even,} \cr
\R \oplus \bigoplus_{f\in \B_{2n+2}} V_f^- &
m=2n>0,\ r \ge 2n+2 \text{ even},\cr
 \bigoplus_{f\in \B_{2n+2}} V_f^+ &
 m=2n>0,\ r > 2n+2 \text{ odd},\cr
0 & \text{otherwise.}
\end{cases}
$$
\end{proposition}

\section{Weighted and Crystalline Completion}
\label{sec:weighted}

In this section we review the definitions of the $\ell$-adic weighted and
crystalline completion functors and make some initial observations about the
crystalline completion of $\pi_1(\M_{1,/\ast/\Q},\vv_o)$. More details
can be found in \cite{hain-matsumoto:weighted,hain-matsumoto:survey}.

Fix a prime number $\ell$. Suppose that $\G$ is a profinite group and that $R$
is a reductive algebraic group over $\Ql$, endowed with a central cocharacter
$c:\Gm \to R$. Suppose that $\rho : \G \to R(\Ql)$ is a continuous, Zariski
dense representation. Say that a representation $V$ of $R$ has weight $m$ if its
pullback to $\Gm$ via $c$ is of weight $m$. Since $c$ is central, Schur's Lemma
implies that every irreducible representation of $R$ has a weight.

Consider the full subcategory $\cR^\cts_\Ql(\G,\rho,c)$ of the category of
continuous $\G$-modules $\Rep_\Ql^\cts(\G)$ consisting of those $\G$-modules $M$
that are finite dimensional and admit a (necessarily unique) filtration
$$
0 = W_n M \subseteq W_{n+1}M \subseteq \cdots \subseteq
W_{N-1}M \subseteq W_N M = M,
$$
where the action of $\G$ on each graded quotient $\Gr^W_m M$ factors through a
rational representation $R \to \Aut(\Gr^W_m M)$ of weight $m$ via $\rho$. Note
that every irreducible $R$-module is an object of $\cR^\cts_\Ql(\G,\rho,c)$.

This category is neutral tannakian over $\Ql$. The weighted completion of $\G$
relative to $\rho$ is the affine group scheme
$\pi_1(\cR^\cts_\Ql(\G,\rho,c),\w)$ over $\Ql$, where $\w$ is the fiber functor
that takes a module $M$ to its underlying vector space. It is an extension
$$
1 \to \U(\G,\rho,c) \to \pi_1(\cR^\cts_\Ql(\G,\rho,c),\w) \to R \to 1
$$
where $\U(\G,\rho,c)$ is prounipotent. There is a natural homomorphism
$$
\rhotilde : \G \to \pi_1(\cR^\cts_\Ql(\G,\rho,c),\w)(\Ql)
$$
whose composition with the canonical quotient map
$\pi_1(\cR^\cts_\Ql(\G,\rho,c),\w) \to R$ is $\rho$. It is Zariski dense.

Levi's Theorem implies that the central cocharacter $c : \Gm \to R$ can be
lifted to a cocharacter $\ctilde : \Gm \to \pi_1(\cR^\cts_\Ql(\G,\rho,c),\w)$.
So every object $V$ of $\cR^\cts_\Ql(\G,\rho,c)$ is a $\Gm$-module via
$\ctilde$. Denote the subspace of $V$ on which $\Gm$ acts via the $n$th power of
the standard  character by $V_n$. The inclusion $V_n \hookrightarrow V$ induces
an isomorphism $V_n \cong \Gr^W_n V$.

\begin{proposition}[{\cite[Thm.~3.12]{hain-matsumoto:weighted}}]
The isomorphism $V \cong \bigoplus_{n\in\Z} V_n$ defines a splitting
$$
V \cong \bigoplus_{n\in \Z} \Gr^W_\dot V
$$
of the weight filtration that is natural with respect to morphisms of
$\cR^\cts_\Ql(\G,\rho,c)$. It is compatible with duals and tensor products.
\end{proposition}

Although this isomorphism is natural, it is {\em not} canonical as it depends on the choice of lift $\ctilde$ of the central cocharacter $c$.

\begin{corollary}
For each $n\in \Z$, the functor $\Gr^W_n : \cR^\cts_\Ql(\G,\rho,c) \to
\Rep_F(R)$ is exact.
\end{corollary}

\begin{example}
\label{ex:wtd_comp}
Here we recall from \cite{hain-matsumoto:weighted} the setup for the weighted
completion of $\pi_1(\Spec\Z[1/\ell])$. Let $\G = \pi_1(\Spec\Z[1/\ell],\Qbar)$,
$R = \Gm$ and $\rho : \G \to \Gm(\Ql)$ be the $\ell$-adic cyclotomic character
$\chi_\ell :G_\Q \to \Zl^\times\subset \Ql^\times$. The defining representation
of $\Gm$ thus corresponds to $\Ql(1)$. Since this has weight $-2$, we define the
central cocharacter to be the homomorphism $c : \Gm \to \Gm$ defined by $z
\mapsto z^{-2}$. The weighted completion of $\G$ with respect to $\chi_\ell$ and
$c$ is
$$
\A_\ell^\wtd = \pi_1(\cR^\cts_\Ql(\pi_1(\Spec\Z[1/l],\Qbar),\rho,c),\w).
$$
\end{example}

\subsection{Crystalline completion}
\label{sec:cris}
Suppose that $K \subset \Qbar$ is a number field. Set $G_K = \Gal(\Qbar/K)$. We
continue with the notation of the previous paragraph. Suppose in addition that
there is a surjection $\G \to G_K$ with a distinguished splitting $s : G_K \to
\G$. Define an object $V$ of $\cR^\cts_\Ql(\G,\rho,c)$ to be {\em crystalline}
(with respect to $s$) if the representation $\rho \circ s : G_K \to \GL(V)$ is
unramified at all primes that do not lie over $\ell$ and that it is crystalline
at all primes that lie over $\ell$.\footnote{In some applications it is natural
relax this definition and allow crystalline representations to be ramified at a 
set $S$ of primes. We will not need to do that here as $\M_{1,\ast/\Z}$ has
good reduction at all primes.}

Let $\cR^\cris_\Ql(\G,\rho,c)$ be the full subcategory of
$\cR^\cts_\Ql(\G,\rho,c)$ whose objects are the crystalline objects. It is
tannakian. The $\ell$-adic crystalline completion of $(\G,s)$ is defined to be
$$
\pi_1(\cR^\cris_\Ql(\G,\rho,c),\w).
$$
It is a quotient of the weighted completion $\pi_1(\cR^\cts_\Ql(\G,\rho,c),\w)$.

As in the case of weighted completion, every object of $\cR^\cts_\Ql(\G,\rho,c)$
has a natural weight filtration $W_\dot$, the functors $\Gr^W_n$ are exact. Each
lift $\ctilde$ of the central cocharacter $c$ to
$\pi_1(\cR^\cts_\Ql(\G,\rho,c),\w)$ defines a natural splitting of the weight
filtration of each object of $\cR^\cts_\Ql(\G,\rho,c)$.

\subsection{Crystalline completion of $G_\Q$}

The setup is similar to that in Example~\ref{ex:wtd_comp}. Let $\G=G_\Q$,
$R=\Gm$ and $\rho : G_\Q \to \Gm(\Ql)$ the $\ell$-adic cyclotomic character
$\chi_\ell$. The central cocharacter $c : \Gm \to \Gm$ is defined by $z\mapsto
z^{-2}$. We also take $K=\Q$ and the homomorphism $\G \to G_\Q$ to be the
identity. Define
$$
\A_\ell^\cris = \pi_1(\cR^\cris_\Ql(G_\Q,\chi_\ell,c),\w)
$$
There is a canonical surjection $\A_\ell^\wtd \to \A_\ell^\cris$, where
$\A_\ell^\wtd$ denotes the weighted completion of $\pi_1(\Spec\Z[1/\ell],\Qbar)$
defined in Example~\ref{ex:wtd_comp}.

The $\ell$-adic realization functor $\MTM\to \cR^\cris_\Ql(G_\Q,\chi_\ell,c)$ induces a homomorphism $\A_\ell^\cris \to \pi_1(\MTM,\w_\ell)$.

\begin{theorem}[\cite{hain-matsumoto:weighted}]
For all $\ell$, the natural homomorphism $\A_\ell^\cris \to \pi_1(\MTM,\w_\ell)$
is an isomorphism.
\end{theorem}

\subsection{Weighted and crystalline completion of
$\pi_1(\M_{1,\ast},\vv_o)$}

The monodromy of the local system $\H_\ell$ over $\M_{1,\ast/\Z}$ induces a
homomorphism
$$
\rho_\ell: \pi_1(\M_{1,\ast/\Z[1/\ell]},\vv_o) \to \GL(H_\ell).
$$
Define $c: \Gm \to \GL(H)$ to be the cocharacter $\lambda \mapsto \lambda\id_H$.
It is central and assigns $W$-weight 1 to $H_\ell$. Recall from Section~\ref{sec:etale} that, as a $G_\Q$-module, $H_\ell$ is isomorphic to $\Ql(0)\oplus\Ql(-1)$. Define \label{def:Gcris}
$$
\cG_\ast^{\wtd,\ell} =
\pi_1(\cR^\cts_\Ql(\pi_1(\M_{1,\ast/\Z[1/\ell]},\vv_o),\rho_\ell,c),\w).
$$
To define the $\ell$-adic crystalline completion of
$\pi_1(\M_{1,\ast/\Q},\vv_o)$, we take $s$ to be the section of
$\pi_1(\M_{1,\ast/\Q},\vv_o) \to G_\Q$ induced by $\vv_o$. Define
$$
\cG_\ast^{\cris,\ell} =
\pi_1(\cR^\cris_\Ql(\pi_1(\M_{1,\ast/\Q},\vv_o),\rho_\ell,c),\w).
$$
Both groups are extensions of $\GL(H_\ell)$ by a prounipotent group:
$$
1 \to \U^{\wtd,\ell}_\ast \to \cG^{\wtd,\ell}_\ast \to \GL(H_\ell) \to 1
$$
$$
1 \to \U^{\cris,\ell}_\ast \to \cG^{\cris,\ell}_\ast \to \GL(H_\ell) \to 1.
$$

\begin{proposition}
\label{prop:surjective}
There are homomorphisms
$$
\xymatrix{
\cG^\rel_\ast \otimes\Ql \ar[r] & \cG^{\wtd,\ell}_\ast \ar@{->>}[r] &
\cG^{\cris,\ell}_\ast \ar@{->>}[r] & \A_\ell^\cris,
}
$$
the last two of which are surjective, such that the diagram
$$
\xymatrix@C=18pt{
\pi_1^\top(\M_{1,\ast},\vv_o) \ar[r]\ar[d] &
\pi_1(\M_{1,\ast/\Q},\vv_o) \ar[r]\ar[d] &
\pi_1(\M_{1,\ast/\Z[1/\ell]},\vv_o) \ar[d] \ar[r] &
\pi_1(\Spec\Z[1/\ell]) \ar[d]
\cr
\cG^\rel_\ast(\Ql) \ar[r] & \cG^{\wtd,\ell}_\ast(\Ql) \ar[r] &
\cG^{\cris,\ell}_\ast(\Ql) \ar[r] & \A_\ell^\cris
}
$$
commutes. The sequence $\cG^\rel_\ast\otimes\Ql \to
\cG^{\cris,\ell}\to\A_\ell^\cris \to 1$ is exact and the morphism
$\cG^{\cris,\ell}\to\A_\ell^\cris$ is split.

\end{proposition}

\begin{proof}
The first homomorphism is induced by the forgetful functor
$$
\cR^\cts_\Ql(\pi_1(\M_{1,\ast/\Z[1/\ell]},\vv_o),\rho,c) \to
\cR_\Ql(\pi_1^\top(\M_{1,\ast},\vv_o),\rho).
$$
The existence of the middle surjection follows from the fact that
$$
\cR^\cris_\Ql := \cR^\cris_\Ql(\pi_1(\M_{1,\ast/\Q},\vv_o),\rho_\ell,c)
$$
is a full subcategory of $\cR^\cts_\Ql$. The existence of the right-hand square
follows from the commutativity of the diagram
$$
\xymatrix{
\pi_1(\M_{1,\ast/\Q},\vv_o) \ar[rr]\ar[d] && G_\Q \ar[d]^{\chi_\ell} \cr
\GL(H_\ell) \ar[rr]^{\det^{-1}} && \Gm(\Ql) \cr
& \Gm(\Ql) \ar[ul]^{c=(\blank)^1\id_H}\ar[ur]_{c=(\blank)^{-2}}
}
$$
The last statement follows from the right exactness of completion and from the
fact that one has functors $\cR^\cris_\Ql(G_\Q,\chi_\ell,c) \to \cR^\cris_\Ql
\to \cR^\cris_\Ql(G_\Q,\chi_\ell,c)$ whose composite is the identity. The first is the inclusion of the geometrically constant objects, the second is restriction to the fiber over $vv_o$.
\end{proof}


\section{Extensions of $\ell$-adic sheaves over $\M_{1,\ast/\Q}$}

Our goal in this section is to prove the $\ell$-adic analogues of the
results in Section~\ref{sec:vmhs}. The work \cite{olsson} of 
Martin Olsson, which we use as a ``black box'', is key. These results are
not needed in the sequel. They are included to show how the computations in
this paper are related to certain fundamental questions in number theory.

The $\ell$-adic analogues of the groups
$\Ext_{\MHS(\M_{1,\ast},\H)}^\dot(\R,S^m\H(r))$ are the groups
$H^\dot(\cG^{\cris,\ell}_\ast,S^m H_\ell(r))$. In this section, we compute these
cohomology groups in degree 1 and bound them in degree 2.

For convenience, we denote the category
$\cR^\cris(\pi_1(\M_{1,\ast/\Z[1/\ell]},\vv_o),\rho_\ell,c)$ by
$\MEM_\ast^\ell$, so that $\MEM_\ast^\ell$ is equivalent to
$\Rep(\cG_\ast^{\cris,\ell})$ and
$$
H^\dot(\cG^{\cris,\ell}_\ast,S^m H_\ell(r))
= \Ext^\dot_{\MEM_\ast^\ell}(\Ql,S^m\H_\ell(r)).
$$

In this and subsequent sections, we will use the following notation. Denote the
category of $G_\Q$ modules that are unramified at all $p\neq \ell$ and
crystalline at $\ell$ by $\cC_\ell$. For an object $V$ of $\cC_\ell$, define
$$
\Hfte^\dot(G_\Q,V) := \Ext^\dot_{\cC_\ell}(\Ql,V).
$$
This group is typically denoted by $H_f^\dot(G_\Q,V)$. We have chosen a
non-standard notation to avoid a notational conflict as $f$ will often denote
a cusp form of $\SL_2(\Z)$. \label{def:h_fte}

\subsection{Computation of $H^1(\cG^{\cris,\ell}_\ast,S^m H_\ell(r))$}

The following result is the \'etale analogue of Theorem~\ref{thm:ext1}.

\begin{proposition}
\label{prop:cris_h1}
When $\ast \in \{1,\uu\}$, we have
$$
H^1(\cG^{\cris,\ell}_\ast,S^mH_\ell(r)) =
\begin{cases}
\Ql & \ast = 1,\ m = 0 \text{ and $r\ge 3$ odd,}\cr
\Ql & \ast = \uu,\ m = 0 \text{ and $r\ge 1$ odd,}\cr
\Ql & m=2n > 0 \text{ and } r=2n+1,\cr
0 & \text{otherwise.}
\end{cases}
$$
When $\ast = 2$, we have
$$
H^1(\cG^{\cris,\ell}_\ast,S^mH_\ell(r)) =
\begin{cases}
\Ql & m = 0 \text{ and $r\ge 3$ odd,}\cr
\Ql & m = 1 \text{ and } r=1,\cr
\Ql & m=2n > 0 \text{ and } r=2n+1,\cr
0 & \text{otherwise.}
\end{cases}
$$
\end{proposition}

\begin{proof}
We will prove the result when $\ast = 1$. The remaining cases are left to the
reader. The proof uses weighted completion, which is easier to compute than
crystalline completion. Since the natural homomorphism $\cG_1^{\wtd,\ell} \to
\cG_1^{\cris,\ell}$ it surjective (Prop.~\ref{prop:surjective}), it induces an
injection
\begin{equation}
\label{eqn:injection}
H^1(\cG_1^{\cris,\ell},S^m H_\ell(r)) \hookrightarrow
H^1(\cG_1^{\wtd,\ell},S^m H_\ell(r)).
\end{equation}
By \cite[Thm.~4.6]{hain-matsumoto:weighted}, there is a natural isomorphism
$$
H^1(\cG_1^{\wtd,\ell},S^m H_\ell(r)) \cong H^1(\pi_1(\M_{1,1/\Z[1/\ell]},\vv_o))
$$

Set $G_\ell = \pi_1(\Spec\Z[1/\ell],\Qbar)$. The Hochschild-Serre spectral
sequence of the extension
$$
1 \to \pi_1(\M_{1,1/\Qbar},\vv_o) \to \pi_1(\M_{1,1/\Z[1/\ell]},\vv_o)
\to G_\ell \to 1
$$
degenerates at $E_2$ as the kernel has virtual cohomological dimension 1 and
because the extension is split. This implies that
$$
H^1(\cG_1^{\wtd,\ell},S^m H_\ell(r)) \cong
\begin{cases}
H^1(G_\ell,\Ql(r)) & m = 0,\cr
H^0\big(G_\ell,\Het^1(\M_{1,1/\Qbar},S^m\H_\ell(r))\big) & m > 0.
\end{cases}
$$

When $m=0$, the image of (\ref{eqn:injection}) equals the image of
$$
\Hfte^1(G_\Q,\Ql(r)) \hookrightarrow
H^1(G_\ell,\Ql(r)) \cong H^1(\cG_1^{\wtd,\ell},\Ql(r)).
$$
The $m=0$ case now follows from Soul\'e's computation \cite{soule}.

Now suppose that $m>0$. Then we have an injection
$$
H^1(\cG_1^{\cris,\ell},S^m H_\ell(r)) \hookrightarrow
H^0(G_\ell,\Het^1(\M_{1,1/\Qbar},S^m\H_\ell(r))).
$$
The comparison theorem and Theorem~\ref{thm:eichler_shimura} imply that
$$
H^0(G_\ell,\Het^1(\M_{1,1/\Qbar},S^m\H_\ell(r))) \cong \Ql
$$
when $m=2n>0$ and $r=2n+1$, and is zero otherwise.

To complete the proof, we have to show that (\ref{eqn:injection}) is surjective
when $n>0$. To do this, consider the extension $\bE_{2n+1}'$ of
$\pi_1(\M_{1,1/\Q},\tate)$-modules corresponding to the sub-extension of the
extension (\ref{eqn:extn_p}) over $\M_{1,1}^\an$ with $m=2n+1$ that is an
extension of $\Ql$ by $S^{2n}\H_\ell(2n+1)$. The restriction of 
$S^{2n}\H_\ell(2n+1)$ to the base point $\tate$ splits as a sum
$$
S^{2n}H_\ell(2n+1) = \Ql(1) \oplus \Ql(2) \oplus \cdots \oplus \Ql(2n+1).
$$
So the fiber of $\bE_{2n+1}'$ over $\tate$ splits as a sum of extensions of
$\Ql$ by $\Ql(j)$, with $1\le j \le 2n+1$. Nakamura \cite[Thms.~3.3 \&
3.5]{nakamura} shows that all but the last of these is trivial, and the last one
is a non-zero multiple of the generator of
$\Hfte^1(G_\Q,\Ql(2n+1))$.\footnote{See also Remarks~\ref{rem:extensions1} and
\ref{rem:extensions2}.}  This implies that the restriction of the $\bE_{2n+1}'$
to $\tate$ is crystalline, so that (\ref{eqn:injection}) is surjective.
\end{proof}

The following technical computation will be needed in the next section.

\begin{corollary}
\label{cor:van_ext}
Suppose that $\ast \in \{1,\uu\}$.
If $e\ge 0$, then
$$
\Ext^1_{\MEM_\ast^\ell}(S^{2n}H_\ell(2n+1),S^{2m}H_\ell(2m-e)) = 0.
$$
\end{corollary}

\begin{proof}
In order for the ext group to be non-zero, the weight of
$S^{2n}H_\ell(2n+1)$ has to be greater than the weight of $S^{2m}
H_\ell(2m-e)$. That is, we have
$-2n-2 > -2m + 2e$ or, equivalently, $0\le n < m-e-1$. In particular, $n<m$.
Since $n<m$
\begin{multline*}
\Ext^1_{\MEM_\ast^\ell}(S^{2n}H_\ell(2n+1),S^{2m}H_\ell(2m-e))
\cong
\Ext^1_{\MEM_\ast^\ell}(\Ql,S^{2n}H_\ell\otimes S^{2m}H_\ell(2m-e-1)) \hfill
\cr
\cong
\bigoplus_{t=0}^{2n} \Ext^1(\Ql,S^{2n+2m-2t}H_\ell(2m-e-t-1)).
\end{multline*}
Since $t\le 2n$ and $e\ge 0$, we have $2m - e - t - 1 < 2n +2m -2t$. So none of
the summands on the right hand side can be non-zero. The result follows.
\end{proof}

\subsection{A bound on $H^2(\cG^{\cris,\ell}_\ast,S^m H_\ell(r))$}

In this section we prove the following analogue of the second statement of
Proposition~\ref{prop:deligne_coho}.

Recall from \cite[Thm.~1.2.4]{scholl} that
$\Het^1(\M_{1,\ast/\Qbar},S^m\H_\ell(r))$ is a crystalline $G_\Q$-module.

\begin{theorem}
Fix a prime number $\ell$. If $\ast \in \{1,\uu\}$ and $m\ge 0$ and $r$ are
integers, then $H^2(\cG_\ast^{\cris,\ell},S^{m}H_\ell(r))$ vanishes when $m$
is odd and when $r<m+2$. When $m=2n$, there is an injection
$$
H^2(\cG_\ast^{\cris,\ell},S^{2n}H_\ell(r)) \hookrightarrow
\Hfte^1(G_\Q, H^1(\M_{1,\ast/\Qbar},S^{2n}H_\ell(r))).
$$
More precisely, when $r\ge 2n+2$, there is an injection
\begin{multline*}
H^2(\cG_\ast^{\cris,\ell},S^{2n}H_\ell(r))
\cr
\hookrightarrow
\begin{cases}
\Ql & \ast = \uu,\ n=0, r \ge 2 \text{ even,} \cr
\Ql \oplus \Hfte^1(G_\Q,H^1_\cusp(\M_{1,\ast/\Qbar},S^{2n}\H_\ell(r)))&
n>0,\ r \text{ even},\cr
\Hfte^1(G_\Q,H^1_\cusp(\M_{1,\ast/\Qbar},S^{2n}\H_\ell(r))) &
 n>0,\ r \text{ odd}.
\end{cases}
\end{multline*}
\end{theorem} 

\begin{proof}
We will prove the theorem when $\ast = 1$. The proof of the case $\ast = \uu$ is
similar and is left to the reader.

We begin by establishing the vanishing of 
$H^2(\cG^{\cris,\ell}_1,S^{m}H_\ell(r))$ when $r<m+2$ and when $m$ is odd. There
are at least two ways to do this. Both exploit the vanishing of
$H^1(\cG_1^{\cris,\ell},S^{m}H_\ell(r))$ when $r<m+1$ or $m$ is odd. We will use
Lie algebra cohomology. The other approach uses Yoneda extensions.

Denote the Lie algebra of $\U^{\cris,\ell}_\ast$ by $\u^{\cris,\ell}_\ast$.
Since
$$
\Ext^j_{\MEM^\ell_1}(\Ql,S^m\H_\ell(r)) \cong
H^j(\u_1^{\cris,\ell},S^m H_\ell(r))^{\GL(H)}
$$
and since there is a natural isomorphism
$$
\Gr^W_\dot H^j(\u_1^{\cris,\ell},S^m H_\ell(r)) \cong
H^j(\Gr^W_\dot\u_1^{\cris,\ell},S^m H_\ell(r)),
$$
it follows that
$$
\Gr^W_\dot H_1(\u^{\cris,\ell}_1) \cong \bigoplus_{n>0} S^{2n}H_\ell (2n+1).
$$
Since $\Gr^W_\dot \u_1^{\cris,\ell}$ is generated as a Lie algebra in the
category of $\GL(H)$-modules by the image of any $\GL(H)$-invariant section
of $\Gr^W_\dot \u_1^{\cris,\ell} \to \Gr^W_\dot H_1(\u^{\cris,\ell}_1)$, 
and since
$$
S^a H \otimes S^b H \cong \bigoplus_{t\ge 0} S^{a+b-2t}H(-t)
$$
it follows that if $S^m H_\ell(r)$ appears in $\Gr^W_\dot\u_1^{\cris,\ell}$,
then $m$ is even and $r \ge m+1$. This implies that if $S^m H_\ell(r)$ occurs in
$\Lambda^2 \Gr^W_\dot \u_1^{\cris,\ell}$, then $m$ is even and $r \ge m+2$.
Consequently, the group of $\GL(H)$-invariant 2-cochains
$$
\Hom_{\GL(H)}(\Lambda^2 \Gr^W_\dot \u_1^{\cris,\ell},S^m H_\ell(r))
$$
is non-zero only when $m$ is even and $r\ge m+2$. It follows that
$H^2(\u_1^{\cris,\ell},S^m H_\ell(r))$ vanishes when $m$ is odd and when
$r<2m+2$.

There are natural homomorphisms
$$
H^2(\cG_1^{\cris,\ell},S^{2n}H_\ell(r))
\to H^2(\cG_1^{\wtd,\ell},S^{2n}H_\ell(r))
\to \Het^\dot(\M_{1,1/\Z[1/\ell]},S^m\H_\ell(r))
$$
General properties of weighted completion (cf.\
\cite[Thm.~4.6]{hain-matsumoto:weighted}) imply that the right-hand map is
injective. Our next task is to show that the second mapping is injective. We do
this using Yoneda extensions.

Suppose that the 2-extension
$$
0 \to S^{2n}H_\ell(r) \to E \to F \to \Ql \to 0
$$
in $\MEM_1^\ell$ represents a class in the kernel of
$$
H^2(\cG_1^{\cris,\ell},S^{2n}H_\ell(r)) \to
H^2(\cG_1^{\wtd,\ell},S^{2n}H_\ell(r))
$$
By Yoneda's criterion, there is a $\cG_1^{\wtd,\ell}$-module $V$ and a
commutative diagram
$$
\xymatrix{
0 \ar[r] & S^{2n}H_\ell(r) \ar[r] \ar@{>->}[d] & V \ar[d]^{\id_V}\ar[r] &
F \ar[r]\ar@{->>}[d] & 0 \cr
0 \ar[r] & E \ar[r] & V \ar[r] & \Ql \ar[r] & 0
}
$$
with exact rows in $\Rep(\cG_1^{\wtd,\ell})$. We view $V$ as a $G_\Q$-module
via the homomorphism
$$
G_\Q \to \pi_1(\M_{1,1/\Q},\tate) \to \cG^\wtd_1(\Ql) \to \Aut V
$$
where the first homomorphism is the section induced by the base point $\tate$.
To prove injectivity, it suffices to show that $V$ is a crystalline
$G_\Q$-module as the fact that it is a $\cG_\ell^\wtd$-module implies that it is
unramified away from $\ell$.

The action of $\cG_\ell^\wtd$ on $V$ determines compatible weight filtrations
$M_\dot$ on $E$, $F$ and $V$.

Let $B$ be the kernel of $F \to \Ql$. By the pruning lemma below
(Lemma~\ref{lem:pruning}), we may assume that each $W_\dot$ graded
quotient of $B$ is a sum of copies of $S^{m}H_\ell(r)$ where $r > m$ or,
equivalently, that $B = M_{-2}B$. In this case $B$ is an extension
$$
0 \to M_{-4}B \to B \to \Ql(1)^d \to 0
$$
for some $d\ge 0$. This implies that $E$ is also an extension
$$
0 \to M_{-4}E \to E \to \Ql(1)^d \to 0.
$$
The pushout of the extension
$$
0 \to E \to V \to \Ql \to 0
$$
along $E \to \Ql(1)^d$ is the extension
$$
0 \to \Ql(1)^d \to F/M_{-4}F \to \Ql \to 0.
$$
Since $F$ is crystalline, this is an extension of crystalline $G_\Q$-modules.
Proposition~11.3 of \cite{hain-matsumoto:ladic} implies that $V$ is also a
crystalline $G_\Q$-module and therefore a $\cG_1^{\cris,\ell}$ module. This
completes the proof of injectivity.

The final task is to show that the image of
\begin{equation}
\label{eqn:cris}
H^2(\cG_1^{\cris,\ell},S^{2n}H_\ell(r)) \to
H^1(G_\ell,\Het^1(\M_{1,1/\Qbar},S^{2n}\H_\ell(r)))
\end{equation}
is contained in $\Hfte^1(G_\Q,\Het^1(\M_{1,1/\Qbar},S^{2n}\H_\ell(r)))$.
When $n=0$, this is clear.

To prove this when $n>0$, we use the fact that $\cG^\rel_1\otimes\Ql$ is a
crystalline representation of $G_\Q$, or more accurately, that its coordinate
ring $\O(\cG^\rel_1)\otimes\Ql$ is a crystalline $G_\Q$-module. (That it is
unramified away from $\ell$ is proved in \cite{hain-matsumoto:ladic}, the full
statement is proved in \cite{olsson}.)

Let $\A_\ell^\Mod$ to be the Zariski closure of the image of $G_\Q$ in
$\Aut(\cG^\rel_{1/\Ql})$.\footnote{This group should be a quotient of the $\Ql$
form of the fundamental group of Francis Brown's category of {\em mixed modular
motives} over $\Z$, perhaps even isomorphic to it.} Representations of $G_\Q$
that factor through $G_\Q \to \A_\ell^\Mod(\Ql)$ are crystalline. Denote by
$\cC^{\cris,\ell}_1$ the tannakian category of finite dimensional
representations of $G_\Q \ltimes \cG^\rel_{1/\Ql}$. Its tannakian fundamental
group (with respect the underlying $\Ql$ vector space) is a split extension
$$
1 \to \cG^\rel_{1/\Ql} \to \pi_1(\cC^\cris_1) \to \A_\ell^\Mod \to 1.
$$
The splitting is induced by the base point $\tate$. There is a commutative
diagram
\begin{equation}
\label{eqn:diag}
\xymatrix{
G_\Q \ltimes \cG^\rel_{1/\Ql} \ar[r]\ar[d] &
\A_\ell^\Mod\ltimes \cG^\rel_{1/\Ql} \ar[d] \cr
\cG^{\wtd,\ell}_1 \ar[r] & \cG^{\cris,\ell}_1
}
\end{equation}

Now suppose that $n>0$. Since $H^2(\cG^\rel_1,S^{2n}H(r))$ vanishes for all $n$
and $r$, since $H^0(\cG^\rel_{1/\Ql},S^{2n}H_\ell(r)) = 0$ when $n>0$, and
since there are $G_\Q$-module isomorphisms
$$
H^1(\cG_1^\rel,S^{2n}H_\ell(r)) \cong
[H^1(\u^\rel_1)\otimes S^{2n}H_\ell(r)]^{\SL(H)}
\cong \Het^1(\M_{1,1/\Qbar},S^{2n}\H_\ell(r))
$$
there are isomorphisms
$$
H^2(G_\Q \ltimes \cG^\rel_{1/\Ql},S^{2n}H_\ell(r))
\cong H^1\big(G_\Q,\Het^1(\M_{1,1/\Qbar},S^{2n}\H_\ell(r))\big)
$$
and
\begin{multline*}
H^2(\cG^\rel_{1/\Ql},S^{2n}H_\ell(r)) \cong
\Hfte^1(G_\Q,\Het^1(\M_{1,1/\Qbar},S^{2n}\H_\ell(r))) \cr
\cong
\Hfte^1(\M_{1,1/\Q},S^{2n}\H_\ell(r)).
\end{multline*}
Applying $H^2$ to the diagram (\ref{eqn:diag}) implies that the image of
(\ref{eqn:cris}) is contained in
$\Hfte^1(G_\Q,\Het^1(\M_{1,1/\Qbar},S^{2n}\H_\ell(r)))$.
\end{proof}

Our final task is to prove the Pruning Lemma.

\begin{lemma}[Pruning Lemma]
\label{lem:pruning}
Suppose that $\ast \in \{1,\uu\}$. Let $\cG$ be either $\cG^{\wtd,\ell}_\ast$ or
$\cG^{\cris,\ell}_\ast$. If
\begin{equation}
\label{eqn:big_extn}
0 \to B \to E \to \Ql \to 0
\end{equation}
is an extension of $\cG$-modules, then there is a $\cG$-sub module $A$ of $B$
such that the extension (\ref{eqn:big_extn}) is pulled back from an extension
$$
0 \to A \to F \to \Ql \to 0
$$
and $A = M_{-2}A$.
\end{lemma}

\begin{proof}
All modules in this proof will be $\cG$-modules and all extension groups will be
extensions of $\cG$-modules. First recall that if $C$ is a sub-module of $B$ and
if $\Ext^1(\Ql, B/C)=0$, then $\Ext^1(\Ql,C)\to \Ext^1(\Ql,B)$ is surjective.

Note that $S^m H_\ell(r) = M_{-2}S^m H_\ell(r)$ if and only if $r > m$.
So the condition that $A=M_{-2}A$ is equivalent to the condition that
each copy of $S^m H_\ell(r)$ in $\Gr^W_\dot A$ satisfies $r>m$.

If $M_{-2}B \neq B$, choose the largest integer $w$ such that $\Gr^W_w B$
contains a copy of $S^m H_\ell(r)$ with $r\le m$. Then decompose $\Gr^W_w B$ 
$$
\Gr^W_w B = D^+ \oplus D^- 
$$
where $D^+$ is the sum of factors $S^m H_\ell(r)$ with $r>m$ and $D^-$ is the
sum of the factors with $r\le m$.

Corollary~\ref{cor:van_ext} implies that the extension
$$
0 \to D^- \to B/W_{w-1} \to C \to 0
$$
splits. There is therefore a projection $B \to D^-$ and an extension
$$
0 \to K \to B \to D^- \to 0.
$$
Corollary~\ref{cor:van_ext} implies that $\Ext^1(\Ql,D^-)$ vanishes, which
implies that $\Ext^1(\Ql,K) \to \Ext^1(\Ql,B)$ is surjective. Repeat this
procedure until $M_{-2}K = K$.
\end{proof}

\section{The Elliptic Polylogarithm}

Recall from Corollary~\ref{cor:elliptic_polylog} that there is an object $\bp$
of $\MEM_2$ whose corresponding local system has fiber the Lie algebra
$\p(E',x)$ of $\pi_1^\un(E',x)$ over the point $[E,x]$ of $\M_{1,2}^\an$. Denote
by $\bp'$ the sub-local system whose fiber over $[E,x]$ is the commutator
subalgebra $[\p(E',x),\p(E',x)]$. It is an object of $\MEM_2$.

The elliptic polylogarithm of Beilinson and Levin is the object
$\Pol^\elliptic_2 := \p/[\p',\p']$ of $\MEM_2$,
\cite[Prop.~1.4.3]{beilinson-levin}. It is an extension
\begin{equation}
\label{eqn:poly_log}
0 \to \big(\Sym \H(1)\big)(1) \to \Pol^\elliptic_2 \to \H(1) \to 0
\end{equation}
and therefore an element of
\begin{align}
\label{eqn:ext_decomp}
&\phantom{\cong} \bigoplus_{m\ge 0} \Ext^1_{\MEM_2}(\H(1),S^m\H(m+1))\cr
&\cong \bigoplus_{m\ge 0} \Ext^1_{\MEM_2}(\Q,\H\otimes S^m\H(m+1))\cr
&\cong \bigoplus_{m\ge 1} \Ext^1_{\MEM_2}(\Q,S^{m-1}\H(m))
\oplus \bigoplus_{m\ge 0}\Ext^1_{\MEM_2}(\Q,S^{m+1}\H(m+1)).
\end{align}

By Proposition~\ref{prop:restriction}, $\bp$, $\bp'$ and $\Pol^\elliptic_2$
restrict to objects of $\MEM_\uu$, which we will also denote by $\bp_\uu$,
$\bp_\uu'$ and $\Pol^\elliptic_\uu$. The fiber of $\bp_\uu$ over the point
$[E,\vv]$ of $\M_{1,\uu}^\an$ is $\p(E',\vv)$.

The monodromy of the local system $\Pol^\elliptic_\uu$ over $\M_{1,\uu}$ about
the fiber of the $\Gm$-bundle $\M_{1,\uu} \to \M_{1,1}$ is the unipotent
automorphism of $\p(E',\vv)$ induced by the element of $\pi_1(E',\vv)$ that
rotates the tangent vector once about the identity. Its logarithm spans a copy
of $\Q(1)$ in $\Der \p(E',\vv)$ and induces a map 
$$
\Pol^\elliptic_\uu(1) \to \Pol^\elliptic_\uu/(\Sym \H(1))(2)
\overset{\simeq}{\longrightarrow} \H(2) \hookrightarrow (\Sym \H(1))(1).
$$
The coinvariants $\Pol^\elliptic_\uu/\H(2)$ of this action descends to a local
system on  $\M_{1,1}$, which we denote by $\Pol^\elliptic_1$. It is an object of
$\MEM_1$.

The computations in the previous sections imply that some of the extensions in
the decomposition (\ref{eqn:ext_decomp}) are trivial in $\MHS(\M_{1,2},\H)$. The
next result shows that all extensions in $\Pol^\elliptic_\ast$ that can be
non-zero are indeed non-trivial.

\begin{proposition}
\label{prop:non-triviality}
The elements of $\Ext^1_{\MHS(\M_{1,\ast},\H)}(\Q,S^m\H(m+1))$ that occur in the
decomposition (\ref{eqn:ext_decomp}) and its analogues for $\Pol^\elliptic_1$
and $\Pol^\elliptic_\uu$, are non-trivial when
\begin{enumerate}

\item $m=1$ and $m = 2n > 0$, when $\ast = 2$,

\item $m=2n \ge 0$, when $\ast = \uu$,

\item $m=2n > 0$, when $\ast = 1$.

\end{enumerate}
\end{proposition}

\begin{proof}
The vanishing follows from Theorem~\ref{thm:ext1}. There are several ways to
prove the non-vanishing. Here will appeal to the result
\cite[Lem.~1.5.3]{beilinson-levin} of Beilinson--Levin, which implies that the
fiber of the kernel of the monodromy logarithm
$$
\log\sigma_o : \Pol^\elliptic_2 \to \Pol^\elliptic_2
$$
over the tangent vector $\vv_o$, where $\sigma_o$ is the generator of the
fundamental group of the punctured $q$-disk, is the fiber over $\ww_o =
\vec{10}$ of the classical polylogarithm local system $\Pol^\class$ over
$\Pminus$. This is the sub-extension
$$
0 \to \bigoplus_{m\ge 1} \Q(m) \to \mathrm{\Pol}_{\ww_o}^\class \to \Q(1) \to 0
$$
of the extension (\ref{eqn:poly_log}) obtained by replacing each copy of $S^m
H(m+1)$ by the space $\Q(m+1)$ of its lowest weight vectors. This is an
extension of $\Q(1)$ by $\Q(m+1)$. It is well known that the fiber of the
polylog variation of MHS over $\vv_o$ is given by the zeta values
$$
\zeta(m) \in \Ext^1_\MHS(\Q,\Q(m)) \cong \Ext^1_\MHS(\Q(1),\Q(m+1)).
$$
This has infinite order when $m\ge 3$ is odd, which implies that the $(2n+1)$st
sub extension of $\Pol_2^\elliptic$ (the ``($2n+1$)st elliptic polylog'')
$$
0 \to S^{2n+1}\H(2n+2) \to \bE_{2n+1} \to \H(1) \to 0
$$
is non-trivial in
$$
\Ext^1(\H(1),S^{2n+1}\H(2n+2)) \cong \Ext^1(\Q,S^{2n+2}\H(2n+2))
\oplus \Ext^1(\Q,S^{2n}\H(2n+1)).
$$
Since the first summand is trivial, this implies that the second component of
$\bE_{2n+1}$ is non-trivial.
\end{proof}

\begin{remark}
\label{rem:extensions1}
The non-triviality of these extensions also follows from the explicit
computation \cite[Prop.~18.3]{hain:kzb} of the limit MHS on $\p =
\p(E'_\tate,\ww_o)$. Alternatively, one can compute the period of the limit MHS
on the fiber of $\Pol^\elliptic_\uu$ over $\vv_o$ directly. This was done by
Brown in \cite[Lem.~7.1]{brown:mmv}.

Finally, one can prove the non-triviality using $\ell$-adic methods. Since $\p$
is a mixed elliptic motive, one can check non-triviality of an extension in any
realization. Nakamura \cite[Thms.~3.3 \& 3.5]{nakamura} computed the $\ell$-adic
Galois action on $\p$ and showed that the extension above is given by a non-zero
rational multiple of the Soul\'e character.
\end{remark}

\section{Simple Extensions in $\MEM_\ast$}
\label{sec:exts}

We can now assemble the results of this part to give a computation of the
extension groups of $\Q$ by $S^m\H(r)$ in $\MEM_\ast$.

\begin{theorem}
\label{thm:exts}
When $\ast \in \{1,\uu\}$, we have
$$
\Ext^1_{\MEM_\ast}(\Q,S^m\H(r)) =
\begin{cases}
\Q & \ast = 1,\ m = 0 \text{ and $r\ge 3$ odd,} \cr
\Q & \ast = \uu,\ m = 0 \text{ and $r\ge 1$ odd,}\cr
\Q & m=2n > 0 \text{ and } r=2n+1,\cr
0 & \text{otherwise.}
\end{cases}
$$
When $\ast = 2$, we have
$$
\Ext^1_{\MEM_2}(\Q,S^m\H(r)) =
\begin{cases}
\Q & m = 0 \text{ and $r\ge 3$ odd,} \cr
\Q & m = 1 \text{ and } r=1,\cr
\Q & m=2n > 0 \text{ and } r=2n+1,\cr
0 & \text{otherwise.}
\end{cases}
$$
The natural maps
$$
\Ext^1_{\MEM_1}(\Q,S^m\H(r)) \to \Ext^1_{\MEM_\uu}(\Q,S^m\H(r)) \to \Ext^1_{\MEM_2}(\Q,S^m\H(r))
$$
are both injective for all $m$ and $r$.
\end{theorem}

When $m=2n$ and $r=2n+1$, the group $\Ext^1_{\MEM_\ast}(\Q,S^m\H(r))$ is generated by the $2n$th elliptic polylogarithmic extension, which corresponds to the Eisenstein series of weight $2n+2$. When $m=0$ and $r\ge 3$ is odd, it is generated by the $r$th zeta element; when $\ast=\uu$ and $r=1$, it is generated by an invertible function associated to the Eisenstein series of weight 2, which can be regarded as the discriminant function on $\M_{1,\uu}$. When $\ast=2$ and $m=1$, it is generated by the extension that corresponds to the tautological section of the universal elliptic curve over $\M_{1,2}$.

\begin{proof}
When $m=0$ and $r\ge 3$, the result follows from the fact that the restriction mapping $\Ext^1_{\MEM_\ast}(\Q,\Q(r)) \to \Ext^1_{\MTM_\ast}(\Q,\Q(r))$ is an
isomorphism. When $m=0$ and $r=1$, the result follows from the fact that $\O(\M_{1,\ast})^\times$ is trivial when $\ast = 1,2$ and generated by the discriminant when $\ast = \uu$.

Suppose that $m>0$. Since the functor $\MEM_\ast \to \MHS(\M_{1,\ast},\H)$ is
fully faithful (Theorem~\ref{thm:surjective}), the induced homomorphisms
\begin{equation}
\label{eqn:regulator}
\Ext^1_{\MEM_\ast}(\Q,S^m\H(r)) \to \Ext^1_{\MHS(\M_{1,\ast},\H)}(\Q,S^m\H(r))
\end{equation}
are injective. So Theorem~\ref{thm:ext1} shows that the extension groups are no
bigger than claimed in the statement.

When $m>0$, surjectivity follows from the existence of the elliptic
polylogarithms in $\MEM_\ast$. (Proposition~\ref{prop:non-triviality}.) The
$\Q$-de Rham realization of the generator of
$$
\Ext^1_{\MEM_\ast}\big(\Q,S^{2n}\H(2n+1)\big),\quad n>0
$$
can be written down explicitly using Proposition~\ref{prop:nabla_0} and
Remark~\ref{rem:scalings}, or deduced from \cite[Thm.~20.2]{hain:kzb}.
\end{proof}

\begin{remark}
\label{rem:extensions2}
Since the fiber $H$ of $\H$ over $\tate$ is $\Q(0) \oplus \Q(-1)$, the fiber $V$
over $\tate$ of a generator of $\Ext^1_{\MHS(\M_{1,1})}(\Q,S^{2n}\H(2n+1))$ is
an extension of $\Q$ by
$$
\Q(1) \oplus \Q(2) \oplus \dots \oplus \Q(2n+1).
$$
and thus determines elements
$$
\lambda_j \in \Ext^1_\MHS(\Q,\Q(j)), \quad 1 \le j \le 2n+1.
$$
It is easy to see that $\lambda_j = 0$ when $j<2n+1$; the proof is in the  next
paragraph. The proof of the theorem implies that the class of $\lambda_{2n+1}$
is a non-zero rational multiple of the class of $\zeta(2n+1)$. And indeed, this
is exactly what one sees in the computations of Brown \cite{brown:mmv} and
Nakamura \cite[Thms.~3.3 \& 3.5]{nakamura}.

We know that $N := \log \sigma_o$ acts on the limit $V$ as a morphism of type
$(-1,-1)$. Since the residue of the form $\psi_{2n}$ at $q=0$ is non-zero, the
map
$$
N : \Gr^W_0 V \to \Gr^W_{-1} V
$$
is non-zero and therefore injective. Suppose that $1\le j \le 2n$. Since
$j<2n+1$, $N : S^{2n} H \to S^{2n} H$ maps $\Q(j)$ onto $\Q(j+1)$. It thus maps
the extension $\lambda_j$ of $\Q$ by $\Q(j)$ isomorphically into an extension of
$\Q(1)$ by $\Q(j+1)$ that lies inside $S^{2n}H(2n+1)$. But all such extensions
are split, so $\lambda_j = 0$. 
\end{remark}

\section{Structure Theorems}
\label{sec:structure}

The results of the preceding sections allow us to determine how the fundamental
groups of the various categories $\MEM_\ast$ are related.

Let $\w$ be any fiber functor that is the composition of the fiber functor
$\MEM_\ast \to \MTM$ with any fiber functor $\MTM \to \Vec_F$. Denote the
unipotent completion of $\pi_1(E_\tate',\ww_o)$ by $\cP$. The inclusion
$E_\tate'\hookrightarrow \M_{1,2}$ induces a restriction functor $\MEM_2 \to
\Rep(\cP)$ and therefore a homomorphism $\cP \to \pi_1(\MEM_2,\w)$.

The inclusion of a formal punctured disk $\D' \hookrightarrow E_\tate'$ induces
a homomorphism $\Q(1)= \pi_1^\un(\D',\ww_o)  \to \cP$. We will denote this
subgroup of $\cP$ by $\Ga(1)$. It is also the unipotent fundamental group of the
fiber of $\M_{1,\uu} \to \M_{1,1}$. As above, restriction to the fiber induces a
homomorphism $\Ga(1) \to \pi_1(\MEM_\uu,\w)$ such that the diagram
$$
\xymatrix{
\Ga(1) \ar[r]\ar[d] & \pi_1(\MEM_\uu,\w) \ar[d] \cr
\cP \ar[r] & \pi_1(\MEM_2,\w)
}
$$
commutes.

\begin{theorem}
\label{thm:structure}
There are exact sequences
$$
1 \to \cP \to \pi_1(\MEM_2,\w) \to \pi_1(\MEM_1,\w) \to 1
$$
and 
$$
1 \to \Ga(1) \to \pi_1(\MEM_\uu,\w) \to \pi_1(\MEM_1,\w) \to 1.
$$
\end{theorem}

\begin{proof}
It suffices to prove the result when $\w=\w^B$. We first show that $\cP$ is a
normal subgroup of $\pi_1(\MEM_2,\w^B)$. Since the Lie algebra of $\cP$ is an
object of $\MEM_2$, there is a natural action of $\pi_1(\MEM_2,\w^B)$ on $\cP$.
The composite
$$
\cP \to \pi_1(\MEM_2,\w^B) \to \Aut\cP
$$
is the inner action. It is injective as $\cP$ has trivial center, which shows
that $\cP$ is a subgroup of $\pi_1(\MEM_2,\w^B)$.

To see that it is normal, consider the diagram
$$
\xymatrix{
\cP \ar[r]\ar@{=}[d] & \cGhat_2 \ar[d] \cr
\cP \ar[r] & \pi_1(\MEM_2,\w^B)
}
$$
Since the right-hand vertical map is surjective, and since $\cP$ is a
normal subgroup of $\cGhat_2$, it follows that $\cP$ is a normal subgroup of
$\pi_1(\MEM_2,\w^B)$.

We can thus consider the quotient group $\pi_1(\MEM_2,\w^B)/\cP$. Its
representations are objects of $\MEM_2$ that have trivial monodromy on each
fiber of $\M_{1,2}\to \M_{1,1}$, and are thus constant on each fiber. This
implies that its representations are precisely the pullbacks of objects of
$\MEM_1$ along the projection. This implies that $\pi_1(\MEM_2,\w^B)/\cP$ is
isomorphic to $\pi_1(\MEM_1,\w^B)$.

That $\Ga(1)$ is a normal subgroup of $\pi_1(\MEM_\uu,\w)$ follows from the
commutativity of the diagram
$$
\xymatrix{
\Ga(1) \ar[r]\ar[d] & \cGhat_\uu \ar[r]\ar[d] & \pi_1(\MEM_\uu,\w^B)\ar[d]
\cr
\cP \ar[r] & \cGhat_2 \ar[r] & \pi_1(\MEM_2,\w^B)
}
$$
As in the $\ast = 2$, case, $\pi_1(\MEM_\uu,\w^B)/\Ga(1) \cong
\pi_1(\MEM_1,\w^B)$.
\end{proof}

\begin{corollary}
\label{cor:structure}
There are natural isomorphisms
$$
\pi_1(\MEM_\uu,\w) \cong \pi_1(\MEM_1,\w)\ltimes \Ga(1)
$$
and
$$
\pi_1(\MEM_2,\w) \cong \pi_1(\MEM_\uu,\w) \ltimes_{\Ga(1)} \cP.
$$
\end{corollary}

\begin{proof}
Since $\Ext^1_\MTM(\Q,\Q(1))=0$, the isomorphism  $\Ext^1_{\MEM_\uu}(\Q,\Q(1))
\cong \Q$ of Theorem~\ref{thm:exts} restricts to an isomorphism
$$
H^1(\pi_1^\geom(\MEM_\uu,\w),\Q(1)) \cong \Q.
$$
The composition of the corresponding homomorphism $\pi_1^\geom(\MEM_\uu,\w) \to
\Q(1)$ with the inclusion $\Ga(1) \to \pi_1^\geom(\MEM_\uu,\w)$ is easily seen
to be an isomorphism. There is therefore an isomorphism
$$
\pi_1^\geom(\MEM_\uu,\w) \cong
\big[\pi_1^\geom(\MEM_\uu,\w)/\Ga(1)\big]\times \Ga(1)
\cong \pi_1^\geom(\MEM_1,\w)\times \Ga(1).
$$
Since all of these isomorphisms are respected by the $\pi_1(\MTM,\w)$ action,
there is a natural isomorphism
$$
\pi_1(\MEM_\uu,\w) \cong
\pi_1(\MTM,\w)\ltimes \big(\pi_1^\geom(\MEM_1,\w)\times \Ga(1)\big)
\cong \pi_1(\MEM_1)\ltimes \Ga(1).
$$

The second isomorphism follows from Theorem~\ref{thm:structure} and the
first assertion. The splitting is induced by the homomorphism
$\pi_1(\MEM_\uu,\w) \to \pi_1(\MEM_2,\w)$:
$$
\xymatrix{
&&& \pi_1(\MEM_\uu,\w) \ar[d]\ar[dl] \cr
1 \ar[r] & \cP \ar[r] & \pi_1(\MEM_2,\w) \ar[r] &
\pi_1(\MEM_1,\w) \ar[r]\ar@/_/[u] & 1.
}
$$
\end{proof}

\section{Motivic Remarks}

Suppose that $\ast \in \{1,\uu\}$. We expect that the ext groups of the category
$\MEM_\ast$ to be faithfully represented in the various categories of
realizations. To be more precise, set $\MHS_\ast := \MHS(\M_{1,\ast},\H)$
and $\MEM_\ast^\ell = \Rep(\cG^{\cris,\ell}_\ast)$. The realization functors
$$
\real_\MHS : \MEM_\ast \to \MHS_\ast
\text{ and }
\real_\ell : \MEM_\ast \to \MEM_\ast^\ell
$$
induce ``regulator mappings''
$$
\reg_\Q: \Ext^\dot_{\MEM_\ast}(\Q,S^m\H(r)) \to
\Ext_{\MHS_\ast}^\dot(\Q,S^m\H(r))^{\Frbar_\infty}
\cong H_\cD^\dot(\M_{1,\ast}^\an,S^m\H_\Q(r))^{\Frbar_\infty}
$$
and
$$
\reg_\ell: \Ext^\dot_{\MEM_\ast}(\Q,S^m\H(r)) \to
\Ext_{\MEM_\ast^\ell}^\dot(\Ql,S^m\H_\ell(r)) \cong
H^\dot(\cG^{\cris,\ell}_\ast,S^mH_\ell(r))
$$
that are compatible with products. We have already seen that, in degree 1, the
first is an isomorphism and the second is an isomorphism after tensoring with
$\Ql$. We can also map to real Deligne cohomology
$$
\reg_\R : \Ext^\dot_{\MEM_\ast}(\Q,S^m\H(r)) \to
\Ext_{\MHS_\ast}^\dot(\R,S^m\H_\R(r))^{\Frbar_\infty}
\cong H_\cD^\dot(\M_{1,\ast}^\an,S^m\H_\R(r))^{\Frbar_\infty}.
$$
In degree 1, this map is an isomorphism after tensoring with $\R$. The story is
more interesting in degree 2, where we have regulator mappings
\begin{multline}
\label{eqn:reg_R}
\reg_\R : \Ext^2_{\MEM_\ast}(\Q,S^m\H(r))
\to H^2_\cD(\M_{1,\ast}^\an,S^m\H_\R(r))^{\Frbar_\infty}
\cr
\cong \Ext^1_\MHS(\R,H^1(\M_{1,\ast}^\an,S^m\H(r)))^{\Frbar_\infty}
\end{multline}
and 
\begin{multline}
\label{eqn:reg_ell}
\reg_\ell: \Ext^2_{\MEM_\ast}(\Q,S^m\H(r))
\to H^2(\cG_\ast^{\cris,\ell},S^m H_\ell(r))
\cr
\cong \Hfte^1(G_\Q,\Het^1(\M_{1,\ast/\Qbar},S^m\H_\ell(r))).
\end{multline}

\begin{conjecture}
\label{conj:h2}
For all $m\ge 0$ and all $r\in \Z$:
\begin{enumerate}

\item\label{item:hodge}
The degree $2$ regulator mapping (\ref{eqn:reg_R}) is an isomorphism after
tensoring $\Ext^2_{\MEM_\ast}(\Q,S^m\H(r))$ with $\R$.

\item\label{item:galois}
For each prime number $\ell$, the degree $2$ regulator mapping
(\ref{eqn:reg_ell}) is an isomorphism after tensoring
$\Ext^2_{\MEM_\ast}(\Q,S^m\H(r))$ with $\Ql$.

\end{enumerate}
\end{conjecture}

The first conjecture is a analogue of Beilinson's conjecture
\cite[Conj.~8.4.1]{beilinson:hodge_coho} and his result for weight 2 modular
forms \cite{beilinson:modular}; the second is an analogue of a version of the
Bloch--Kato conjecture (cf.\ \cite[4.2.2]{fontaine}).

In Section~\ref{sec:pollack_motivic} we show that (\ref{item:hodge}) holds if
and only if the cup product
\begin{multline*}
\sum_{a+b=n}
\Ext^1_{\MEM_\ast}(\Q,S^{2a}\H(2a+1)) \otimes
\Ext^1_{\MEM_\ast}(\Q,S^{2b}\H(2b+1))
\cr
\to
\bigoplus_{r=0}^b \Ext^2_{\MEM_\ast}(\Q,S^{2n-2r}\H(2n+2-r))
\end{multline*}
is surjective.

\part{Towards a Presentation of $\pi_1(\MEM_\ast)$}
\label{part:presentation}

The goal of this part is to find a presentation of $\pi_1(\MEM_\ast,\w^\DR)$, or
more accurately, a presentation of the Lie algebra $\u^\MEM_\ast$ of its
prounipotent radical. While we do not achieve this goal, we are able to come
close, thanks to computations of Brown \cite{brown:mmv} and Pollack
\cite{pollack}. Specifically, we are able to determine the ``quadratic heads''
of all relations in $\u^\MEM_\ast$. If the variant
Conjecture~\ref{conj:h2}(\ref{item:hodge}) of standard conjecture holds, then
the quadratic head of every non-trivial minimal relation is non-trivial.

In Section~\ref{sec:splitting_DR}, we show that the de~Rham realization of
every object of $\MEM_\ast$ has a canonical bigrading which splits the Hodge
filtration and both weight filtrations. This is an important tool. It means that
$\u^\MEM_\ast$ is the completion of a bigraded Lie algebra in the category of
$\sl(H)$-modules.

\section{Presentations of Graded Lie algebras}
\label{sec:presentations}

Before attempting the problem of finding a presentation of
$\pi_1(\MEM_\ast,\w)$, we first consider a more abstract setting. Suppose that
$\cG$ is an affine group scheme over a field of characteristic zero that is an
extension
$$
1 \to \U \to \cG \to S \to 1
$$
of a connected, reductive group $S$ by a prounipotent group $\U$. Suppose that
the Lie algebra $\g$ of $\cG$ is graded:
$$
\g = \prod_{n\le 0} \g_n
$$
where $\g_0$ is isomorphic to the Lie algebra $\s$ of $S$ and
$$
\u = \prod_{n<0} \g_n
$$
is the Lie algebra of $\U$. For simplicity, we will suppose that each $\g_n$  is
finite dimensional. The splitting $\g = \s \ltimes \u$ gives a Levi
decomposition $\cG \cong S\ltimes \U$ of $\cG$. To give a presentation of $\cG$,
it suffices to give a presentation of the associated graded Lie  algebra
$\u_\dot$ of $\u$ in the category of $S$-modules.

The first step in finding a minimal such presentation of $\u_\dot$ is to choose
an $S$-invariant splitting of the graded surjection
$$
\phi : \u_\dot \to H_1(\u_\dot).
$$
This induces a graded $S$-equivariant surjection $\L(H_1(\u_\dot)) \to \u_\dot$
that induces the identity on $H_1$.

Set $\f = \L(H_1(\u_\dot))$. The ideal of relations $\r := \ker \phi$ is a graded
submodule of the commutator subalgebra $[\f,\f]$. The Lie algebra version of
Hopf's theorem implies that, since $\phi$ induces an isomorphism on $H_1$, there
is a natural isomorphism
$$
H_2(\u_\dot) \cong \r/[\r,\f]
$$
of $S$-modules. The image of any $S$-invariant section $\psi : H_2(\u_\dot) \hookrightarrow \r$ of the canonical surjection $\r \to H_2(\u_\dot)$ will be a minimal set of generators of $\r$. That is, $\r$ is the ideal $(\im \psi)$ generated by $\im \psi$ and there are isomorphisms
$$
\im\psi \overset{\simeq}{\longrightarrow} \r/[\r,\f] = H_2(\u_\dot).
$$
So, loosely speaking, every pronilpotent Lie algebra in $\Rep_F(R)$ has a
minimal presentation of the form
$$
\u \cong \L(H_1(\u))^\wedge/\big(H_2(\u)\big).
$$
Of course, one has to determine the mapping $\psi$. Its quadratic part is easy to describe.

Denote the lower central series of a Lie algebra $\h$ by
$$
\h = L^1\h \supseteq L^2\h \supseteq L^3 \h \supseteq \cdots
$$
The bracket $\f\otimes \f \to \f$ induces an isomorphism
$$
\Lambda^2 H_1(\u_\dot) \overset{\sim}{\longrightarrow} L^2\f/L^3\f.
$$
Examining the Chevalley-Eilenberg cochain complex of $\u_\dot$, one sees
immediately that:

\begin{lemma}
\label{lem:quad_relns}
The composition
$$
\xymatrix{
H_2(\u_\dot) \ar[r]^\psi & L^2 \f \ar[r] & L^2\f/L^3\f \ar[r]^\simeq &
\Lambda^2 H_1(\u_\dot)
}
$$
is the dual of the cup product $\Lambda^2 H^1(\u_\dot) \to H^2(\u_\dot)$.
Equivalently, there is an exact sequence
$$
\xymatrix{
H_2(\u_\dot) \ar[r]^(0.4){\Delta} & \Lambda^2 H_1(\u_\dot)
\ar[r]^(.45){[\blank,\blank]} &
L^2\u_\dot/L^3\u_\dot \ar[r] & 0
}
$$
where $\Delta$ denotes the coproduct; i.e., the graded dual of the cup product.
\end{lemma}

\begin{definition}
\label{def:quad_part}
The {\em quadratic leading term} (or part) of a relation
$r\in \r$ is its image in $L^2\u_\dot/L^3\u_\dot$.
\end{definition}

The final observation is that the canonical isomorphism
$$
H^\dot(\cG,V) \cong H^\dot(\u,V)^S
$$
implies that (provided that every irreducible $S$-module is absolutely
irreducible) there is a natural isomorphism
\begin{equation}
\label{eqn:homology}
H_\dot(\u_\dot) \cong
\bigoplus_{\alpha\in \Sdual} H^\dot(\cG,V_\alpha)^\vee\otimes V_\alpha,
\end{equation}
where $\Sdual$ denotes the set of isomorphism classes of irreducible $S$-modules, $(\blank)^\vee$ denotes dual, and $\{V_\alpha\}$ an $S$-module in the class $\alpha$.

If we fix a Cartan subalgebra $\h$ of $\s$ and a set of positive roots, we can
decompose $\s$ as
$$
\s = \n_+ \oplus \h \oplus \n_-
$$
where $\n_+$ (resp.\ $\n_-$) is the nilpotent subalgebra consisting of
positive (resp.\ negative) root vectors. So we can think of $\u_\dot$ as
being generated as an $\n_-$ module by its highest weight vectors. Applying
this to the isomorphism (\ref{eqn:homology}) gives the following result.

\begin{proposition}
\label{prop:gens+relns}
The space of highest weight vectors of the generating space of a minimal
presentation of $\u_\dot$ projects isomorphically to
$$
H_0\big(\n_-,H_1(\u_\dot)\big) \cong
H_1(\u_\dot)/\n_-H_1(\u_\dot) \cong
\bigoplus_{\alpha \in \Sdual} H^1(\cG,V_\alpha)^\vee.
$$
The space of highest weight vectors in a minimal space $\r/[\r,\f]$ of relations
projects isomorphically to
$$
H_0(\n_-,\r/[\r,\f]) \cong H_0(\n_-,H_2(\u)) \cong H_2(\u)/\n_-H_2(\u_\dot)
\cong \bigoplus_{\alpha \in \Sdual} H^2(\cG,V_\alpha)^\vee.
$$
\end{proposition}

Combined with Lemma~\ref{lem:quad_relns}, we see that the space $\r/(\r\cap
L^3\f)$ of quadratic leading terms of the space of minimal relations is dual to the cup product.

\begin{corollary}
\label{cor:quad_relns}
The diagram
$$
\xymatrix{
H_0\big(\n_-,\r/(\r\cap L^3\f)\big)\ar@{^{(}->}[d]\ar@{^{(}->}[r] &
\Lambda^2 H_0\big(\n_-,H_1(\u_\dot)\big) \ar[d]
\cr
\bigoplus_{\alpha \in \Sdual} H^2(\cG,V_\alpha)^\vee \ar[r]^(.4)\Delta &
\bigoplus_{\beta,\gamma \in \Sdual}
H^1(\cG,V_\beta)^\vee \otimes H^1(\cG,V_\gamma)^\vee
}
$$
commutes, where $\Delta$ denotes the dual of the cup product.
\end{corollary}

\section{Splittings}

A fundamental fact about motives is that (say) the complex vector space
underlying a mixed Hodge structure is naturally (though not canonically)
isomorphic to the direct sum of its weight graded quotients. In
Appendix~\ref{sec:splittings} we show that we can simultaneously split the
weight filtrations $M_\dot$ and $W_\dot$ on the various realizations of all
objects of $\MEM_\ast$, as well as all $\cGhat_\ast$ and $\cG_\ast^{\cris,\ell}$
modules. This will simplify the problem of writing down the presentations as it
allows us to work in the category of bigraded Lie algebras with $\sl(H)$-action,
instead of in the much more slippery category of pronilpotent Lie algebras.

Specifically, let $\cG$ be any of the groups $\pi_1(\MEM_\ast,\w)$, where $\ast
\in \{1,\uu,2\}$ and $\w$ is $\w^\DR$ or $\w^B$. Then every $\cG$-module $V$ has a natural bigrading
$$
V = \bigoplus_{m,n\in\Z} V_{m,n}
$$
in the category of $\Q$-vector spaces with the property that
$$
M_mV = \bigoplus_{\substack{r\le m\cr n\in \Z}} V_{r,n}
\text{ and }
W_nV = \bigoplus_{\substack{r\le n\cr m\in \Z}} V_{m,r}.
$$
It is natural in the sense that it is preserved by $\cG$-module homomorphisms.
It is compatible with $\otimes$ and $\Hom$. One can also choose the bigradings
to be compatible with the functors $\MEM_2 \to \MEM_\uu \to \MEM_1$.

\begin{remark}
One consequence of the proof is that the functors $\Gr^M_\dot\Gr^W_\dot$ are
exact. If one knew this in advance, one could use the fiber functor
$\Gr^M_\dot\Gr^W_\dot$ to prove the result. Analogous results hold for
$\cGhat_\ast$-modules and $\cG^{\cris,\ell}$-modules. One can arrange for their
splittings to be compatible with those of $\pi_1(\MEM_\ast,\w)$ via the
homomorphisms $\cGhat_\ast \to \pi_1(\MEM_\ast,\w^B)$ and $\cG_\ast^{\cris,\ell}
\to \pi_1(\MEM_\ast,\w^\ell)\otimes\Ql$.
\end{remark}

\begin{remark}
In Section~\ref{sec:splitting_DR}, we will construct a canonical $\Q$-de Rham
bigrading of each object of $\MEM_\ast$ which splits both weight filtrations and
also the Hodge filtration. This bigrading generalizes, and reduces to, the
canonical grading of the $\Q$-DR realization of a mixed Tate motive which splits
the weight filtration and the Hodge filtration.
\end{remark}

\subsection{Bases}
\label{sec:bases}
In the subsequent sections, it will be important to distinguish between $H =
H^1(E_\tate)$ and $\Hdual := H(1) = H_1(E_\tate)$. \label{def:Hdual} The
intersection pairing induces an isomorphism
$$
H^B \to \Hom(H^B,\Q) = \Hdual^B;\quad x \mapsto (x\cdot\blank)
$$
of their Betti realizations. Both are isomorphic to $H^B = \Q\aa \oplus \Q\bb
\cong \Hdual^B$. Their de~Rham realizations are different:
$$
H^\DR = \Q\aa \oplus \Q\bw
$$
where $\bw = -2\pi i \bb$. Note that $\bw$ spans a copy of $\Q(-1)$ in $H(1)$.
The de Rham realization of $H(1)$ is
$$
\label{def:A+T}
\Q A \oplus \Q T
$$
where $\aa = 2\pi i A$ and $T = -\bb$. This is the $\Q$-de Rham basis used in
\cite{hain:kzb}. In this part, we will be mainly working in homotopy, and thus
with $H(1) = H_1(E_\tate)$.

Note that $\SL(H)$ and $\SL(\Hdual)$ are naturally isomorphic, as are their Lie
algebras. We will identify them.

Below we will take $\w$ to be either $\w^B$ or $\w^\DR$. When we use $\w^B$, we
will use the $\Q$ Betti basis $\aa,\bb$ of $\Hdual$; when we use $\w^\DR$, we
will use the $\Q-\DR$ basis $A,T$ of $\Hdual$.

\subsection{Splittings of $\g^\MEM_\ast$}

Fix $\ast \in \{1,\uu,2\}$. Set $\g=\g^\MEM_\ast$ and $\u=\u^\MEM_\ast$. By the
results of the previous section, the Lie algebras $\g$ and $\u$ are
isomorphic to the degree completion of their associated bigraded Lie algebras
$$
\g \cong \prod_{\substack{n \le 0\cr m\in \Z}} \g_{m,n} \text{ and }
\u \cong \prod_{\substack{n < 0\cr m\in \Z}} \u_{m,n}.
$$
Since $\gl(H) \cong \Gr^W_0 \g$ and $\u_{m,n}=\g_{m,n}$ when $n<0$, we have the
decomposition
$$
\g \cong \gl(H) \ltimes
\prod_{\substack{n < 0\cr m\in \Z}}\u_{m,n} \cong \gl(H)\ltimes \u.
$$
This decomposition corresponds to a Levi decomposition
$$
\pi_1(\MEM_\ast,\w) \cong \GL(H) \ltimes \U^\MEM_\ast.
$$
Since a prounipotent group is determined by its Lie algebra, to give a
presentation of $\pi_1(\MEM_\ast,\w)$, it suffices to give a presentation of
$\u=\u^\MEM_\ast$ in the category of $\GL(H)$-modules. To do this, it suffices
to give a presentation of the bigraded Lie algebra
$$
\Gr\u := \bigoplus_{m,n}\u_{m,n}
$$
as a bigraded $\sl(H)$-module, where $\sl(H)$ is the bigraded Lie algebra
$$
\sl(H) = \s_{-2,0} \oplus \s_{0,0} \oplus \s_{2,0}
\cong \Q(1) \oplus \Q(0) \oplus \Q(-1).
$$
The elements
\begin{equation}
\label{eqn:e_0}
\ee_0^\DR = -A\frac{\partial}{\partial T} \text{ and }
\ee_0^B = 2\pi i\,\ee_0^\DR = \aa\frac{\partial}{\partial \bb} 
\end{equation}
span $\s_{-2,0}^\DR$ and $\s_{-2,0}^B$, respectively. The subalgebra $\s_{0,0}$
is a Cartan subalgebra and is diagonal with respect to the basis $A$, $T$ of
$H^\DR$. It is natural to assign the $\sl(H)$ weights $1$ to $T$ and $-1$ to
$A$. With this convention, $A\partial/\partial T$ has $\sl(H)$ weight $-2$.

Each weight graded quotient $\Gr^W_n V$ of a $\cG$-module $V$ is an
$\SL(H)$-module. The subspace $V_{m,n}$ of $\Gr^W_n V$ is the set of vectors of
$\sl(H)$ weight $m-n$. In other words, the three notions of weight are related
by the formula
\begin{equation}
\label{eqn:weights}
M\text{-weight} = \sl(H)\text{-weight} + W\text{-weight}
\end{equation}

\begin{convention}
Most formulas will hold for both the Betti and de~Rham realizations. For this
reason, we will often omit the $B$ or the $\DR$ and write, for example, $\ee_0$,
which can be interpreted as $\ee_0^B$ in the Betti realization and $\ee_0^\DR$
in the de~Rham realization.
\end{convention}

\section{Generators of $\u_\ast^\MEM$}
\label{sec:gens}

In this and subsequent sections, $\Gr V$, where $V$ is in $\MEM_\ast$, will
denote $\Gr^M_\dot\Gr^W_\dot V$. Since there is a natural isomorphism
$\Gr\u^\MEM_\ast \cong \u^\MEM_\ast$, to find a presentation of $\u^\MEM_\ast$,
it suffices to find a presentation of its associated bigraded $\Gr\u^\MEM_\ast$.

The abelianization of $\Gr \u^\MEM_\ast$ is easily computed by plugging the
computations of Theorem~\ref{thm:exts} into isomorphism (\ref{eqn:homology}) and
Proposition~\ref{prop:gens+relns}.

\begin{proposition}
For $\ast \in \{1,\uu,2\}$, we have
\begin{align*}
H_1(\Gr\u^\MEM_\ast)
&\cong
\bigoplus_{\substack{m\ge 0\cr r\in \Z}}
\Ext^1_{\MEM_\ast}(\Q,S^m\H(r))^\vee\otimes S^mH(r)
\cr
&\cong
\begin{cases}
\bigoplus_{m>0} \Q(2m+1) \oplus \bigoplus_{n>0} S^{2n}H(2n+1) & \ast = 1, \cr
\bigoplus_{m>0} \Q(2m+1) \oplus \bigoplus_{n\ge 0} S^{2n}H(2n+1) &
\ast = \uu, \cr
\bigoplus_{m>0} \Q(2m+1) \oplus H(1) \oplus \bigoplus_{n>0} S^{2n}H(2n+1) &
\ast = 2.
\end{cases}
\end{align*}
\end{proposition}

Since the surjection $\pi_1(\MEM_\ast,\w) \to \pi_1(\MTM,\w)$ is split, it
follows that the sequence
$$
0 \to H_1(\Gr\u^\geom_\ast) \to H_1\big(\Gr\u^\MEM_\ast\big)
\to H_1(\Gr\k) \to 0
$$
is exact. Since
$$
H_1(\Gr\k) = \bigoplus_{m>0} \Q(2m+1),
$$
is a trivial $\SL(H)$-module and since $H_1(\u^\geom_\ast)^{\SL(H)}=0$, we
obtain the following result.

\begin{corollary}
There is a natural $\GL(H)$-module isomorphism
$$
H_1(\Gr\u^\MEM_\ast) \cong H_1(\Gr \k) \oplus H_1(\Gr\u_\ast^\geom).
$$
\end{corollary}

Following the procedure in Section~\ref{sec:presentations}, we choose a graded
$\GL(H)$-invariant (and therefore bigraded) section of $\Gr \u^\MEM_\ast \to
H_1(\Gr \u^\MEM_\ast)$. This induces a surjection
$$
\L\big(H_1(\Gr\u^\MEM_\ast)\big) \to \Gr\u^\MEM_\ast.
$$
We may assume that these maps are compatible with the projections
$\u^\MEM_\uu\to \u^\MEM_2 \to \u^\MEM_1$. We now fix a basis of
$H_1(\Gr\u^\MEM_\ast)$.\footnote{We shall see shortly that there is a natural choice.} Its image under the section will be a generating set of
$\Gr\u^\MEM_\ast$. These generators will not be unique as they will depend on
the choice of the section. More will be said about this in
Remark~\ref{rem:gens}.

For the rest of this section $\ast \in \{1,\uu\}$. Recall that $\Q(n)$ is the
1-dimensional $\Gm$-module $\Q$ on which $\Gm$ acts by the $n$th power of the
standard character. Let \label{def:zz}
$$
\ssigma_{2m+1}\in
\Ext^1_\MTM(\Q,\Q(2m+1))^\vee\otimes \Q(2m+1) \subseteq H_1(\Gr \k)
$$
be the basis vector on which the motivic zeta value $\zeta^\m(2m+1)$ defined in
\cite{brown:mtm} takes the value 1. Let
$$
\zz_{2m+1}\in
\Ext^1_{\MEM_\ast}(\Q,S^0 H(2m+1))^\vee\otimes \Q(2m+1)
\subseteq H_1(\Gr \u^\MEM_\ast)
$$
be the corresponding element of $H_1(\Gr \u^\MEM_\ast)$. The projection
$\u^\MEM_\ast \to \k$ takes $\zz_{2m+1}$ to $\ssigma_{2m+1}$. Let
$$
\ee_{2n} \in
\Ext^1_{\MEM_\ast}(\Q,S^{2n-2}\H(2n-1))^\vee \otimes S^{2n-2}H(2n-1)
$$
\label{def:e_2n} be the element whose value on the class $\psi_{2n}$ of the
normalized Eisenstein series is the highest weight vector $2\pi i \bb^{2n-2} \in
\Q(1)$ of $S^{2n-2}H(2n-1)$.

\begin{remark}
\label{rem:gens}
Just as in the case of the generators $\ssigma_{2m+1}$ of $\Gr\k$, it is
important to note that, even though the $\zz_{2m+1}$'s are canonical in
$H_1(\Gr\u^\MEM_\ast)$, they are not canonical in $\Gr\u^\MEM_\ast$ as:
\begin{enumerate}

\item The lift of $\zz_{2m+1}$ from $H_1(\u^\MEM_\ast)$ to $\Gr\u^\MEM_\ast$ is
not unique. For example, $\zz_{11}$ can be replaced by any element of $\zz_{11}
+ \Q[\zz_3,[\zz_3,\zz_5]]$.

\item Even if we fix the lift of $\zz_{2m+1}$ to $\Gr\u^\MEM_\ast$, we can
adjust it by any $\SL(H)$ invariant element of
$\Gr^M_{-4m-2}\Gr^W_{-4m-2}\u^\geom_\ast$.

\end{enumerate}
Note, however, that the geometric generators $\ee_{2n}$ are unique as
$\Gr^M_{-2}\Gr^W_{-2n}\u^\MEM_\ast$ is one dimensional for each $n\ge 1$.
\end{remark}

\begin{remark}
The results of Section~\ref{sec:ihara-takao} imply that there is a canonical
$\Q$-de~Rham choice of the homomorphism
$$
\bigoplus_{m\ge 1} \Gr^M_{-4m-2}\k \to
\bigoplus_{m\ge 1} \Gr^M_{-4m-2}\Gr^W_{-4m-2}\u^\MEM_\uu.
$$
Namely, it is the homomorphism that takes $\zz \in \Gr^M_{-4m-2} \k$ to its
component (denoted by $\zz^\round{4m+2}$ in Section~\ref{sec:ihara-takao}) of
the image of $\zz$ in the $\sl(H)$ invariants of $\Gr^W_\dot\u^\MEM_\uu$, where
$\u_\uu^\MEM$ is identified with its $W$-graded quotient using the canonical
$\Q$-DR splitting of $W$ constructed in Section~\ref{sec:splitting_DR}.
\end{remark}

\begin{notation}
Denote the adjoint action of the enveloping algebra of $\g^\MEM_\ast$ on
$\u^\MEM_\ast$ by $f\cdot u$. In particular, if $x\in \g^\MEM_\ast$ and
$u\in \u^\MEM_\ast$, then $x^j\cdot u = \ad_x^j(u)$.
\end{notation}

With this notation,
$$
S^{2n-2}H(2n-1) = \Span\{\ee_0^j\cdot\ee_{2n}: 0\le j \le 2n-2\}.
$$
One also has the relation $\ee_0^{2n-1}\cdot \ee_{2n}=0$. Note that
$$
\zz_{2m+1} \in \Gr^M_{-4m-2}\Gr^W_{-4m-2}\u^\MEM_\ast \text{ and }
\ee_0^j\cdot\ee_{2n} \in \Gr^M_{-2-2j}\Gr^W_{-2n} \u^\MEM_\ast
$$
Their images under the chosen section of $\L(H_1(\Gr\u^\MEM_\ast)) \to
\Gr\u^\MEM_\ast$ will be a generating set, so that, for example,
$$
\Gr\u^\MEM_\uu = \f_\ast/\text{(relations)},
$$
where
$$
\label{def:f}
\f_\uu := \L(\ee_0^j\cdot\ee_{2n+2},\zz_{2m+1} :
n\ge0,\ m\ge 1,\ 0\le j\le 2n)
$$
and $\f_1 := \f_\uu/(\ee_2)$. This is free. Denote the Lie subalgebra of
$\f_\ast$ generated by the $\ee_0^j\cdot\ee_{2n+2}$, $j,n\ge 0$, by
$\f_\ast^\geom$. The Lie algebra $\u_\ast^\geom$ is a quotient of
$\f_\ast^\geom$.

We will refer to the $\ee_{2n}$ with $n\ge 0$ as {\em geometric generators} as
they generate $\Gr\u^\geom_\ast$. The $\zz_{2m+1}$'s will be called {\em
arithmetic generators} as they come from lifts of elements of $\k$.

\begin{remark}
\label{rem:degree}
Each element of $\Gr\u^\geom_\ast$ can be expressed as a Lie word in the
$\ee_{2n}$'s, where $n\ge 0$. Since each $\ee_{2n}$ is in
$\Gr^M_{-2}\u^\MEM_\ast$, each element of $\Gr^M_{-2d}$ can be expressed as a
Lie word of degree $d$ in the generators $\ee_{2n}$. We will therefore refer to
$\Gr^M_{-2d}\u^\MEM_\ast$ as the {\em degree $d$ elements of $\Gr\u^\MEM_\ast$}.
\end{remark}

\section{General Comments about the Relations in $\u^\MEM_\ast$}
\label{sec:relns}

Before embarking on the problem of determining the relations between the
generators of $\u^\MEM_\ast$, it is useful to ponder their shape and structure.
As above, $\ast \in \{1,\uu\}$. In this and subsequent sections, we denote
$\pi_1(\MEM_\ast,\w^B)$ by $\cG^\MEM_\ast$ and $\pi_1(\MEM_\ast,\w^B)^\geom$
by $\cG^\geom_\ast$.\label{def:GMEM}

Relations between the $\ee_{2n}$'s will be called {\em geometric relations}.
Since $\k$ is free, there are no relations between the $\ssigma_{2m+1}$'s and
hence no relations between the $\zz_{2m+1}$'s modulo the $\ee_{2n}$'s. Since
$\u^\geom_\ast$ is an ideal in $\u_\ast^\MEM$, there are relations of the form
\begin{equation}
\label{eqn:arith_reln}
[\zz_{2m+1},\ee_{2n+2}] =
\text{Lie word in the geometric generators } \ee_{2j},\ j\ge 0.
\end{equation}
These will be called {\em arithmetic relations}. They describe the action of
$\k$ on $\u^\geom_\ast$ and will be discussed in greater detail in
joint work of the first author with Francis Brown.


The bigraded ideals of relations are
$$
\r_\ast = \ker\{\f_\ast \to \Gr\u^\MEM_\ast\}
\text{ and }
\r^\geom_\ast := \ker\{\f_\ast^\geom \to \Gr\u^\geom_\ast\}
= \r_\ast \cap \f_\ast^\geom.
$$
These are $\gl(H)$-modules.

A minimal space of relations of $\u_\ast^\MEM$ is a minimal bigraded subspace of
the space of $\sl(H)$ highest weight vectors of $\r_\ast$ that generate it as an
ideal. A minimal set of relations of $\u^\MEM_\ast$ is a bigraded basis of a
minimal subspace of generators of $\r_\ast$.

Proposition~\ref{prop:gens+relns} implies that each minimal subspace of
relations in $\u^\MEM_\ast$ is isomorphic to
$$
H_0\big(\n_-,\r_\ast/[\r_\ast,\f_\ast]\big),
$$
where $\n_-$ is the Lie subalgebra $\Q\ee_0$ of $\sl(H)$. It also implies that
the space of minimal relations of $\u^\MEM_\ast$ of $\sl(H)$-weight $2n$ and
$M$-weight $-2d$, and thus $W$-weight $-2n-d$, is dual to
$$
H^2\big(\cG^\MEM_\ast,S^{2n}H(2n+d)\big).
$$
Remark~\ref{rem:degree} implies that the geometric relations of $M$-weight $-2d$
have degree $d$ in the $\{\ee_{2j}:j\ge 0\}$. It also implies that the right
hand side of the arithmetic relation (\ref{eqn:arith_reln}) has degree $2m+2$ in
the $\{\ee_{2j}:j\ge 0\}$ as it has $M$-weight $-4m-4$.

A minimal space of geometric relations is a subspace of the $\sl(H)$ highest
weight vectors in $\r_\ast^\geom$ that injects into
$H_0\big(\n_-,\r_\ast/[\r_\ast,\f_\ast]\big)$ under the canonical projection
$$
\r_\ast^\geom \to \r_\ast/[\r_\ast,\f_\ast]
\to H_0\big(\n_-,\r_\ast/[\r_\ast,\f_\ast]\big).
$$
Since $\k$ is free, $[\r_\ast,\f_\ast]\subseteq \r_\ast^\geom$, which implies
that any minimal space of geometric relations is isomorphic to
$$
H_0(\n_-,\r_\ast^\geom/[\r_\ast,\f_\ast]).
$$
A minimal set of geometric relations is a basis of a minimal space of geometric
relations. Those of $\sl(H)$-weight $2n$ and $M$-weight $-2d$ are dual to the
image of the restriction mapping
$$
H^2\big(\cG^\MEM_\ast,S^{2n}H(2n+d)\big) \to
H^2\big(\cG_\ast^\geom,S^{2n}H(2n+d)\big).
$$
As remarked above, these have degree $d$ in the $\ee_{2m}$'s and $W$-weight
$-2n-2d$. They occur as highest weight vectors of copies of $S^{2n}H(d)$ in the
relations.


\subsection{Conjectural size and shape of the relations}

In this section we explain how standard conjectures in number theory, if true,
would constrain the size and form of a minimal set of relations of $\Gr
\u^\MEM_\ast$. In Section~\ref{sec:pollack_motivic} we will give an {\em
unconditional} proof that there is a set of minimal relations is of this form and give a partial computation of it.

Corollary~\ref{cor:quad_relns} implies that the highest weight vectors in the
space of leading quadratic terms
$$
\r_\ast/\big([\r_\ast,\f_\ast]+\r_\ast\cap L^3\f_\ast\big)
= \r_\ast/(\r_\ast\cap L^3\f_\ast)
$$
of a set of minimal relations of $\u_\ast^\MEM$ of $M$-weight $d$ and
$\sl(H)$-weight $2n$ is dual to the image of the map
\begin{multline}
\label{eqn:cup_prod}
\bigoplus_{\substack{j+k=n+d-2\cr j,k\ge 0}}
H^1\big(\cG_\ast^\MEM,S^{2j}H(2j+1)\big) \otimes
H^1\big(\cG_\ast^\MEM,S^{2k}H(2k+1)\big)
\cr
\to H^2\big(\cG_\ast^\MEM,S^{2n}H(2n+d)\big)
\end{multline}
obtained by composing the cup product with a $\GL(H)$-invariant projection
\begin{equation}
\label{eqn:projn}
S^{2j}H(2j+1)\otimes S^{2k}H(2k+1) \to S^{2n}H(2n+d).
\end{equation}

Since $\k$ is free and acts trivially on $H_1(\u^\geom)$, we have
$[\k,\r^\geom_\ast] \subseteq L^3\f^\geom_\ast$. It follows that
$$
\r_\ast^\geom/(\r_\ast^\geom\cap L^3\f_\ast^\geom)
\to \r_\ast/(\r_\ast\cap L^3\f_\ast)
$$
is injective. The image is the set of quadratic leading terms (Def.~\ref{def:quad_part}) of the geometric
relations. The quadratic parts of geometric relations of $M$-weight $-2d$ have
degree $d$ as expressions in the $\{\ee_{2m}:m\ge 0\}$, but are quadratic in
$\f_\ast$.\footnote{For example, $[\ee_0^j\cdot \ee_{2a},\ee_0^k\cdot
\ee_{2b}]$, where $a,b>1$ and $j,k\ge 0$, is a quadratic element of $\f_1$, but
has degree $d=j+k+2$ in the $\ee_m$'s.}

To better
understand the relations we restrict to the case $\ast = 1$. Since $\u^\MEM_\uu
= \u^\MEM_1 \oplus \Q\ee_2$, where $\ee_2$ is central, there is no loss of
generality.\footnote{To understand the case $\ast =2$, by Corollary~\ref{cor:structure}, it suffices to understand the action of $\u_\uu^\MEM$ on $\p$. The action of the geometric generators is described in Section~\ref{sec:monod}.} The generator $\ee_{2m+2}$ of $\f_1$ (Betti or de~Rham, according to
taste) is dual to a generator of $H^1(\cG^\MEM_\ast,S^{2m}\H(2m+1))$ and is a
highest weight vector of the unique copy of $S^{2m}H$ in $H_1(\f_1)$. When
$m=0$, the generator $\zz_{2m+1}$ of $\f_1$ is dual to a generator of
$H^1(\cG^\MEM_1,\Q(2m+1))$. The projection (\ref{eqn:projn}) is dual to an
inclusion
$$
S^{2n}H(2n+d) \hookrightarrow H_1(\f_1)\otimes H_1(\f_1)
$$
and thus corresponds to a line of highest weight vectors in $H_1(\f_1)^{\otimes
2}$ of $\sl(H)$-weight $2n$. The line lies in $\Gr^M_{-2d}$ as the highest
weight vector of $S^{2n}H(2n+d)$ has $M$-weight $-2d$.  Since the cup product is
anti-commutative, the image of $S^{2n}H(2n+d)$ lies in $\Lambda^2 H_1(\f_1)
\cong L^2\f_1/L^3$.

To get an idea of the shape of the relations and to see how they might be
computed, we assume Conjecture~\ref{conj:h2}(\ref{item:hodge}) for the rest of
this section. This means that we identify
$H^\dot(\cG^\MEM_\ast,S^{2m}H(r))\otimes\R$ with
$H^\dot_\cD(\M_{1,\ast}^\an,S^{2m}\H_\R(r))^{\Frbar_\infty}$. 

Recall that $V_f$ denotes the two-dimensional real Hodge structure associated to
a Hecke eigen cusp form $f$ and that $\B_{2n+2}$ denotes the set of normalized
Hecke eigen cusp forms of weight $2n+2$. Each $V_f$ decomposes as the sum
of two 1-dimensional eigenspaces $V_f^\pm$ of the de~Rham Frobenius
$\Frbar_\infty$. The cup product
$$
H^1_\cusp(\SL_2(\Z),S^{2n}H) \otimes H^1_\cusp(\SL_2(\Z),S^{2n}H) \to H^2_\cusp(\SL_2(\Z),\Q)
$$
induces an isomorphism
$\Hom(V_f^\pm,\R) \cong V_f^\mp$. Conjecture~\ref{conj:h2}(\ref{item:hodge}) and Proposition~\ref{prop:deligne_coho} imply that there is an isomorphism
$$
H^2\big(\cG^\MEM_1,S^{2n}H(2n+d)\big)^\vee\otimes\R \cong
\begin{cases}
\R \oplus \bigoplus_{f\in \B_{2n+2}} V_f^+ & n>0,\ d \ge 2 \text{ even},\cr
 \bigoplus_{f\in \B_{2n+2}} V_f^- & n>0,\ d > 2 \text{ odd},\cr
0 & \text{ otherwise}.
\end{cases}
$$
The copy of $\R$ when $d$ is even corresponds to the Eisenstein series
$G_{2n+2}$.

\begin{proposition}
\label{prop:injectivity}
When $\ast = 1$, Conjecture~\ref{conj:h2}(\ref{item:hodge}) implies that the cup
product (\ref{eqn:cup_prod}) is surjective for all $n$ and $d$.
\end{proposition}

\begin{proof}
The interpretation of Brown's period computations \cite{brown:mmv} in terms of
Deligne--Beilinson cohomology \cite[\S7]{hain:db_coho} implies that the image
of the cup product (\ref{eqn:cup_prod}) under the regulator $\reg_\R$, after
tensoring with $\R$, is isomorphic to $H^2_\cD(\M_{1,1}^\an,S^{2n}\H(2n+d))$.
\end{proof}

\begin{corollary}
If we assume the Conjecture~\ref{conj:h2}(\ref{item:hodge}), then each minimal
relation is determined by its leading quadratic term. That is, the projections
$\r_1/[\r_1,\f_1] \to \r_1/(\r_1\cap L^3\f_1)$ and
$$
H_0\big(\n_-,\r_1/[\r_1,\f_1]\big) \to H_0\big(\n_-,\r_1/(\r_1\cap L^3\f_1)\big)
$$
are isomorphisms.
\end{corollary}

So, granted Conjecture~\ref{conj:h2}(\ref{item:hodge}), for each $f\in
\B_{2n+2}$ and each $d\ge 2$, there is a map
\begin{equation}
\label{eqn:cuspidal_relns}
V_{f}^\e \to \Gr^M_{-2d}\Gr^W_{-2n-2d}H_0(\n_-,\r_1/[\r_1,\f_1])
\end{equation}
where $\e$ is $+$ when $d$ is even and $-$ when $d$ is odd.\footnote{This map is
dual to the composite of the cup product with the projection onto $V_f^{-\e}$.}
The map
$$
V_f^\e \hookrightarrow \Gr^M_{-2d}H_0(\n_-,\r_1/(\r_1\cap L^3\f_1))
$$
obtained by composing (\ref{eqn:cuspidal_relns}) with the projection is dual to
the cup product. It is non-zero and therefore injective. The last statement of
Proposition~\ref{prop:injectivity} implies that the image of this
last map lies
in the geometric part
$$
\Gr^M_{-2d}H_0\big(\n_-,\r_1^\geom/(\r_1^\geom\cap L^3\f_1^\geom)\big).
$$
That is, the leading quadratic terms of each cuspidal relation are geometric.
More is true:

\begin{proposition}
Under the assumption that Conjecture~\ref{conj:h2}(\ref{item:hodge}) holds, the
image of (\ref{eqn:cuspidal_relns}) lies in the subspace
$\Gr^M_{-2d}\Gr^W_{-2n-2d}H_0(\n_-,\r_1^\geom/[\r_1,\f_1])$ of the space of
minimal relations.
\end{proposition}

\begin{proof}
As pointed out above, there are no relations between the $\zz_{2m-1}$'s modulo
the ideal generated by the $\ee_{2n}$'s. Since $\u_1^\geom$ is an ideal, each
relation $[\zz_{2m-1},\ee_{2n}]$ lies in $\f_1^\geom$. So we can replace every
occurrence of $\zz_{2m-1}$ in a cuspidal relation by an element of $\f_1^\geom$.
So every cuspidal relation lies in $\f_1^\geom$.
\end{proof}

When $d\ge 2$ is even, write $d=2m+2$, where $m\ge 0$. The dual
$$
\R \hookrightarrow
\Gr^M_{-2d}\Gr^W_{-2n-2d}H_0\big(\n_-,\r_1/(\r_1\cap L^3\f_1)\big)
$$
of the composition of the cup product with the projection associated to the
Eisenstein series $G_{2n+2}$ will contain the leading quadratic terms of the
arithmetic relation (\ref{eqn:arith_reln}) associated to the Eisenstein series
$G_{2m+2}$. Determining this map is a joint project with Francis Brown.

\section{The Monodromy Representation}
\label{sec:monod}

Any relation that holds in $\g^\MEM_\uu$ will hold in any homomorphic image.
Relations that hold between the images of the $\ee_{2n}$ in a homomorphic image
give an upper bound on the relations in $\u^\MEM_\uu$. In this section we 
determine how the generators $\ee_{2n}$ of $\Gr \u^\MEM_\uu$ act on the
unipotent fundamental group of $E_\tate'$.

As in previous sections, $\p$ denotes the Lie algebra of the unipotent
completion of $\pi_1(E_\tate',\ww)$. Recall that $\Hdual$ denotes $H(1)$. It can
be though of as the first homology group of the first order Tate curve
$E_\tate$. The splittings given by Proposition~\ref{prop:splittings} give
natural isomorphisms
$$
\Gr^W_\dot\p^B \cong \L(\Hdual^B) = \L(\aa,\bb)
\text{ and }
\Gr^W_\dot\p^\DR \cong \L(\Hdual^\DR) = \L(A,T)
$$
which extend to the natural bigraded isomorphism obtained by splitting each
$W$-graded quotient into its $\sl(H)$ weight spaces.

The natural monodromy action $\cG^\MEM_\uu \to \Aut \p$ induces a Lie algebra
homomorphism $\u^\MEM_\uu \to \Der \p$. This induces a $\gl(H)$-invariant
bigraded monodromy homomorphism
\begin{equation}
\label{eqn:gr_monod}
\Gr\u_\uu^\MEM \to \Gr \Der \p \cong \Der \Gr\p \cong \Der \L(\Hdual).
\end{equation}
Set $\theta = [\aa,\bb] \in \Gr^M_{-2}\Gr^W_{-2}\L(\Hdual)$.

\begin{proposition}
\label{prop:der0}
The image of the graded monodromy action (\ref{eqn:gr_monod}) lies in
$$
\Der^0 \L(\Hdual) := \{\d \in \Der \L(\Hdual) : \d(\theta) = 0\}.
$$
\end{proposition}

\begin{proof}
Set $E=E_\tate$. Let $\gamma_o$ be the element of $\pi_1(E',\ww)$ obtained by
rotating the tangent vector once about the identity. Observe that $\log \gamma_o
\in W_{-2}\p$ and that its image in  $\Gr^W_{-2}$ is $[\aa,\bb]$. Each element
of the group $\pi_1(\M_{1,\uu}^\an,\vv_o)$ is represented by an element of
$\Diff^+ (E,0)$ that acts trivially on $T_0 E$, the tangent space of $E$ at the
identity. This implies that the image of  $\Aut\p$ in
$\pi_1(\M_{1,\uu}^\an,\vv_o)$ is contained in the subgroup of elements that fix
$\log \gamma_o \in \p$ and therefore act trivially on $\Gr^W_{-2}\p$. This
implies that the image of $\Gr\u_\uu^\geom \to \Der\Gr\L(H)$ lies in
$\Der^0\L(H)$.

To complete the proof, recall that, using the splitting $\k \to \u_{\uu}^\MEM$
induced by the tangent vector $\tate$, we can identify $\u^\MEM_\uu$ with
$\k\ltimes \u^\geom_\uu$. Since $\log\gamma_o$ spans a copy of $\Q(1)$ in $\p$,
the image of $\k \to \u_\uu^\MEM \to \Der \p$ acts trivially on it. It follows
that $\Gr\k$ acts trivially on $\theta$, which completes the proof.
\end{proof}

Since, by Corollary~\ref{cor:structure}, $\u_\uu^\MEM \cong \u_1^\MEM \oplus
\Q\ee_2$ where $\ee_2$ is central, there are, by restriction, natural monodromy
representations
\begin{equation}
\label{eqn:monod}
\u_1^\MEM \to \Der\p \text{ and } \Gr\u^\MEM_1 \to \Der^0 \L(\Hdual).
\end{equation}
Restricting to this smaller algebra will simplify the computations in subsequent
sections.

\begin{remark}
Another way to see that there is a natural monodromy representation
$\Gr\u^\MEM_1 \to \Der^0\L(\Hdual)$ is to observe that, since the centralizer of a non-zero element $x$ in a free Lie algebra is spanned by $x$,
$$
\Der^0\L(\Hdual)\cap \Inn\L(\Hdual) = \Q\ad\theta.
$$
This implies that $W_{-3}\Der^0\L(\Hdual) \to W_{-3}\OutDer\L(\Hdual)$ is an
isomorphism. Since every generator of $\Gr\u_1^\MEM $ has weight $\le -4$, the
homomorphism $\u^\MEM_1 \to \OutDer \L(\Hdual)$ lifts to $\Der^0 \L(\Hdual)$.
\end{remark}

For a basis $\v_1,\v_2$ of $\Hdual$ and each $n\ge 0$, define the derivation
$\epsilon_{2n}(\v_1,\v_2)$ of $\L(\Hdual)$ by
\begin{equation}
\label{eqn:deltas}
\epsilon_{2n}(\v_1,\v_2) :=
\begin{cases}
-\v_2\frac{\partial}{\partial \v_1} & n = 0; \cr
\ad_{\v_1}^{2n-1}(\v_2)-
\sum_{\substack{j+k=2n-1\cr j>k > 0}}
(-1)^j[\ad_{\v_1}^j(\v_2),\ad_{\v_1}^k(\v_2)]
\frac{\partial}{\partial \v_2} & n > 0.
\end{cases}
\end{equation}
Here we are identifying $\L(\Hdual)$ with its image in $\Der\L(\Hdual)$ under
the inclusion $\ad : \L(\Hdual) \hookrightarrow \Der \L(\Hdual)$. Each
$\e_{2n}(\v_1,\v_2)$ annihilates $\theta$ and is thus in $\Der^0\L(\Hdual)$,
\cite[Prop.~21.2]{hain:kzb}.\footnote{These derivations occur in the work
\cite{tsunogai} of Tsunogai on the action of the absolute Galois group on the
fundamental group of a once punctured elliptic curve. They also occur in the
paper of Calaque et al \cite[\S 3.1]{kzb}.} It is useful to note that if
$c_1,c_2 \in \C$, then
$$
\e_{2n}(c_1\v_1,c_2\v_2) = c_1^{2n-1} c_2\, \e_{2n}(\v_1,\v_2).
$$

Recall the notation of Section~\ref{sec:bases}. Set
$$
\e_{2n} = \e_{2n}^\DR := \e_{2n}(T,A)
\text{ and }
\e_{2n}^B = 2\pi i \e_{2n}^\DR = \e_{2n}(-\bb,\aa).
$$
These lie in $\Gr^M_{-2}\Gr^W_{-2n}\Der^0\L(\Hdual)$ and are $\sl(H)$-highest
weight vectors of weight $2n-2$.\footnote{One can show that there is a unique
copy of $S^{2n-2}H(2n-1)$ in $\Gr^W_{-2n}\Der^0\L(H)$. It has highest weight
vector $\e_{2n}$.} When $n=0$, this agrees with the definition (\ref{eqn:e_0})
of $\ee_0$ given earlier.

Pollack \cite{pollack} found relations between the $\e_{2n}$.\footnote{Note that
in \cite{pollack}, $\e_{2n}$ denotes our $\e_{2n}(\aa,-\bb)$.} The relevance of
Pollack's computations to bounding the relations in $\Gr\u_\dot^\MEM$ comes from
the following result.

\begin{theorem}[{\cite[Thm.15.7]{hain:modular}}]
\label{thm:monod}
The monodromy representation (\ref{eqn:monod}) takes $\ee_0$ to $\e_0$ and
$\ee_{2n}\in \Gr\u^\MEM_\uu$ to $2\e_{2n}/(2n-2)!$ when $n>0$.
\end{theorem}

\begin{remark}
One consequence of this result is that a presentation of $\u^\geom_\uu$
determines a presentation of $\u^\geom_2$ as $\Gr \u^\geom_2 \cong
\Gr\u_1\ltimes\L(H)$ where the action is given by the $\ee_{2n}$.
\end{remark}

\section{The Canonical De~Rham Splitting}
\label{sec:splitting_DR}

The $\Q$-vector space $V^\DR$ that underlies a universal mixed elliptic motive
$\V$ in $\MEM_\ast$ is naturally isomorphic to its associated graded module
$$
V^\DR \cong \Gr^M_\dot V^\DR
$$
with respect to the weight filtration $M_\dot$. The canonical way to  realize
this isomorphism is via the identification
$$
F^m V^\DR \cap M_{2m}V^\DR \cong \Gr^M_{2m} V^\DR
$$
induced by the inclusion $F^m V^\DR \cap M_{2m}V^\DR \hookrightarrow
M_{2m}V^\DR$. This grading splits both the Hodge and weight filtrations. As
shown in Appendix~\ref{sec:splittings}, this isomorphism can be lifted to a
natural isomorphism
$$
V^\DR \cong \Gr^W_\dot \Gr^M_\dot V^\DR
$$
of $V^\DR$ with its associated bigraded module which splits the Hodge filtration
and both weight filtrations. The purpose of this section is to show that this
lift is canonical.

\begin{theorem}
\label{thm:DR_splitting}
For every object $\V$ of $\MEM_\ast$, there is a unique isomorphism
$$
V^\DR \cong \bigoplus_{m,n} \Gr^W_n \Gr^M_{2m} V^\DR
$$
of rational vector spaces, where $V^\DR$ is the $\Q$-de~Rham realization of the fiber $V\in \MTM$ of $\V$ over $\vv_o$ with the following properties:
\begin{enumerate}

\item The isomorphism is natural with respect to morphisms in $\MEM_\ast$ and
compatible with tensor products and duals.

\item This bigrading of $V^\DR$ refines the standard grading $V^\DR \cong
\bigoplus_p F^pM_{2p}V^\DR$ of the de~Rham realization of a mixed Tate motive.
That is, for each $p\in\Z$, this isomorphism restricts to an isomorphism
$$
F^p V^\DR \cap M_{2p}V^\DR \cong \bigoplus_n  \Gr^W_n \Gr^M_{2p} V^\DR.
$$
\end{enumerate}
\end{theorem}

\begin{proof}
Denote $\pi_1(\MEM_\ast,\w^\DR)$ by $\cG_\ast$, its prounipotent radical by
$\U_\ast$ and its Lie algebra by $\g_\ast$. Results from
Section~\ref{sec:tannaka} imply that there are natural homomorphisms $\cG_\uu
\to \cG_2 \to \cG_1$. This means that every object of $\MEM_\ast$ can be viewed
as an object of $\MEM_\uu$. So it suffices to prove the result for $\ast = \uu$.
Applying Proposition~\ref{prop:splittings} to $\cG_\uu$ implies that such
compatible bigradings which split the Hodge filtration and both weight
filtrations exist. 

To prove uniqueness, observe that such bigradings correspond to a lift of the
central cocharacter $\Gm \to \GL(H)$ that gives the $\Q$-DR splitting of $H$ to
a cocharacter $\Gm \to \cG_\uu$. This cocharacter lies in $F^0W_0M_0\cG_\uu$ and
is unique up to conjugation by an element of $F^0W_0M_0\U_\uu$. But since $\Gr
\u_\uu$ is generated the $\ee_0^j\cdot \ee_{2n}$ and the $\zz_{2m+1}$ with
$n,m>0$, all of which lie in $M_{-2}$, and since $F^0 \cap M_{-2}=0$,
$F^0W_0M_0\U_\uu = 0$. Uniqueness follows.
\end{proof}

Denote the $\Q$-DR realization of the Lie algebra of the unipotent fundamental
group of $(E'_\tate,\ww_o)$ by $\p$. The isomorphism $H_1(\p) \cong \Q A \oplus
\Q T$ given in Section\ref{sec:bases} induces an isomorphism $\Gr \p \cong
\L(A,T)$ of bigraded Lie algebras. The canonical $\Q$-DR bigrading induces an
isomorphism
$$
\psi : \p \to \L(A,T)^\wedge
$$
which is compatible with the bracket and satisfies
\begin{equation}
\label{eqn:bigrading}
T\in \Gr^W_{-1}\Gr^M_0 \L(A,T),\qquad A \in \Gr^W_{-1}\Gr^M_{-2} \L(A,T).
\end{equation}

The KZB connection \cite{kzb,levin-racinet} is the $\Q$-DR realization of connection associated to the local system of unipotent fundamental groups of punctured elliptic curves over $\M_{1,2}$. This was suggested in \cite{levin-racinet} and proved in \cite{hain:kzb,rome}. It gives an isomorphism
$$
\phi : \p\otimes\C \to \L(A,T)^\wedge \otimes \C
$$
of $\p\otimes \C$ with the completion of the bigraded Lie algebra $\L(A,T)$
which splits the Hodge filtration and both weight filtrations. (Cf.\
\cite[\S15]{hain:kzb}.) In order to exploit the formulas derived in
\cite{hain:kzb} using the elliptic KZB connection, we need to show that the KZB bigrading agrees with the canonical one constructed above.

\begin{proposition}
\label{prop:kzb-splitting}
The isomorphism
$$
\phi \circ \psi^{-1} : \L(A,T)^\wedge \otimes\C \to \L(A,T)^\wedge\otimes \C
$$
is the identity, so that the canonical bigrading of $\p$ constructed above and  the bigrading given by the KZB connection agree.
\end{proposition}

\begin{proof}
The isomorphism $\phi : \p\otimes\C \to \L(A,T)^\wedge\otimes\C$ constructed
from the KZB connection induces the isomorphism
$$
H_1(\p)\otimes\C \cong \C T \oplus \C A
$$
of Section~\ref{sec:bases} which respects the canonical bigrading. It follows
that the map induced by $\phi \circ \psi^{-1}$ on $H_1$ is the identity. 

Since $\phi \circ \psi^{-1}$ corresponds to a morphism of mixed Tate structures,
it induces an isomorphism
$$
F^p\cap M_{2p} \L(A,T)^\wedge \otimes \C \to
F^p\cap M_{2p} \L(A,T)^\wedge\otimes \C
$$
for all $p\in \Z$. Since $F^0 \L(A,T)^\wedge\otimes\C = \C T$, it follows
that $\phi \circ \psi^{-1}(T) = T$. Since
$$
F^{-1}\cap M_{-2} \L(A,T)^\wedge = \prod_{n\ge 0} \C\, T^n\cdot A,
$$
it follows that
$
\phi \circ \psi^{-1} : A \mapsto A + \sum_{n=1}^\infty c_n T^n \cdot A
$
where each $c_n \in \C$.

As in the previous section, we denote by $\gamma_o$ the element of
$\pi_1^\un(E'_\tate,\ww_o)$ that rotates the tangent vector $\ww_o$ in $T_0
E_\tate$ once about $0$. Since it spans a copy of $\Q(1)$ in $\p$, its image
under $\psi$ is $2\pi i[T,A]$. But since, by \cite[\S12]{hain:kzb},  the residue
of the KZB connection at the identity of $E_\tate$ is $[T,A]$, the image of
$\log \gamma_o$ under $\phi$ is also $2\pi i[T,A]$. So $\phi \circ
\psi^{-1}([T,A]) = [T,A]$. Therefore
$$
[T,A] = \phi \circ \psi^{-1}([T,A]) =
[T,A] + \sum_{n=1}^\infty c_n T^{n+1} \cdot A.
$$
It follows that all $c_n$ vanish, so that $\phi \circ \psi^{-1}(A) = A$ and
$\phi \circ \psi^{-1}$ is the identity.
\end{proof}

By identifying $\g^\MEM_\uu$ with its associated bigraded Lie algebra via the
de~Rham splitting, we can regard each $\ee_{2n}$ as an element of $\g^\MEM_\uu$.
The following result is a consequence of the previous result and
Theorem~\ref{thm:monod}.

\begin{corollary}
The homomorphism $\g^\MEM_\ast \to \Der \L(A,T)^\wedge$ respects the bigrading
when $\ast = \uu,2$. It takes $\ee_{2n}$ to $2\e_{2n}/(2n-2)!$.
\end{corollary}

\begin{remark}
The theorem implies that the Hodge realization of an object of $\MEM_1$ is an
Eisenstein variation of MHS over $\M_{1,1}^\an$. (See
\cite[Rem.~16.5]{hain:modular} for a definition.) This implies, in particular,
that the weight filtration $W_\dot$ of the vector bundle $\cV$ over
$\Mbar_{1,1/\Q}$ that underlies the canonical extension of $\V$ to $\Mbar_{1,1}$
is isomorphic to its weight graded quotient: $$ \cV \cong \bigoplus_n \Gr^W_n
\cV $$ and that, in the notation of \cite{hain:modular}, the connection on $\cV$
is of the form $\nabla_0 + \Omega$, where $\nabla_0$ denotes the connection on
$\Gr^W_n \cV$ that is described in Section~\ref{sec:connection} and
$$
\Omega = \sum_n \psi_{2n}(\bphi_{2n}) \in
\Omega^1_{\Mbar_{1,1/\Q}}\otimes \End\big(\bigoplus_n \Gr^W_n \cV\big)
$$
where $\bphi_{2n} \in \Gr^M_{-2}\Gr^W_{-2n}\End V^\DR$ is the image of $\ee_{2n}$
under the monodromy homomorphism $\Gr \u_1^\DR \to \End \Gr V^\DR$.
\end{remark}

\section{Pollack's Relations}
\label{sec:pollack}

In this section we recall Pollack's relations \cite{pollack}, which give an
upper bound on the leading quadratic terms of these relations. If the cup
product (\ref{eqn:cup_prod}) is surjective, as would follow from
Conjecture~\ref{conj:h2}(\ref{item:hodge}), this will give an upper bound on the
space of minimal relations in $\u^\MEM_1$.

Pollack found all highest weight relations of degree $d=2$ that hold between the
$\e_{2n}$ and all highest weight relations between the $\e_{2n}$ of higher
degree $d>2$ that hold modulo a certain filtration that we now define.

Suppose that $\n$ is a Lie algebra. Denote its commutator subalgebra by $\n'$.
Define the filtration $P^\dot$ of $\n$ by \label{def:lcs}
$$
P^0 = \n,\quad P^m\n = L^m \n' \qquad m > 0.
$$
where $L^m$ denotes the $m$th terms of the lower central series. The filtration
$P^\dot$ of $\L(\Hdual)$ induces a filtration\footnote{In
Section~\ref{sec:depth}, we will see that this is the filtration induced by the
natural {\em elliptic depth filtration} of $\Der\p$.} on $\Der \L(\Hdual)$ by
$$
P^m \Der^0 \L(\Hdual) =
\big\{\delta \in \Der^0 \L(\Hdual) :
\delta P^j\L(\Hdual) \subseteq P^{j+m}\L(\Hdual)\big\}.
$$
It has the property that $[P^j\n,P^k\n] \subseteq P^{j+k}\n$.

\begin{lemma}
The image of $L^m \Gr\u^\MEM_\uu$ in $\Der^0\L(\Hdual)$ under the monodromy homomorphism $\Gr \u_\uu^\MEM \to \Der^0\L(\Hdual)$ is contained in $P^m\Der^0\L(\Hdual)$.
\end{lemma}

\begin{proof}
A derivation of $\L(\Hdual)$ lies in $P^1\Der^0\L(\Hdual)$ if and only if it
acts trivially on both $H_1(\L(\Hdual))$ and $H_1\big(\L(\Hdual)'\big)$. The map
$(\Sym \Hdual)(1)\to H_1\big(\L(\Hdual)'\big)$ defined by
$$
x_1\dots x_n \mapsto
\sum_{\sigma\in \Sigma_n} \big(x_{\sigma(1)}\dots x_{\sigma(n)}\big)\cdot \theta
$$
is an isomorphism. If $\delta \in \Der^0\L(\Hdual)$, then
$$
\delta \big((x_1 x_2 \dots x_n)\cdot \theta\big) \equiv
\sum_{j=1}^n \big(\delta(x_j)
x_1 x_2 \dots x_{j-1}x_{j+1}\cdots x_n\big)\cdot \theta
\bmod L^2\L(\Hdual).
$$
So if $\delta$ acts trivially on $H_1(\L(\Hdual))$, then it acts trivially
on $H_1\big(\L(\Hdual)'\big)$.

Since the image of $\u^\MEM_\uu$ in $\sl(H)$ is trivial, its image in
$\Der^0\L(\Hdual)$ acts trivially on $H_1(\L(\Hdual))$ and therefore lies in
$P^1\Der^0\L(\Hdual)$. It follows that $L^m\u_\uu^\MEM$ is mapped into
$P^m\Der^0\L(\Hdual)$.
\end{proof}

Regard $S^{2m}H(2m+1)$ as the copy of $S^{2a}H$ in $\Gr\u_\uu^\MEM$ generated by
$\ee_{2m+2}$. If $a\ge b >0$, then the image of the bracket is
$$
\big[S^{2a}H(2a+1), S^{2b}H(2b+1)\big] \cong 
\begin{cases}
\bigoplus_{r=0}^{2b} S^{2a+2b-2r}H(2a+2b+2-r) & a > b,\cr
\bigoplus_{r=1}^{2a} S^{4a-4r+2}H(4a-2r+3) & a=b.
\end{cases}
$$
The determination of the highest weight vectors in $L^2\f_1/L^3$ of degree $d$
is an exercise in the representation theory of $\sl_2$.

\begin{proposition}
If $a$, $b$ and $d$ are positive integers satisfying $0 \le d-2\le 2\min(a,b)$,
then
$$
\what^d_{a,b} := \frac{1}{4} \sum_{\substack{i+j=d-2\cr i\ge 0, j\ge 0}}
(-1)^i\binom{d-2}{i}(2a-i)!(2b-j)!
[\ee_0^i\cdot \ee_{2a+2},\ee_0^j\cdot \ee_{2b+2}]
$$
spans the space of $\sl(H)$ highest weight vectors in the subspace
$$
\Gr^M_{-2d}[S^{2a}H(2a+1),S^{2b}H(2b+1)]
$$
of $\u_\uu^\MEM$. Its image in $\Gr^M_{-2d}\Der^0\L(\Hdual)$ is
$$
w^d_{a,b} := \sum_{\substack{i+j=d-2\cr i\ge 0, j\ge 0}}
(-1)^i\binom{d-2}{i}\frac{(2a-i)!(2b-j)!}{(2a)!(2b)!}
[\e_0^i\cdot \e_{2a+2},\e_0^j\cdot \e_{2b+2}],
$$
which is an $\sl(H)$ highest weight vector in $\Gr^M_{-2d}\Der^0\L(\Hdual)$.
\end{proposition}

Both $\what_{a,b}^d$ and $w_{a,b}^d$ have $W$-weight $-2a-2b-4$ and
$\sl(H)$-weight $2a+2b-2d+4$. Note that $\what_{a,b}^d$ vanishes when either $a$
or $b$ is zero as $\ee_2$ and $\e_2$ are central. Since $\what_{a,b}^d =
(-1)^{d+1}\what_{b,a}^d$, we can, and will, assume that the coefficients in an
expression
$$
\sum_{a+b=n} c_a w_{a,b}^d
$$
satisfy $c_a + (-1)^d c_b = 0$.

Before stating Pollack's result we need to recall a few basic facts about
cuspidal cocycles. This material is standard, but conventions are not. We
will use those given in \cite[\S17]{hain:modular}. The left action of
$\SL_2(\Z)$ on $H^1(E)$ is given by $(\aa,-\bb)\mapsto (\aa,-\bb)\gamma$. (Cf.\
\cite[Lem.~9.2]{hain:modular}.)

Denote the standard cochain complex of $\SL_2(\Z)$ with coefficients in the
left-module $V$ by $C^\dot(\SL_2(\Z),V)$ and its differential by $\delta$. Set
$$
Z^1_\cusp(\SL_2(\Z),S^m H) =
\{ r \in C^1(\SL_2(\Z),S^m H) : \delta r = 0, r(T) = 0\},
$$
where $T =\begin{pmatrix}1 & 1 \cr 0 & 1 \end{pmatrix}$. These groups vanish
when $m$ is odd. Cuspidal cocycles are determined by their value on $S =
\begin{pmatrix} 0 & -1 \cr 1 & 0\end{pmatrix}$ and are characterized by the
functional equations
$$
(I + S) r = (I+ U + U^2) r = 0,
$$
where $U = ST$. We will identify a cuspidal cocycle with its value on $S$.

The real Frobenius operator $\Fr_\infty$ acts on cuspidal cocycles via the
formula
$$
\Fr_\infty : \aa \to \aa,\quad \Fr_\infty: \bb \to -\bb.
$$
Let
$$
Z^1_\cusp(\SL_2(\Z),S^m H) =
Z^1_\cusp(\SL_2(\Z),S^m H)^+ \oplus Z^1_\cusp(\SL_2(\Z),S^m H)^-
$$
be the eigenspace decomposition. The $+$ (resp.\ $-$) eigenspace consists of
polynomials with even (resp.\ odd) degree in $\aa$.

The map $Z^1_\cusp(\SL_2(\Z),S^{2n}H) \to H^1_\cusp(\M_{1,1},S^{2n}\H)$ that
takes a cuspidal cocycle to its cohomology class has kernel spanned by the
unique cuspidal coboundary $\d(\aa^{2n}) = \bb^{2n} - \aa^{2n}$. There is thus a
short exact sequence
\begin{equation}
\label{eqn:cuspidal_ses}
0 \to \Q\delta(\aa^{2n}) \to
Z^1_\cusp(\SL_2(\Z),S^{2n}H) \mapsto H^1_\cusp(\M_{1,1},S^{2n}\H) \to 0.
\end{equation}
It is equivariant with respect to the $\Fr_\infty$ action. The class map has a
section given by modular symbols.

Recall that the modular symbol of a cusp form $f$ of $\SL_2(\Z)$ of weight
$2n+2$ is the homogeneous polynomial \label{def:mod_symb}
$$
\sr_f(\aa,\bb) = \sum_{j=0}^{2n} a_f(j)\aa^j\bb^{2n-j}
:= (2\pi i)^{2n+1} \int_0^{i\infty} f(\tau)(\bb-\tau\aa)^{2n} d\tau \in S^{2n}H
$$
of degree $2n$. This decomposes $\sr_f = \sr_f^+ + \sr_f^-$, where $\sr_f^\pm
\in Z^1_\cusp(\SL_2(\Z),S^{2n}H)^\pm\otimes\C$.

\begin{theorem}[Pollack]
\label{thm:pollack}
If $n > 0$, then
$$
\sum_{a+b=n} c_a [\e_{2a+2},\e_{2b+2}] = 0
$$
if and only if there is a cusp form $f$ of $\SL_2(\Z)$ of weight $2n+2$ with
$\sr_f^+(\aa,\bb) = \sum c_a \aa^{2a} \bb^{2n-2a}$. If $n\ge d \ge 2$, then
$$
\sum_{a+b=n} c_a w_{a,b}^d \equiv 0 \bmod P^3 \Der^0 \L(\Hdual)
$$
if and only if
$$
\sum_{\substack{a+b=n\cr 2a,2b\ge d-2}}
c_a \aa^{2a-d+2} \bb^{2b-d+2} \in Z^1_\cusp\big(\SL_2(\Z),S^{2n-2d+4}H\big)^\e,
$$
where $\e$ is $+$ when $d$ is even and $-$ when $d$ is odd.
\end{theorem}

Here $f$ is arbitrary, in that it can have complex Fourier coefficients. The relations with arithmetic significance correspond to normalized Hecke eigenforms.

\begin{remark}
The second statement can be rewritten to focus on the cocycle:
$$
\sum_{\substack{A+B=2N\cr A,B \equiv d\, (2)}} c_A \aa^A \bb^B \in
Z^1_\cusp\big(\SL_2(\Z),S^{2N}H\big)^\e
$$
if and only if
$$
\sum_{\substack{A+B=2N\cr A,B \equiv d\, (2)}} c_A w^d_{a,b}
\equiv 0 \bmod P^3\Der^0\L(\Hdual),
$$
where $a=(A+d-2)/2$ and $b=(B+d-2)/2$.
\end{remark}

\begin{remark}
The trivial cuspidal cocycle $\bb^{2N}-\aa^{2N}$ in
$Z^1_\cusp(\SL_2(\Z),S^{2N}H)$ corresponds to the even degree relation
$$
w^{2k}_{k-1,N+k-1} - w^{2k}_{N+k-1,k-1}  = 2 w^{2k}_{k-1,N+k-1}
\in \Gr^M_{-4k}\Gr^W_{-2N-4k}\Der^0\L(\Hdual)/P^3.
$$
This has $\sl(H)$-weight $2N$. There is one non-trivial such relation for each
$k>1$.
\end{remark}

\section{Pollack's Relations are Motivic}
\label{sec:pollack_motivic}

In this section we use Brown's period computations \cite{brown:mmv} to prove
that there are relations in $\Gr\u_1^\MEM$ that project to Pollack's relations
in $\big(\Der\L(H)\big)/P^3$. These have also been proved in \cite[\S\S16,20]{brown:mmv} using tannakian methods. Both proofs use the same general approach, which was suggested in \cite{hain:letter}. Namely, that relations in $\u^\MEM_1$ correspond to certain non-trivial extensions of MHS in the coordinate ring $\O(\cG^\rel_1)$ of the relative completion of $\SL_2(\Z)$. The period computations in \cite{brown:mmv}, which are key, imply the non-triviality of these extensions and thus the existence of the lifted relations.

One can ask if the lifted Pollack relations generate all relations in $\u^\MEM_1$. We show that if we assume Conjecture~\ref{conj:h2}, then the lifts of the Pollack relations generate all relations in $\u^\MEM_1$. We continue with the notation of the previous four sections.

\begin{theorem}
\label{thm:motivic}
Pollack's relations that correspond to period polynomials of cusp forms are
motivic. That is, for each cusp form $f$ of $\SL_2(\Z)$ of weight $2n+2$ and
each $d\ge 2$, there is an element $\br_{f,d}$ of
$\Gr^M_{-2d}\r^\geom_1\otimes \C$ with
$$
\br_{f,d} \equiv \sum_{a+b=d-2} c_a \what_{a,b}^d \bmod L^3 \f_1^\geom,
$$
where
$$
\sr_f^\e(\aa,\bb) = \sum_{\substack{a+b=n\cr 2a,2b\ge d-2}}
c_a \aa^{2a-d+2} \bb^{2b-d+2} \in Z^1_\cusp\big(\SL_2(\Z),S^{2n-2d+4}\big)^\e
$$
is the even or odd part of the modular symbol of $f$, where $\e$ is the sign of
$(-1)^d$. The Pollack relation corresponding to the trivial cuspidal cocycle
lifts to the arithmetic relation
\begin{multline*}
[\zz_{2m-1},\ee_{2n+2}]
\cr
\equiv
\frac{(2m-2)!}{(2n+2m)!}\binom{2n+2}{2}
\frac{B_{2n+2m}}{B_{2n+2}}
\sum_{\substack{i+j=2m-2\cr i,j\ge 0}}
(-1)^i \frac{(2n+i)!}{i!}[\ee_0^i\cdot \ee_{2m},\ee_0^j\cdot \ee_{2n+2m}]
\end{multline*}
mod $L^3 \Gr^M_{-4m}\f_1^\geom$.
\end{theorem}

The ``tail'' of $\br_{f,d}$ is well defined only modulo the ideal generated by
the $\br_{g,e}$ where $g$ has lower weight and $e<d$. Note also that if the
Fourier coefficients of an eigenform $f$ are rational, then the coefficients of the quadratic part of $\br_{f,d}$ a rational multiplies of a fixed complex number.

The arithmetic relations look a little nicer if we set
$$
\be_{2k} = \frac{(2k-2)!}{2}\ee_{2k}.
$$
The image of $\be_{2k}$ in the derivation algebra is $\e_{2k}$. With this
normalization, the arithmetic relations become
$$
[\zz_{2m-1},\be_{2n+2}] \equiv
\frac{(2n+2)!}{(2n+2m)!}
\frac{B_{2n+2m}}{B_{2n+2}}
\bw^{2m}_{m-1,n+m-1}
\bmod L^3 \Gr^M_{-8}\f_1^\geom.
$$
When $m=2$, these become
$$
[\zz_3,\be_{2n}] \equiv
\frac{1}{2}
\frac{1}{\binom{2n+2}{2}}
\frac{B_{2n+2}}{B_{2n}}
\bw^{4}_{1,n}
\bmod L^3 \Gr^M_{-8}\f_1^\geom.
$$
This agrees with Pollack's computations of $[\tilde{z}_3,\e_{2n}]$ for $n\le 6$
in $\Der^0\L(H)$ if one takes $\zz_3$ to $1/12$ times his derivation
$\tilde{z}_3$, as $\bw^{4}_{1,n}$ goes to his $E^4_{2n}$ in the derivation
algebra.

Baumard and Schneps \cite{schneps:relns} used combinatorial methods to prove
that Pollack's cubic relations between the $\e_{2n}$ in $\big(\Der^0
\L(H)\big)/P^3$ lift to relations in $\Der^0 \L(H)$. Brown gives a stronger
version of the second statement in \cite{brown:depth3}. Since we have proved
that Pollack's relations lift to relations in $\u^\MEM$, it follows that all of
his relations lift to relations in the derivation algebra.

\begin{corollary}
Pollack's relations in $\Der^0\L(\Hdual)/P^3$ lift to relations in
$\Der^0\L(\Hdual)$ for all $d\ge 3$. The Lie algebra $\k$ acts trivially on
$\Der^0\L(\Hdual)/P^3$.
\end{corollary}

\begin{proof}[Proof of Theorem~\ref{thm:motivic}]
The idea behind the proof is to bound the quadratic heads of the relations in
$\u^\MEM_1$ from above using Pollack's relations and from below using the cup
product in the Deligne cohomology of $\cG^\rel_1$.

For $\ast \in \{1,\uu\}$, denote the image of $\pi_1(\cG^\geom_\ast) \to
\Aut \p$ by $\cS_\ast^\geom$. Corollary~\ref{cor:elliptic_polylog} implies that
the coordinate ring of the image is an object of $\MTM$. 
Set
$$
\cS_\ast = \pi_1(\MTM)\ltimes \cS_\ast^\geom.
$$
Corollary~\ref{cor:structure} implies that there is a natural isomorphism
$\cS^\geom_\uu \cong \cS^\geom_1 \times \Ga(1)$. So it suffices to prove the
case $\ast = 1$.

Denote by $P^\dot$ the filtration of $\Aut \p$ induced by the filtration
$P^\dot$ of $\p$. It is obtained by exponentiating the filtration $P^\dot$ of
$\Der\p$ defined above. It restricts to a filtration $P^\dot$ of $\cS_\ast$. One
has the following factorizations
$$
\xymatrix{
\cGhat_\ast \ar[r] & \cG_\ast^\MEM \ar@{->>}[r]\ar[dr] &
\cS_\ast \ar@{->>}[r]\ar@{^{(}->}[d] & \cS_\ast/P^3 \ar@{^{(}->}[d] 
\cr
&& \Aut \p \ar[r] & \big(\Aut \p\big)/P^3
}
$$
of the monodromy representation. There are therefore maps
\begin{multline*}
H^\dot\big(\cS_\ast/P^3,S^{m}H(r)\big)
\to
H^\dot\big(\cG^\MEM_\ast,S^{m}H(r)\big)
\cr
\to
H^\dot\big(\cGhat_\ast,S^m H(r)\big) \cong
H^\dot_\cD\big(\cG^\rel_\ast,S^{m}H(r)\big)
\end{multline*}
that are compatible with cup products. Proposition~\ref{prop:deligne_coho} and
Theorems~\ref{thm:exts} and \ref{thm:monod} imply that each of these maps is an
isomorphism in degree 1 for all $m$ and $r$.

In the rest of the proof, $\e_d\in \{+,-\}$ is the sign of $(-1)^d$ and $\e_d'$
is the sign of $(-1)^{d+1}$.

Recall from \cite{brown:mmv} that when $n>0$, the Haberlund--Petersson inner
product induces a non-singular pairing
$$
Z^1_\cusp\big(\SL_2(\Z),S^{2n}H\big)^\pm \otimes
H^1\big(\M_{1,1}^\an,S^{2n}\H\big)^\mp \to \Q.
$$
under which the sequence (\ref{eqn:cuspidal_ses}) is dual to the sequence
$$
0 \to H^1_\cusp(\M_{1,1}^\an,S^{2n}\H)
\to H^1(\M_{1,1}^\an,S^{2n}\H) \to S^{2n}H/\im \ee_0 \to 0.
$$
The duality (Lemma~\ref{lem:quad_relns}) between cup products and quadratic
heads of the relations implies that the Pollack relations of degree $d$ with
$\SL_2$ highest weight $2n$ correspond to a surjection
$$
r_\cS : H^2\big(\cS_1/P^3,S^{2n}H(2n+d)\big) \to
\big[Z^1_\cusp\big(\SL_2(\Z),S^{2n}H\big)^{\e_d}\big]^\vee
\cong H^1\big(\M_{1,1}^\an,S^{2n}\H\big)^{\e_d'},
$$
where $[\blank]^\vee$ denotes dual, and that the restriction of $r_\cS$ to
the image of the cup product
\begin{multline*}
\bigoplus_{j+k=n+d-2}
H^1\big(\cS_1/P^3,S^{2j}H(2j+1)\big)\otimes
H^1\big(\cS_1/P^3,S^{2k}H(2k+1)\big)
\cr
\to H^2\big(\cS_1/P^3,S^{2n}H(2n+d)\big)
\end{multline*}
is an isomorphism.

By the results of Section~\ref{sec:vmhs}, there is an isomorphism
$$
H^\dot\big(\cGhat_1,S^{2n}H(2n+d)\big) \cong
H^\dot_\cD\big(\M_{1,1}^\an,S^{2n}H(2n+d)\big).
$$
Since the class of the Eisenstein series in
$H^1_\cD\big(\M_{1,1}^\an,S^{2n}\H(2n+1)\big)$ is $\Frbar_\infty$-invariant, the
image of the cup product lies in
$H^2_\cD\big(\M_{1,1}^\an,S^{2n}H_\R(2n+d)\big)^{\Frbar_\infty}$. Define $r_\cD$ to be the projection
\begin{align*}
H^2_\cD\big(\M_{1,1}^\an,S^{2n}H_\R(2n+d)\big)^{\Frbar_\infty}
&\cong
\Ext^1_\MHS(\R,H^1(\M_{1,1}^\an,S^{2n}H_\R(2n+d)))^{\Frbar_\infty}
\cr
&\cong H^1(\M_{1,1}^\an,S^{2n}\H_\R)^{\e_d'}
\cr
&\to H^1_\cusp\big(\M_{1,1}^\an,S^{2n}\H_\R\big)^{\e_d'}.
\end{align*}

Since the $H^1$'s of $\cS_1/P^3$, $\cG^\MEM_1$ and $\cGhat_1$ are all
isomorphic, we have the following commutative diagram:
{\tiny
$$
\xymatrix@C=-42pt{
&
\bigoplus_{j+k=n+d-2}
H^1\big(\cG^\MEM,S^{2j}H(2j+1)\big)\otimes
H^1\big(\cG^\MEM,S^{2k}H(2k+1)\big)
\ar@{->>}[dl]\ar@{->>}[d]\ar@{->>}[dr]
\cr
\text{image of cup prod}_\cS \ar@{^{(}->}[d] \ar[r] &	
\text{image of cup prod}_\MEM \ar@{^{(}->}[d] \ar[r] &	
\text{image of cup prod}_\cD \ar@{^{(}->}[d]		
\cr
H^2(\cS_1/P^3,S^{2n}H(2n+d)\big) \ar@{->>}[r]\ar@{->>}[d]^{r_\cS} &
H^2(\cG^\MEM_1,S^{2n}H(2n+d)\big) \ar[r]  &
H^2(\cGhat_1,S^{2n}H(2n+d)\big)^{\Frbar_\infty}  \ar[d]^{r_\cD}
\cr
H^1\big(\M_{1,1}^\an,S^{2n}\H_\Q\big)^{\e_d'} \ar@{.>}[rr]^\varphi
&&
H^1\big(\M_{1,1}^\an,S^{2n}\H_\R\big)^{\e_d'}.
}
$$
}

Brown's period computations \cite[Cor.~11.2]{brown:mmv} and the computation of
the cup product in \cite[\S8,\S10]{hain:db_coho} imply that there is a
homomorphism $\varphi$ that makes the diagram commute and is a multiple of the
adjoint of the Haberlund--Petersson pairing. It is therefore an isomorphism
after tensoring with $\R$. Since the restriction of $r_\cS$ to the image of
$\text{cup prod}_\cS$ is an isomorphism,
$$
\text{image of cup prod}_\cS \to \text{image of cup prod}_\MEM
$$
is injective. This implies that Pollack's relations lift to $\u_1^\MEM$.

The precise form of the arithmetic relation follows from
\cite[Thm.~10.5]{hain:db_coho} and the duality (Corollary~\ref{cor:quad_relns})
between cup products and the quadratic terms of the relations.
\end{proof}

As a corollary of the proof we obtain the following statement which relates
surjectivity of the cup product with standard conjectures. The point of this
result is that the degree 1 cohomology is motivic, so the image of the cup
product should be as well.

\begin{theorem}
\label{thm:h2_hodge}
The image of the composition
\begin{multline*}
\bigoplus_{j+k=n+d-2}
H^1\big(\cG_1^\MEM,S^{2j}H(2j+1)\big)\otimes
H^1\big(\cG_1^\MEM,S^{2k}H(2k+1)\big)
\cr
\to H^2\big(\cG_1^\MEM,S^{2n}H(2n+d)\big)
\to H^1(\M_{1,1}^\an,S^{2n}\H_\R\big)^{\e'}
\end{multline*}
of the cup product with the projection to the $\Frbar_\infty$ invariant part of
the real Deligne cohomology of $\M_{1,1}^\an$ is a $\Q$-form of
$H^2_\cD(\M_{1,1}^\an,S^{2n}\H_\R(2n+d)\big)^{\Frbar_\infty}$. Here, $\e'\in \{+,-\}$ is the sign of $(-1)^{d+1}$.
\end{theorem}

The corresponding result also holds for $\ast = \uu,2$. The cohomology classes dual to the relations account for the degree 2 real Deligne cohomology of $\M_{1,\ast}^\an$ that is in the image of the cup product. For example, when $\ast = \uu$, the relation $[\ee_2,\ee_{2n}]=0$ is dual to the copy of $\R$ in $H^2_\cD(\M_{1,\uu}^\an,\R(2n))$ when $n\ge 2$ given in Proposition~\ref{prop:deligne_coho}. Similarly, when $\ast = 2$, the relations that give the action of the $\ee_{2n}$ on the generators $A,T$ (i.e., the formula for $\e_{2n}$ of $\L(\Hdual)$) correspond to copies of $\R$ in $H^2_\cD(\M_{1,2}^\an,S^{2m+1}\H_\R(2r))^{\Frbar_\infty}$.

From the discussion in Section~\ref{sec:presentations}, we know that the map
\begin{equation}
\label{eqn:head}
\r_\ast/[\r_\ast,\f_\ast] \to \r_\ast/(\r_\ast\cap L^3\f_\ast),
\end{equation}
which takes a minimal relation to its quadratic head, is injective if and only
if the cup product
\begin{multline}
\label{eqn:prod}
\bigoplus_{j+k=n+d-2}
H^1\big(\cG_\ast^\MEM,S^{2j}H(2j+1)\big)\otimes
H^1\big(\cG_\ast^\MEM,S^{2k}H(2k+1)\big)
\cr
\to H^2\big(\cG_\ast^\MEM,S^{2n}H(2n+d)\big)
\end{multline}
is surjective for all $n$ and $d$. Since, by Theorem~\ref{thm:h2_hodge}, the
image of the cup product under
$$
\reg_\Q :
H^2\big(\cG_1^\MEM,S^{2n}H(2n+d)\big)
\to H^2_\cD(\M_{1,1}^\an,S^{2n}\H_\R(2n+d)\big)^{\Frbar_\infty}
$$
is a $\Q$-form of
$$
H^2_\cD(\M_{1,1}^\an,S^{2n}\H_\R(2n+d)\big)^{\Frbar_\infty} \cong 
H^1(\M_{1,1}^\an,S^{2n}\H_\R\big)^{\e'},
$$
Conjecture~\ref{conj:h2}(\ref{item:hodge}) is equivalent to the surjectivity of
(\ref{eqn:head}).

\begin{corollary}
\label{cor:h2_hodge}
Suppose that $\ast = 1$.
The following statements are equivalent:
\begin{enumerate}

\item Conjecture~\ref{conj:h2}(\ref{item:hodge}) is true;

\item every non-trivial minimal relation in $\u_\ast^\MEM$ has a non-trivial
quadratic head --- that is, (\ref{eqn:head}) is injective;

\item the cup product (\ref{eqn:prod}) is surjective for all $n$ and $d$.

\end{enumerate}
\end{corollary}

\section{Problems, Questions and Conjectures}

\subsection{$\ell$-adic Analogues}

The existence of the lifts of the Pollack relations to $\u^\MEM_1$ was proved
using Hodge theory. One can ask whether one can also establish their existence
using $\ell$-adic methods and whether the $\ell$-adic analogues of
Theorem~\ref{thm:h2_hodge} and Corollary~\ref{cor:h2_hodge} hold.

\begin{conjecture}
For all prime numbers $\ell$ and all $n,d>0$, the cup product
\begin{multline*}
\bigoplus_{j+k=n+d-2}
H^1\big(\cG^{\cris,\ell}_1,S^{2j}H_{\Ql}(2j+1)\big)\otimes
H^1\big(\cG^{\cris,\ell}_1,S^{2k}H_{\Ql}(2k+1)\big)
\cr
\to H^2\big(\cG^{\cris,\ell}_1,S^{2n}H_\Ql(2n+d)\big)
\cr
\overset{\text{projn}}{\longrightarrow}
\Hfte^1\Big(G_\Q,H^1_\cusp(\M_{1,1/\Qbar},S^{2n}\H_\Ql(2n+d)\big)\Big)
\end{multline*}
is surjective and the natural homomorphism
$$
H^2\big(\cG^\MEM_1,S^{m}H(r)\big)\otimes \Ql
\to H^2\big(\cG^{\cris,\ell}_1,S^{m}H_\Ql(r)\big)
$$
is an isomorphism.
\end{conjecture}

\subsection{Does $\p$ Generate $\MEM_1$?}

Another approach to understanding $\cG^\MEM_1$ is to ask whether it is
faithfully represented in $\Aut \p$.

\begin{question}
Is the homomorphism $\cG^\MEM_1 \to \Der \p$ injective? Equivalently,
does $\p$ generate $\MEM_1$ as a tannakian category?
\end{question}

If true, this would be an elliptic analogue of Brown's Theorem \cite{brown:mtm}.
The proof \cite{takao} of the Oda Conjecture implies the kernel of $\cG^\MEM_1
\to \Aut \p$ lies in $\cG_1^\geom$, so the conjectured statement is equivalent
to injectivity of $\cG^\geom_1 \to \Aut \p$.

\subsection{Eisenstein Quotients}

One can ask how faithfully $\MEM_1$ is represented in Hodge theory. To explain
this, recall from \cite[\S 16]{hain:modular} that $\cG^\eis_1$ denotes the
maximal Tate quotient of $\cG^\rel_1$ in the category of affine groups with MHS.
Since the Hodge realization of $\cG^\geom_1$ is a mixed Hodge-Tate structure,
the quotient mapping $\cG^\rel_1 \to \cG^\geom_1$ factors through a surjective
homomorphism $\cG^\eis_1 \to \cG^\geom_1$.

\begin{conjecture}
The natural surjection $\cG^\eis_1 \to \cG^\geom_1$ is an isomorphism.
\end{conjecture}

If true, this implies, via Corollary~\ref{cor:structure}, that $\cG^\MEM_\ast
\cong \pi_1(\MTM)\ltimes \cG_\ast^\eis$ for $\ast \in \{1,\uu,2\}$.

Denote the Lie algebra of the prounipotent radical of $\cG^\eis_1$ by
$\u_1^\eis$. Brown's period relations imply that Pollack's geometric (i.e.,
cuspidal) relations lift to relations in $\u^\eis_1$. If
conjecture~\ref{conj:h2}(\ref{item:hodge}) is true, then $\cG^\eis_1 \to
\cG^\geom_1$ is an isomorphism.

\subsection{Massey Products}

The higher order terms of the relations $\br_{f,d}$ correspond to non-vanishing
matric Massey products\footnote{See \cite{may} for the definition of matric
Massey products. The point for us is that the cohomology of a pronilpotent Lie
algebra is generated from its degree 1 cohomology by matric Massey products.} in
$H^2_\cD\big(\M_{1,1}^\an,S^{2n}\H(2n+d)\big)$ of the classes
$$
\bG_{2m+2} \in H^1_\cD\big(\M_{1,1}^\an,S^{2m}\H(2m+1)\big)
$$
of the normalized Eisenstein series $G_{2m+2}$. One approach to determining the
relations in $\cG^\MEM_1$ is to compute these.

Pollack \cite{pollack} found explicit lifts of the first two cubic relations
to $\Der^0\L(H)$. The additional terms imply that there are non trivial matric
Massey triple products of degree 2 in the cohomology of $\cG_1^\MEM$.
Specifically, as Pollack points out, the cubic relation
\begin{multline*}
80[\e_{12},[\e_4,\e_0]] + 16[\e_4,[\e_{12},\e_0]]
-250[\e_{10},[\e_6,\e_0]] - 125[\e_6,[\e_{10},\e_0]]
\cr
+ 280[\e_8,[\e_8,\e_0]] -462[\e_4,[\e_4,\e_8]] -1725[\e_6,[\e_6,\e_4]] = 0
\end{multline*}
that corresponds to the normalized cusp form $\Delta$ of weight 12, can be
rewritten (after dividing by $-40$) as
$$
4(\w^3_{1,5} + \w^3_{5,1}) - 25(\w^3_{2,4} + \w^3_{4,2}) + 42\,\w^3_{3,3}
+ \frac{231}{20}[\e_4,[\e_4,\e_8]] +\frac{345}{8}[\e_6,[\e_6,\e_4]] = 0.
$$
The first 3 terms comprise the quadratic head of the relation. It corresponds to
the multiple $4(x^9y + xy^9) -25(x^7y^3+x^5y^7) + 42\, x^5y^5$ of
$\br_\Delta^-$. The terms of the cubic tail imply that the projection of the
Massey triple products\footnote{Since there are no cusp forms of weight $<12$,
these Massey products have no indeterminacy.}
$$
\frac{231}{20}\langle \bG_4,\bG_4, \bG_8 \rangle
\text{ and }
\frac{345}{8} \langle \bG_6,\bG_6, \bG_4 \rangle
$$
to the cuspidal summand of $H^2_\cD\big(\M_{1,1}^\an,S^{10}\H(13)\big)$ under
the multiplication maps
\begin{multline*}
S^2\H(3) \otimes S^2\H(3) \otimes S^6\H(7) \to S^{10}\H(13)
\text{ and }
\cr
S^4\H(5) \otimes S^4\H(5) \otimes S^2\H(3) \to S^{10}\H(13)
\end{multline*}
are equal. It also implies that (with the correct normalization) these equal the
images of
$$
4\, \bG_{12}\smile \bG_4 = -25\, \bG_{10}\smile\bG_6 = 42\, \bG_8\smile \bG_8
$$
in $H^2_\cD\big(\M_{1,1}^\an,S^{10}\H(13)\big)$.

\begin{problem}
Compute (matric) Massey products in
$H^2_\cD\big(\M_{1,1}^\an,S^{2n}\H(2n+d)\big)$ of the classes of Eisenstein
series. Use them to compute higher order terms of the relations $\br_{f,d}$.
\end{problem}

The indeterminacy in the tails of the relations $\br_{f,d}$ and the
indeterminacy in the matric Massey products correspond.

\subsection{Motivic Sheaves}

Ayoub \cite{ayoub}, using the work of Voevodsky \cite{voevodsky} has constructed
a category of motivic sheaves. Arapura \cite{arapura}, using the work of Nori,
has also constructed a category of motivic sheaves. These constructions are not
known to be equivalent.

\begin{question}
Are universal elliptic motives the realization of motivic sheaves over
$\M_{1,\ast/\Z}$?
\end{question}

This makes sense when $\ast \in \{\uu,2\}$ as $\M_{1,\ast/\Z}$ is a scheme 
in these cases. When $\ast = 1$, define a motivic sheaf over $\M_{1,1/\Z}$
to be a $\Gm$-invariant motivic sheaf on $\M_{1,\uu/\Z}$ that is trivial
on $\Gm$ orbits.

An affirmative answer to this question will imply that the specialization of a
universal mixed elliptic motive (as a set of compatible realizations) to a point
$[E] \in \M_{1,1}(F)$ corresponding to an elliptic curve over $F$ will actually
be the set of realizations of a motive over $F$.

\part{Relation to $\MTM$ and the Genus $0$ Story}
\label{part:genus0}

Universal mixed elliptic motives are related to the study of the unipotent
fundamental group of $\Pminus$ via degeneration to the nodal cubic $\Ebar_0$.
This is because the nodal cubic can be identified with $\P^1_{/\Z}$ with $0$ and
$\infty$ identified. The corresponding group is $E_0 = \Gm_{/\Z}$ and the
associated punctured elliptic curve $E_0'$ is $\Gm-\{\id\} = \Pminus$. Unless
otherwise stated, in this part we will work with $\Q$-de~Rham realizations and
its canonical bigrading constructed in Section~\ref{sec:splitting_DR}.

\section{Degeneration to the Nodal Cubic}
\label{sec:pminus}

In this section we recall some basic formulas from \cite[Part~3]{hain:kzb} which
were deduced from the elliptic KZB connection \cite{kzb,levin-racinet}. Denote
the natural parameter on $\P^1_{\Z}$ by $w$, so that
$$
\P^1_{\Z} - \{0,1,\infty\} = \Spec \Z[w,1/w,1/(w-1)].
$$
Denote the tangent vector $\partial/\partial w$ at $w=1$ by $\ww_o$. We can
identify the nodal cubic $E_0$ (the fiber of the universal elliptic curve over
$q=0$) with $\P^1$ with $0$ and $\infty$ identified.

The morphism $(\P^1_{\Z} - \{0,1,\infty\},\ww_o) \to
(E_{\tate'/\Z[[q]]},\ww_o)$ induces a homomorphism (cf.\ \cite[\S18]{hain:kzb})
\begin{equation}
\label{eqn:homom}
\pi_1^\un(\Pminus,\ww_o) \to \pi_1^\un(E_\tate',\ww_o)
\end{equation}
on unipotent fundamental groups. The induced Lie algebra homomorphism
$$
\Lie\pi_1^\un(\Pminus,\ww_o) \to \Lie\pi_1^\un(E_\tate',\ww_o)
$$
is {\em not} a morphism in $\MEM_\uu$, but it is a morphism of $\MTM$ and is
thus $\pi_1(\MTM)$-equivariant.

To write down formulas, we now identify the de~Rham realizations of $\k$ and
$\Lie\pi_1^\un(\Pminus,\ww_o)$ with the completions of their $M_\dot$ graded
quotients
$$
\L(\ssigma_3,\ssigma_5,\ssigma_7,\ssigma_9,\dots) \text{ and }
\L(X_0,X_1) \cong \L(X_0,X_1,X_\infty)/(X_0+X_1+X_\infty),
$$
respectively via the canonical $\Q$-de~Rham splitting of mixed Tate motives.
Likewise, we identify $\p:= \Lie\pi_1^\un(E_\tate',\ww_o)^\DR$ with the
completion of its associated $(M_\dot, W_\dot)$-bigraded module via the
$\Q$-de~Rham splitting of mixed elliptic motives constructed in
Section~\ref{sec:splitting_DR}. To fix notation, recall from
Section~\ref{sec:bases} that
$$
\Hdual^\DR = H_1(E'_\tate)^\DR = \Q A \oplus \Q T
$$
where $T$ spans the copy of $\Q(0)^\DR$ and $A$ spans the copy of $\Q(1)^\DR$.
The associated bigraded of $\p$ is naturally isomorphic to $\L(\Hdual^\DR) =
\L(A,T)$.

In \cite[\S18]{hain:kzb} it is shown that, after identification with associated
graded Lie algebras as above, the map
\begin{equation}
\label{eqn:pminus}
\L(X_0,X_1,X_\infty)^\wedge/(X_0+X_1+X_\infty) \to \L(A,T)^\wedge
\end{equation}
induced by (\ref{eqn:homom}) is given by
\begin{align}
\label{eqn:inclusion}
X_0 &\mapsto R_0 = \bigg(\frac{T}{e^T-1}\bigg)\cdot A, \cr
X_1 &\mapsto R_1 = [T,A], \cr
X_\infty &\mapsto R_\infty = \bigg(\frac{T}{e^{-T}-1}\bigg)\cdot A,
\end{align}
where $\cdot$ denotes the adjoint action of $\Q\ll A,T\rr = U\L(A,T)$ on
$\L(A,T)$. This homomorphism is injective and $\k$-equivariant.\footnote{Similar
and related formulas appear in \cite[Prop.~4.9]{kzb} and
\cite[Prop.~3.8]{enriquez}.}

Recall from Section~\ref{sec:localsys_H} that $\sigma_o$ denotes the positive
integral generator of the fundamental group $\pi_1(\D^\ast,\tate)$ of the
$q$-disk, which is isomorphic to $\Q(1)$. It acts on $\pi_1^\un(E_\tate',\ww_o)$
and thus on $\p$. The image of $\sigma_o$ in $\Aut \p$ lies in the prounipotent
subgroup $W_{-2}\Aut\p$ and thus has a logarithm. Set
$$
N = N^\DR := \frac{1}{2\pi i} \log \sigma_o \in \Der\p.
$$
In \cite[\S13]{hain:kzb} it is shown that, after identification with the
associated bigraded module,
\begin{equation}
\label{eqn:monod_log}
N = \sum_{m\ge0} (2m-1)\frac{B_{2m}}{(2m)!}\, \e_{2m}
\in W_0 \Gr^M_{-2}\Der^0\L(A,T)^\wedge,
\end{equation}
where $B_{2m}$ denotes the $2m$th Bernoulli number. It commutes with the action
of $\k$ as it spans a copy of $\Q(1)$.

\section{Depth Filtrations}
\label{sec:depth}

There is a natural depth filtration in the elliptic case which generalizes the depth filtration in classical case, $\Pminus$. The classical and elliptic depth filtrations are quite closely related via the weight filtration $W_\dot$. Unfortunately, the link between them is not as straightforward as one might hope, as we shall see.

In this section we set $E=E_\tate$ and $U=\Pminus$. We will denote the Lie
algebra of $\pi_1^\un(X,\ww_o)$ by $\p(X)$, when $X=U$, $\Gm$, $E$, and
$E'$.

Recall that the depth filtration $D^\dot$ of the Lie algebra
$\pi_1^\un(U,\ww_o)$ is defined by
$$
D^d\p(U) = 
\begin{cases}
\p(U) & d = 0,\cr
\ker\big\{\p(U) \to \p(\Gm)\big\} & d = 1,\cr
L^d D^1\p(U) & d > 1,
\end{cases}
$$
where $L^d$ denotes the $d$th term of the lower central series. In the elliptic
case, define\footnote{This is the filtration $P^\dot$ defined in
Section~\ref{sec:pollack}.}
$$
D^d\p(E') = 
\begin{cases}
\p(E') & d = 0,\cr
\ker\big\{\p(E') \to \p(E)\big\} & d = 1,\cr
L^d D^1\p(E') & d > 1.
\end{cases}
$$
These filtrations are motivic and thus preserved by $\k$. The first is a
filtration of $\p(U)$ in $\MTM$ and the second is a filtration of $\p(E')$
in $\MEM_\uu$. They can thus be described on the associated graded modules: $D^d
\L(X_0,X_1)$ is the ideal spanned by the Lie monomials whose degree in $X_1$ is
$\ge d$; and $D^d\L(A,T)$ is the ideal spanned by the Lie monomials in which
$\theta := [T,A]$ occurs at least $d$ times.

\begin{remark}
Note that $\Gr^M_\dot \p(E')$ is naturally isomorphic to the $T$-adic completion
of $\L(A,T)$.
\end{remark}

The morphisms
$$
\xymatrix{
U \ar@{^{(}->}[r]\ar@{^{(}->}[d] & \Gm\ar@{^{(}->}[d] \cr	   
E' \ar@{^{(}->}[r] & E						
}
$$
preserve the base point $\ww_o$ and thus induce homomorphisms
$$
\xymatrix{
\p(U) \ar[r]\ar@{^{(}->}[d] & \p(\Gm)\ar@{^{(}->}[d] \cr	
\p(E') \ar[r] & \p(E)
}
$$
That the vertical homomorphisms are injective follows from the fact that a
Lie subalgebra of a free Lie algebra is free. It follows that $\p(U) \to
\p(E')$ preserves $D^\dot$.

\begin{proposition}
\label{prop:depth_strict}
The inclusion $\p(U) \hookrightarrow \p(E')$ is strictly compatible with
the depth filtrations. That is,
$$
D^d \p(U) = \p(U) \cap D^d\p(E').
$$
\end{proposition}

\begin{proof}
Since $\Gr^W_\dot$ and $\Gr^M_\dot$ are exact, we can work with the associated gradeds. The
Lie algebra $D^1\L(X_0,X_1)$ is freely generated by the set $\{X_0^m\cdot
X_1:m\ge 0\}$. Since it is free, it follows that for all $m\ge 1$, $D^m
\L(X_0,X_1) = L^m D^1\L(X_0,X_1)$. This remains true after completion when we
replace ``generated'' by ``topologically generated''. The Lie algebra $D^1
\L(A,T)$ is also free. Its abelianization is naturally isomorphic to the free
$\Sym \Hdual$-module generated by $\theta$. The map (\ref{eqn:pminus}) induces
the inclusion
$$
H_1\big(D^1\L(X_0,X_1)^\wedge\big) \to H_1\big(D^1\L(A,T)^\wedge\big)
$$
that takes $X_0^m\cdot X_1$ to $T^m\cdot\theta$ and is therefore injective. It
follows that
$$
L^m D^1\L(X_0,X_1)^\wedge = \L(X_0,X_1)^\wedge \cap L^m D^1\L(A,T)^\wedge,
$$
which implies the result.
\end{proof}

The weight filtration $W_\dot$ of $\p(E')$ restricts to a filtration of
$\p(U)$:
$$
W_{-m}\p(U) := \p(U)\cap W_{-m}\p(E').
$$
The three filtrations $D^\dot$, $M_\dot$ and $W_\dot$ are related by the
``convolution formula''.\footnote{The {\em convolution} of two filtrations
$F^\dot$ and $G^\dot$ of a vector space $V$ is the filtration $F\ast G$ defined
by $(F\ast G)^n V := \sum_{j+k=n} F^jV\cap G^k V$, There is a natural
isomorphism $\Gr^n_{F\ast G} V \cong \bigoplus_{j+k=n} \Gr^j_F\Gr^k_G V$.}

\begin{proposition}
\label{prop:convolution}
For all $m\ge 0$, we have
$$
W_{-m}\p(U) = \sum_{n+d\ge m} M_{-2n}\p(U)\cap D^d\p(U),
$$
so that
$$
\Gr^W_{-m}\p(U) \cong \bigoplus_{n+d=m} \Gr^M_{-2n}\Gr_D^d \p(U).
$$
\end{proposition}

\begin{proof}
Identify $\p(U)$ with $\L(X_0,X_1)^\wedge$ and $\p(E')$ with $\L(A,T)^\wedge$.
We first show that if $m=n+d$, then 
$$
M_{-2n}\p(U)\cap D^d\p(U) \subseteq W_{-m}\p(E').
$$
Suppose that $w$ is homogeneous element of $\L(X_0,X_1)$ of bidegree $(a,b)$ in
$X_0$ and $X_1$. Then $w \in W_{-2n}\cap D^d$ if and only if $b\ge d$ and $a+b
\ge n$. These conditions imply that $a+2b\ge m$. On the other hand, the
corresponding word in $\L(A,T)^\wedge$ lies in $W_{-a-2b}\L(A,T)^\wedge$,
which is contained in $W_{-m}$ when $a+2b\ge m$. It follows that
$$
W_{-m}\p(U) \supseteq
\sum_{n+d\ge m} M_{-2n}\p(U)\cap D^d\p(U).
$$

To prove the reverse inclusion, define a filtration $G_\dot$ on
$\L(X_0,X_1)^\wedge$ by
$$
G_{-m}\L(X_0,X_1)^\wedge = \sum_{n+d\ge m} M_{-2n}\L\cap D^d\L(X_0,X_1)^\wedge.
$$
The first part implies that the identity $\big(\L(X_0,X_1)^\wedge,G_\dot\big)
\to \big(\L(X_0,X_1)^\wedge,W_\dot\big)$ is filtration preserving. To prove
$G_\dot = W_\dot$, it suffices to show that the identity induces an isomorphism
on associated gradeds. An easy linear algebra argument implies that
$$
\Gr^G_{-m}\L(X_0,X_1) \cong \bigoplus_{n+d=m} \Gr^M_{-2n}\Gr_D^d \L(X_0,X_1).
$$
This implies that $\Gr^G_{-m}\L(X_0,X_1)$ is spanned by the monomials $w$ in
$X_0$ and $X_1$ of bidegree $(a,b)$ with $a+2b=m$. Since these also span
$\Gr^W_{-m}\L(X_0,X_1)$, the identify induces an isomorphism on associated
gradeds.
\end{proof}

\begin{remark}
One might suspect that one can invert this formula to get a formula for $D^\dot$
in terms of $M_\dot$ and $W_\dot$. However, while it is true that
$$
\Gr_D^d\L(X_0,X_1) = \sum_{m-n=d} \Gr^W_{-m}\Gr^M_{-2n}\L(X_0,X_1),
$$
it is {\em not} true that $D^d\L(X_0,X_1) =
\bigoplus_{m-n=d} W_{-m}\L(X_0,X_1) \cap M_{-2n}\L(X_0,X_1)$.
The point being that the right hand side is {\em not} the convolution of
$W_\dot$ and $M_\dot$.
\end{remark}

\begin{remark}
The proposition does not hold if $U$ is replaced by $E'$. For example,
$$
A^{n-1}T^{m-n-1}\cdot\theta \in \Gr^M_{-2n}\Gr^W_{-m}\L(\Hdual).
$$
On the other hand, it lies in $D^1\L(\Hdual)$ and projects to a non-trivial
element of $\Gr_D^1\L(\Hdual)$. So the relation $n+d=m$ between depth and the
weights holds if and only if the element is of the form $A^{n-1}\cdot\theta$.
\end{remark}

The depth filtrations on $\p(E')$ and $\p(U)$ induce depth filtrations
on their derivation algebras in the standard way. These are motivic as the
depth filtrations on $\p(E')$ and $\p(U)$ are.

Define the Lie algebra of {\em nodal derivations} $\Der^N\p(E')$ of $\p(E')$
to be \label{def:derN}
$$
\big\{
\delta \in \Der\p(E') : [\delta,N]=0,\ \delta(\log\sigma_o)=0
\text{ and }\delta\big(\p(U)\big)\subset \p(U)
\big\}.
$$
This is a pro-object of $\MTM$. Define the {\em extendable derivations} $\Der^E
\p(U)$ to be the image of the restriction map \label{def:derE}
$$
\Der^N \p(E') \to \Der \p(U).
$$
It is also a pro-object of $\MTM$. The image of the natural action of $\k$
on $\p(E')$ lands in $\Der^N\p(E')$, so that we have the commutative diagram
$$
\xymatrix{
\k \ar[r]^(.36){\phi_E} \ar[rd]_(.42){\phi_\P} & \Der^N\!\p(E') \ar[d]^\res \cr
& \Der^E\p(U).
}
$$
The kernel of the restriction map $\res$ is spanned by $\e_0$. Since this has
weight $0$, the restriction map is an isomorphism after applying $W_{-1}$. The
depth filtrations of $\Der\p(E')$ and $\Der\p(U)$ induce depth filtrations
on $\Der^N\p(E')$ and $\Der^E \p(U)$, respectively.

The following useful result was proved by Pollack \cite[4.5]{pollack}.

\begin{lemma}[Pollack]
\label{lem:pollack}
If $\delta \in \Der^0\L(A,T)$, then $\delta(A) \in D^d \L(A,T)$ if and only if
$\delta(T) \in D^d\L(A,T)$. \qed
\end{lemma}

Note that it is important that $\delta(\theta) = 0$. The derivation $\delta =
\ad_A$ is not in $D^1\Der\L(A,T)$, even though $\delta(A),\delta(T) \in
D^1\L(A,T)$.

\begin{lemma}
\label{lem:depth}
If $\delta \in \Der^0 \L(A,T)$, then
\begin{align*}
\delta \in D^d\Der^0\L(A,T) &\Longleftrightarrow
\delta(A),\delta(T) \in D^d\L(A,T) \cr
&\Longleftrightarrow
\delta(A) \in D^d\L(A,T) \text{ or } \delta(T) \in D^d\L(A,T).
\end{align*}
\end{lemma}

\begin{proof}
The left to right implications are clear. Pollack's Lemma~\ref{lem:pollack}
implies that the right-hand statement implies the middle one. To complete the
proof, we have to show that the middle statement implies the left-hand
statement. To do this, it suffices to prove that if $\delta(A)$ and $\delta(T)$
are in $D^d\L(A,T)$ and if $u\in D^1\L(A,T)$, then $\delta(u) \in
D^{d+1}\L(A,T)$. Since $D^1\L(A,T)$ is free, and since its abelianization is the
rank 1 free $\Sym \Hdual$ module generated by the class of $\theta$,
$D^1\L(A,T)$ is generated by elements of the form $f(A,T)\cdot\theta$, where
$f(A,T) \in \Q\langle A,T\rangle$. The assumptions imply that $\delta f(A,T) \in
D^d \Q\langle A,T\rangle$. So, since $\delta(\theta)=0$,
$$
\delta \big(f(A,T)\cdot\theta\big)
= \delta\big(f(A,T)\big)\cdot \theta \in D^{d+1}\L(A,T)
$$
from which it follows that $\delta$ takes $D^1\L(A,T)$ into $D^{d+1}\L(A,T)$.
\end{proof}

We can now prove that the elliptic and classical depth filtrations agree.

\begin{proposition}
The restriction mapping $\Der^N\p(E') \to \Der^E\p(U)$ preserves the depth
filtration and is strict with respect to it.
\end{proposition}

\begin{proof}
Identify $X_0$ and $X_1$ with their images in $\L(A,T)^\wedge$. If $\delta\in
\Der^0\L(A,T)$, then
\begin{equation}
\label{eqn:delta_cong}
\delta(X_0) = \delta\left(\frac{T}{e^T-1}\cdot A\right)
= \delta(A) + \sum_{m=1}^\infty \frac{B_{2m}}{(2m)!}\delta(T^{2m-1})\cdot\theta.
\end{equation}
Observe that if $\delta\in D^d\Der^0\L(A,T)$ and $n\ge 1$, then
$$
\delta(T^n\cdot A) = \delta(T^{n-1}\cdot\theta) =\delta(T^{n-1})\cdot \theta
\in D^{d+1}\L(A,T).
$$
So if $\delta \in D^d\Der^N \L(A,T)$, then $\delta(X_0)\in D^d\L(A,T)^\wedge$.
So, by Proposition~\ref{prop:depth_strict}, 
$$
\delta(X_0) \in \L(X_0,X_1)^\wedge \cap D^d \L(A,T)
\in  D^d \L(X_0,X_1)^\wedge.
$$
This implies that the restriction mapping $\Der^N \p(E') \to \Der^E \p(U)$
respects the depth filtrations.

It remains to show that the restriction mapping is strictly compatible with the
depth filtrations. Suppose that the restriction of $\delta \in \Der^N\L(A,T)$ to
$\L(X_0,X_1)^\wedge$ is in $D^d\Der\L(X_0,X_1)$. That is, that $\delta(X_0) \in
D^d\L(X_0,X_1)^\wedge$. Then equation (\ref{eqn:delta_cong}) implies that
$\delta(A) \in D^d\L(A,T)$. Applying Lemma~\ref{lem:depth}, we see that $\delta
\in D^d\Der^0\L(A,T)$, as required.
\end{proof}

So the depth filtration on $\k$ can be computed by pulling back the depth
filtration either from $\Der\p(U)$ (the ``classical'' case) or from $\Der
\p(E')$ (the ``elliptic'' case). Next, we use this to show that the depth
filtration on $\k$ is closely related to the elliptic weight filtration
$W_\dot$.

\begin{theorem}
For all $m\ge 0$, we have
$$
W_{-m}\k = \sum_{n+d\ge m} M_{-2n}\k\cap D^d\k,
$$
so that
$$
\Gr^W_{-m}\k \cong \bigoplus_{n+d=m} \Gr^M_{-2n}\Gr_D^d \k.
$$
\end{theorem}

\begin{proof}
Let $\gamma_1$ be the class of the canonical loop in $\pi_1(\Pminus,\vv_o)$
about $1$ in $\U$. Its logarithm  $\log\gamma_1$ spans a copy of $\Q(1)$ in
$\p(U)$. Define
$$
\Der^0 \p(U) = \{\delta \in \Der\p(U) : \delta(\log \gamma_1) = 0\}.
$$
This is motivic as it is the kernel of the morphism $\Der\p(U) \to \p(U)$ that
takes the derivation $\delta$ to $\delta(\log \gamma_1)$. So it is a pro-object
of $\MTM$ that is filtered by $W_\dot$. The image of the canonical homomorphism
$\k \to \Der\p(U)$ is contained in $\Der^0 \p(U)$.

Choose any element $\gamma_0$ of $\pi_1(U,\vv_o)$ that, together with
$\gamma_1$, generates $\pi_1(U,\vv_o)$. Define a linear map $\phi:\p(U) \to
\Der^0\p(U)$ by $f\mapsto \delta_f$, where $\delta_f(\log\gamma_0) =
[\log\gamma_0,f]$. This has kernel $\Q\log\gamma_0$. Note that $\phi$ is {\em
not} motivic, or even a morphism of MHS. However, its restriction to
$M_{-4}\p(U)$ is strict with respect to the depth filtration and both weight
filtrations as we now show.

Since $\log\gamma_0$ is not in $M_{-4}\p(U)$ and since $M_{-2n}$ is the $n$th
term of the lower central series of $\p(U)$, the restriction
\begin{equation}
\label{eqn:restn}
M_{-4}\p(U) \to \Der^0\p(U)
\end{equation}
of $\phi$ to $M_{-4}\p(U)$ is injective and strictly compatible with $M_\dot$.

The inclusion (\ref{eqn:restn}) is also strict with respect
to the depth filtration $D^\dot$ because, for $f\in M_{-4}\p(U)$,
$$
\delta_f \in D^d\Der\p(U)
\Longleftrightarrow \delta_f(\log\gamma_0)\in D^d\p(U)
\Longleftrightarrow f \in D^d\p(U).
$$
Similarly, when $m\ge 3$, $\delta_f \in W_{-m}\Der\p(U)$ if and only if $f \in
W_{-m} \p(U)$. In other words, (\ref{eqn:restn}) is strict with respect to
$W_\dot$ as well.

It is well known that the image of $\k \to M_{-4}\Der^0\p(U)$ is contained in
the image of $\phi$. Since $M_{-4}\p(U) \to \Der^0 \p(U)$ is injective, this
lifts to a (non-motivic) map $\k \to \p(U)$.  The result now follows from
Proposition~\ref{prop:convolution} and the fact that $\p(U) \to \Der^0\p(U)$ is
strict with respect to the depth filtration and the two weight filtrations.
\end{proof}

Although we cannot ``invert'' the convolution formula, we nonetheless are able
to express the depth graded quotients of $\k$ in terms of the classical and
elliptic weight filtrations.

\begin{corollary}
The depth and elliptic weight filtrations on $\Gr^M_{-2n}\k$ are related by
$$
W_{-m}\Gr^M_{-2n}\k = D^{m-n}\Gr^M_{-2n}\k.
$$
Consequently, there is a natural isomorphism
$\Gr^W_{-m}\Gr^M_{-2n}\k \cong \Gr^M_{-2n}\Gr_D^{m-n}\k$. \qed
\end{corollary}

\section{The Infinitesimal Galois action and Ihara--Takao Congruences}
\label{sec:ihara-takao}

The Ihara--Takao congruence (Theorem~\ref{thm:ihara-takao} below) was proved
numerically by Ihara and Takao in \cite{ihara} and proved again using modular
symbols by Schneps \cite{schneps}. In this section, we use Pollack's relations
to give a more conceptual explanation of the Ihara--Tako congruence, which makes
it clearer why cusp forms impose relations in the depth graded quotients of
$\k$. We begin with a discussion of the infinitesimal Galois action $\phi_E : \k
\to \Der\p(E')$. We continue with the notation of the previous section. In
particular, $U$ denotes $\Pminus$.

The derivations $\e_{2n}$ are highest weight vectors. As we will see below, to
``first order'', the images of the $\zz_{2m+1}$ in $\Der^0\L(\Hdual)$ are {\em
lowest} weight vectors. For each $n\ge 0$, set\footnote{Recall the definition
of $\e_{2n}(\v_1,\v_2)$, equation (\ref{eqn:deltas}).}
$$
\label{def:eop}
\eop_{2n} = \eop_{2n}^\DR := \e_{2n}(A,T) =
\frac{1}{(2n-2)!}\e_0^{2n-2}\cdot\e_{2n}
$$
Its Betti counterpart is
$$
\eop_{2n}^B := \e_{2n}(\aa,-\bb) =
\frac{1}{(2n-2)!}(\e_0^B)^{2n-2}\cdot\e_{2n}^B
= (2\pi i)^{2n-1}\eop_{2n}^\DR.
$$
These are lowest weight vectors for $\sl(H)$.\footnote{The equality
$\eop_{2n}^B = (2\pi i)^{2n-1}\eop_{2n}^\DR$ reflects the fact that $\eop_{2n}$
spans a copy of $\Q(2n-1)$ in $\Gr^W_{-2n+2}$ $\Der\p(E')$.}

We will identify $\k$ with the completion of $\Gr^M_\dot\k$ using the
$\Q$-de~Rham splitting of $M_\dot$. Identify $\p(E')$ with the completion of its
associated $(M_\dot,W_\dot)$ bigraded quotient via the $\Q$-de~Rham splitting.
This gives an isomorphism of $\Gr^M_\dot\p(E')$ with the $T$-adic completion of
$\L(A,T)$. Choose generators of $\zz_{2m+1}$ of $\Gr^M_\dot\k$ that project to
the natural generators (see Section~\ref{sec:gens}) of $H_1(\k)$.

The infinitesimal Galois action $\phi_E : \k \to \Der\p(E')$ is a morphism in
$\MTM$. So it can be identified with the mapping
$$
\phi_E : \L(\ssigma_3,\ssigma_5,\ssigma_7,\dots) \to \Der^0 \L(A,T)^\wedge
$$
induced by applying $\Gr^M$ to $\phi_E$. Since $\phi_E$ is not a morphism in
$\MEM_\uu$, it does not preserve $W_\dot$. For $\ssigma\in \Gr^M_{-4m-2}\k$, let
$\ssigma^\round{r}$ be the component of $\phi_E(\ssigma)$ that lies in
$\Gr^W_{-r}\Gr^M_{-4m-2}\Der^0\L(A,T)$. 

Call an element of $\Der\p(E')$ {\em geometric} if it lies in the image of
$\g^\geom_\uu \to \Der\p(E')$. Theorem~\ref{thm:monod} implies that, after
identifying with the associated bigraded, the Lie algebra of geometric
derivations is the subalgebra of $\Der^0\L(A,T)$ generated by $\eop_0$ and the
$\e_0^j\cdot \e_{2n}$ with $j\ge 0$ and $n\ge 0$. Observe that
$$
\ssigma_{2m+1}^\round{4m+2} \equiv \zz_{2m+1}
\bmod \text{ geometric derivations}
$$
as $\zz_{2m+1} \mapsto \ssigma_{2m+1}$ under $u^\MEM_\uu \to \k$.

\begin{proposition}
\label{prop:geometric}
If $r\neq 4m+2$, then $\ssigma_{2m+1}^\round{r}$ is geometric. The remaining
term $\ssigma_{2m+1}^\round{4m+2}$ is well-defined modulo the image of
$L^3\Gr\u^\geom_\uu$.
\end{proposition}

\begin{proof}
Denote the normalizer of the image of a Lie algebra homomorphism $\a \to
\Der\p(E')$ by $N(\a)$. Since $\g^\MEM_\uu/\g^\geom_\uu \cong \k$, and since
there is a monodromy homomorphism $\g^\MEM \to \Der\p(E')$, the image of
$\phi_E$ lies in $N(\g^\geom_\uu)$. Since $\u^\geom_\uu$ is an ideal of
$\g^\geom_\uu$, since $N(\u^\geom_\uu)/\im\u^\geom_\uu$ is a
$\g^\geom_\uu/\u^\geom_\uu$-module, and since $\g^\geom_\uu/\u^\geom_\uu \cong
\sl(H)$, it follows that
$$
N(\g^\geom_\uu)/\im\g^\geom_\uu \subseteq
\big[N(\u^\geom_\uu)/\im\u^\geom_\uu\big]^{\sl(H)}
$$
After identification with associated bigradeds, we see that, modulo geometric
derivations, the image of $\k$ is contained in $[\Der^0\L(A,T)]^{\sl(H)}$.

Since the $\sl(H)$-invariants have the property that their $M$- and $W$-weights
are equal by (\ref{eqn:weights}), $\ssigma_{2m+1}^\round{r}$ must be geometric
when $r\neq 4m+2$.

The second assertion follows immediately from the fact that
$$
H_1(\u^\geom_1)^{\SL(H)} = \Lambda^2 H_1(\u^\geom_1)^{\SL(H)} = 0.
$$
Since $\Gr^W_r \u^\geom_1 = \Gr^W_r \u^\geom_\uu$ for all $r \neq -2$, this
implies that all $\sl(H)$-invariants in $\Gr^W_{-4m-2} \u_\uu^\geom$ lie in
$L^3\Gr\u^\geom_\uu$.
\end{proof}

The next result is an immediate consequence of the fact that the odd $W$-graded
quotients of $\u^\geom_\uu$ vanish.

\begin{corollary}
\label{cor:expansion}
If $\ssigma_{2m+1}^\round{r}\neq 0$, then $r$ is even and $\ge 2m+2$, so that
$$
\phi_E(\ssigma_{2m+1}) = \sum_{k\ge m+1} \ssigma_{2m+1}^\round{2k}.
$$
\end{corollary}

\begin{remark}
Brown's proof \cite{brown:mtm} of the injectivity of $\phi_\P : \k \to
\Der\p(U)$ and the Oda Conjecture, whose proof was completed by Takao
\cite{takao}, imply that the homomorphism $\k \to
N(\g^\geom_\uu)/\im\g^\geom_\uu$ is injective.
\end{remark}

The fact that the image of $\k$ lies in the centralizer of $N$ gives finer
information about the $\zz_{2m+1}^\round{r}$.

\begin{proposition}
\label{prop:congruences}
For all $m\ge 1$, we have $\phi_E(\ssigma_{2m+1}) \equiv \eop_{2m+2} \bmod D^2$.
\end{proposition}

\begin{proof}
Recall the formula for the monodromy logarithm $N$ from (\ref{eqn:monod_log}).
Since $\phi_E(\ssigma_{2m+1})$ commutes with $N$, the
$\ssigma_{2m+1}^\round{2k}$ satisfy the system of equations
$$
[\e_0,\ssigma_{2m+1}^\round{2k}]
= c_4[\e_4,\ssigma_{2m+1}^\round{2k-4}]
+ c_6[\e_6,\ssigma_{2m+1}^\round{2k-6}]
+ \dots 
+ c_{2k-2m-2}[\e_{2k-2m-2},\ssigma_{2m+1}^\round{2m+2}],
$$
where $c_{2k} = (2k-1)B_{2k}/(2k)!$. This equation determines
$\ssigma_{2m+1}^\round{2k}$ up to an element of $\ker \e_0$ and implies that
$\ssigma_{2m+1}^\round{2m+2}$ is a lowest weight vector of $\Der^0\L(A,T)$.
Since
$$
\Gr^M_{-4m-2}\ker\{\e_0 : H_1(\u^\geom_\uu) \to H_1(\u^\geom_\uu)\}
= \Q\eop_{2m+2},
$$
it follows that $\ssigma_{2m+1}^\round{2m+2}$ is a multiple (possibly zero) of
$\eop_{2m+2}$. It also implies that
$$
W_{-2m-3}\ker\big\{
\e_0 : \Der^\geom_{<0}\L(A,T) \to \Der^\geom_{<0}\L(A,T)\big
\} \subset L^2\Der^\geom_{<0}\L(A,T),
$$
where $\Der^\geom_{<0}\L(A,T)$ denotes the geometric derivations of $\L(A,T)$ of
negative $W$-weight. Combined with Proposition~\ref{prop:geometric}, this and
the equation above imply that $\ssigma_{2m+1}^\round{2k} \in L^2
\Der^\geom_{<0}\L(A,T)$ when $2k > 2m+2$. Since $D^\dot$ is a central filtration
of $\Der\L(A,T)$, $D^2\Der^\geom_{<0}\L(A,T) \supseteq
L^2\Der^\geom_{<0}\L(A,T)$. It follows that
$$
\phi_E(\ssigma_{2m+1}) \equiv \ssigma_{2m+1}^\round{2m+2} \bmod D^2.
$$

It remains to show that $\ssigma_{2m+1}^\round{2m+2} = \eop_{2m+2}$. This can be
proved either by computing the cocycle of the \'etale realization or the period
of a Hodge realization. Nakamura computed the cocycle in \cite[Thms.~3.3 \&
3.5]{nakamura}. We deduce it from Brown's computation \cite[Lem.~7.1]{brown:mmv}
of the period of the Eisenstein series $G_{2m+2}$. His result implies that
$$
\phi_E(\ssigma_{2m+1}) \equiv \frac{(2m)!}{2}\,\im\eeop_{2m+2} \bmod D^2,
$$
where $\eeop_{2m+2}:= \ee_0^{2m}\cdot \ee_{2m+2}/(2m)!$ and $\im$ means its
image in $\Der^0\L(A,T)$. Thus, by Theorem~\ref{thm:monod}, is $\eop_{2m+2}$.
\end{proof}

Brown \cite[Thm.~1.2]{brown:depth3} has computed $\phi_E(\ssigma_{2m+1})$ mod $W_{-2m-5}$ for a certain choice of $\ssigma_{2m+1}$. This gives the second term in the expansion in Corollary~\ref{cor:expansion}.

The $\eop_{2n}$ satisfy Pollack's relations as well. The elliptic analogue of
the Ihara--Takao congruence is now an immediate consequence of
Proposition~\ref{prop:congruences} and Pollack's relations.

\begin{theorem}
\label{thm:elliptic_ihara-takao}
If $n > 0$, then
$$
\sum_{a+b=n} c_a [\phi_E(\ssigma_{2a+1}),\phi_E(\ssigma_{2b+1})]
\equiv 0 \bmod D^3\Der\p(E')
$$
if and only if there is a cusp form $f$ of $\SL_2(\Z)$ of weight $2n+2$ with
$\sr_f^+(x,y) = \sum c_a x^{2a} y^{2n-2a}$.
\end{theorem}

The classical case of the Ihara--Takao congruences now follows from the elliptic
case using the formulas (\ref{eqn:inclusion}).

\begin{theorem}[Ihara--Takao, Goncharov, Schneps]
\label{thm:ihara-takao}
If $n > 0$, then
$$
\sum_{a+b=n} c_a [\phi_\P(\ssigma_{2a+1}),\phi_\P(\ssigma_{2b+1})]
\equiv 0 \bmod D^3\Der\p(U)
$$
if and only if there is a cusp form $f$ of $\SL_2(\Z)$ of weight $2n+2$ with
$\sr_f^+(x,y) = \sum c_a x^{2a} y^{2n-2a}$.
\end{theorem}

\begin{proof}
The formulas (\ref{eqn:inclusion}) and the definition of $\eop_{2n+2}$ imply
that
$$
\eop_{2m+2}(A) = -A^{2m+2}\cdot T = A^{2m+1}\cdot R_1
\text{ and } \eop_{2m+2}(T) \in D^2\L(A,T).
$$
Consequently
\begin{align*}
\eop_{2m+2}(R_0)
&= \eop_{2m+2}\big(A - [T,A]/2 + \frac{1}{12}T \cdot [T,A] + \cdots\big) \cr
&\equiv A^{2m+1}\cdot R_1  &\bmod D^2\p(E') \cr
&\equiv R_0^{2m+1}\cdot R_1  &\bmod D^2\p(E')
\end{align*}
Since $\p(U) \to \p(E')$ is strict with respect to $D^\dot$
(Prop.~\ref{prop:depth_strict}), Proposition~\ref{prop:congruences} implies that
$$
\phi_\P(\ssigma_{2m+1})(X_0) \equiv X_0^{2m+1}\cdot X_1 \bmod D^2\p(E').
$$
Since $\phi_E(\ssigma_{2m+1})(T)\in D^2\L(A,T)$, we have
\begin{align*}
&\phantom{xxxi}\sum_{a+b=n}
c_a [\phi_\P(\ssigma_{2a+1}),\phi_\P(\ssigma_{2b+1})]
\equiv 0 \bmod D^3\Der\p(U)
\cr
& \Longleftrightarrow 
\sum_{a+b=n}
c_a[\phi_\P(\ssigma_{2a+1}),\phi_\P(\ssigma_{2b+1})](X_0)\in D^3\p(U)
\cr
& \Longleftrightarrow 
\sum_{a+b=n}
c_a [\phi_E(\ssigma_{2a+1}),\phi_E(\ssigma_{2b+1})](R_0)\in D^3\p(E')
\cr
& \Longleftrightarrow 
\sum_{a+b=n}
c_a [\phi_E(\ssigma_{2a+1}),\phi_E(\ssigma_{2b+1})](A)\in D^3\p(E')
\cr
& \Longleftrightarrow 
\sum_{a+b=n} c_a [\phi_E(\ssigma_{2a+1}),\phi_E(\ssigma_{2b+1})] \equiv 0
\bmod D^3\Der^0\L(A,T).
\end{align*}
The result now follows from the elliptic case,
Theorem~\ref{thm:elliptic_ihara-takao}.
\end{proof}

\appendix

\section{Relative Weight Filtrations}
\label{sec:rel_wt_filt}

Universal mixed elliptic motives are mixed Tate motives with additional
structure. This extra structure includes a second weight filtration $W_\dot$ and
a nilpotent endomorphism $N$ of $V$ that preserves $W_\dot$. Every universal
mixed elliptic motive $V$ has two weight filtrations: its weight filtration
$M_\dot$ as an object of $\MTM$ and the second weight filtration $W_\dot$. One
of the axioms of a universal mixed elliptic motive is that $M_\dot$ be the relative
weight filtration of the nilpotent endomorphism $N$ of the filtered vector space
$(V,W_\dot)$. In this section we review Deligne's definition
\cite{deligne:weil2} of the relative weight filtration of a nilpotent
endomorphism of a filtered vector space. More information about relative weight
filtrations can be found in \cite{steenbrink-zucker}. A concise exposition is
given in Section~7 of \cite{hain:morita}.

\subsection{The weight filtration of a nilpotent endomorphism}

There is a natural weight filtration of a vector space associated to a nilpotent
endomorphism $N$ of it.

\begin{proposition}
\label{prop:wt-filt}
If $N$ is a nilpotent endomorphism of a finite dimensional vector space $V$ over
a field of characteristic zero, then there is a unique filtration
\begin{multline*}
0 = W(N)_{-m-1} \subseteq W(N)_{-m} \subseteq W(N)_{-m+1} \subseteq \cdots
\cr
\dots \subseteq W(N)_{m-1} \subseteq W(N)_m = V
\end{multline*}
of $V$ such that
\begin{enumerate}

\item for all $n\in \Z$, $NW(N)_n \subseteq W(N)_{n-2}$;

\item for each $k\in \Z$,  $N^k : \Gr_k^{W(N)} V \to \Gr_{-k}^{W(N)} V$ is an
isomorphism.

\end{enumerate}
\end{proposition}

The filtration $W(N)_\dot$ of $V$ is called the {\em weight filtration of $N$}.

Note that $W(N)_\dot$ is centered at $0$. When $V$ is a motive of weight $m$,
it is natural to reindex the weight filtration of a nilpotent endomorphism $N$
of $V$ so that it is centered at $m$. The shifted filtration
$$
M_k V := W(N)_{k-m}
$$
is centered at $m$. The reindexed filtration $M_\dot$ satisfies $N M_k \subseteq
M_{k-2}$ and
$$
\xymatrix{
N^k : \Gr_{m+k}^M V \ar[r]^(0.55)\simeq & \Gr_{m-k}^M V
}
$$
is an isomorphism for all $k \in \Z$. We will call the shifted weight filtration
$M_\dot$ the {monodromy weight filtration} of $N:V\to V$.

\begin{example}
Let $H=\C\bw \oplus \C\aa$, regarded as a vector space of weight 1. Let $N$ be
the nilpotent endomorphism $\aa\frac{\partial}{\partial \bw}$ of $H$. It induces
a nilpotent endomorphism of $V=S^n H$, the space of homogeneous polynomials in
$\aa$ and $\bw$ of degree $n$, which we regard as a vector space of weight $n$.
The shifted monodromy weight filtration $M_\dot$ of $V$ is obtained by giving
$\aa$ weight $0$ and $\bw$ weight 2. The monomial $\aa^{n-j} \bw^j$ has weight
$2j$. Then
$$
M_k V = \text{span of the monomials $\aa^{n-j}\bw^j$ of weight $\le k$}.
$$
\end{example}

\subsection{The weight filtration of a nilpotent endomorphism of a filtered
vector space}

Now suppose that $N$ is a nilpotent endomorphism of a {\em filtered} finite
dimensional vector space $V$ over a field of characteristic zero. That is, $V$
has a filtration
$$
0 \subseteq \dots \subseteq 
W_{m-1}V \subseteq W_m V \subseteq W_{m+1}V \subseteq \dots \subseteq V
$$
which is stable under $N$.

Since $N$ preserves the weight filtration, it induces a nilpotent endomorphism
$$
N_m := \Gr^W_m N : \Gr^W_m V \to \Gr^W_m V.
$$
of the $m$th weight graded quotient of $V$. Proposition~\ref{prop:wt-filt}
implies that each graded quotient has a weight filtration $W(N_m)$. The
reindexed filtration $W(N_m)[m]_\dot$ is centered at $m$. Denote it by
$M^{(m)}_\dot$.

\begin{definition}
A filtration $M_\dot$ of $V$ is called a {\em relative weight filtration} of $N
: (V,W_\dot) \to (V,W_\dot)$  if
\begin{enumerate}

\item for each $k \in \Z$, $NM_k \subseteq M_{k-2}$;

\item the filtration induced by $M_\dot$ on $\Gr^W_m V$ is the reindexed weight
filtration $M^{(m)}_\dot$.

\end{enumerate}
\end{definition}

Relative weight filtrations, if they exist, are unique. (Cf.\
\cite{steenbrink-zucker}).

\begin{example}
\label{ex:trivial}
If $N : (V,W_\dot) \to (V,W_\dot)$ satisfies $N(W_mV)\subseteq W_{m-2}V$ for all
$m\in \Z$, then each $N_m = 0$ and the relative weight filtration $M_\dot$ of
$N$ exists and equals the original weight filtration $W_\dot$.
\end{example}

Even though the weight filtration of a nilpotent endomorphism of a finite
dimensional vector space always exists, the relative weight filtration of a
nilpotent endomorphism of a {\em filtered} vector space $(V,W_\dot)$ does not.
Necessary and sufficient conditions for the existence of a relative weight
filtration are given in \cite{steenbrink-zucker}. They imply that the generic
nilpotent endomorphism of $(V,W_\dot)$ does not have a relative weight
filtration. 

\begin{example}
Let $E$ be a compact Riemann surface of genus 1 and $P,Q$ two distinct points of
$E$. Let $V = H_1(E,\{P,Q\};\Q)$. Then one has the exact sequence
$$
0 \to H_1(E;\Q) \to V \to \widetilde{H}_0(\{P,Q\}) \to 0.
$$
Define a filtration $W_\dot$ on $V$ by $W_{-2}V=0$, $W_{-1}V = H_1(E)$, and $W_0
V = V$. Choose any path $\gamma$ from $P$ to $Q$. It determines a class
$[\gamma]$ in $V$. Let $u$ be any non-trivial element of $H_1(E;\Q)$. Define a
nilpotent automorphism $N$ of $(V,W_\dot)$ by insisting that $N$ be trivial on
$\Gr^W_\dot V$ and that $N[\gamma] = u$. Since $N$ is trivial on $\Gr^W_\dot V$,
the relative weight filtration $M_\dot$, should it exist, would equal $W_\dot$.
But $W_\dot$ is not a relative weight filtration because $N W_0V$ is not 
contained in $W_{-2}V$.
\end{example}

\section{Splitting the Weight Filtrations $M_\dot$ and $W_\dot$}
\label{sec:splittings}

In this section, we show that both weight filtrations of an object of
$\MEM_\ast$ can be simultaneously split and that such splittings are compatible
with tensor products and duals, and are preserved by morphisms in $\MEM_\ast$.
Moreover, such a splitting of $M_\dot$ can be chosen to agree with any given
natural splitting of the weight filtration in $\MTM$. The existence of such
splittings imply that $\Gr^M_\dot$, $\Gr^W_\dot$ and $\Gr^W_\dot\Gr^M_\dot$ are
exact functors on $\MEM_\ast$.

Fix a fiber functor $\w : \MTM \to \Vec_F$, where $F$ is a field of
characteristic zero. This induces the fiber functor
$$
\xymatrix{
\MEM_\ast \ar[r]^{\vv_o} & \MTM \ar[r]^\w & \Vec_F
}
$$
which we also denote by $\w$. We will regard objects of $\MEM_\ast$ as
$F$-vector spaces with an action of $\pi_1(\MEM_\ast,\w)$. Similarly for
mixed Tate motives.

Suppose that we have chosen a natural splitting
$$
V \cong \bigoplus_m \Gr^M_\dot V.
$$
of the weight filtration $M_\dot$ of each object $V$ of $\MTM$. That is, we have
chosen a splitting of the canonical homomorphism $\pi_1(\MTM,\w) \to \Gm$. The
weight filtration $M_\dot$ of a mixed Tate motive $V$ then splits under the
$\Gm$-action.

\begin{proposition}
\label{prop:splittings}
For each $\ast \in \{1,\uu,2\}$, the $F$-vector space $V$ underlying an object
of $\MEM_\ast$ has a bigraded splitting
$$
V = \bigoplus V_{m,n}
$$
in the category of $\Q$-vector spaces with the property that
$$
M_mV = \bigoplus_{\substack{r\le m\cr n\in \Z}} V_{r,n}
\text{ and }
W_nV = \bigoplus_{\substack{r\le n\cr m\in \Z}} V_{m,r}.
$$
This bigrading is natural in the sense that it is preserved by morphisms $\phi :
V \to V'$ of $\MEM_\ast$:
$$
\phi(V_{m,n})\subseteq V_{m,n}'
$$
and also by the functors $\MEM_1 \to \MEM_2 \to \MEM_\uu$. It is compatible with
tensor products and duals. Finally, this splitting of $(V,M_\dot)$ can be
chosen so that it agrees with the splitting of $M_\dot$ when $(V,M_\dot)$ is
viewed as an object of $\MTM$ via the action
$$
\pi_1(\MTM,\w) \overset{s_\vv}{\longrightarrow} \pi_1(\MEM_\ast,\w)
\longrightarrow \Aut V.
$$
\end{proposition}

This result will follow from the following lemma. Set
$$
\diag(t_1,t_2) := \begin{pmatrix} t_1 & 0 \cr 0 & t_2 \end{pmatrix}.
$$
Define a homomorphism $c : \Gm \to \GL_2$ by $c(t) = \diag(1,t)$. This is a
section of $\det : \GL_2 \to \Gm$.

\begin{lemma}
\label{lem:splittings}
Suppose that $F$ is a field of characteristic $0$ and that
$$
\xymatrix{
1 \ar[r] & U \ar[r]\ar[d] & G \ar[r]^{\rho_G}\ar[d] & \GL_2 \ar[r]\ar[d]^\det
 & 1 \cr
1 \ar[r] & K \ar[r] & A \ar@/_1pc/[u]_\sigma\ar[r] & \Gm \ar[r] & 1
}
$$
is a commutative diagram of affine $F$-groups with exact rows and where $K$
and $U$ are prounipotent and $\sigma$ is a section. If $s_A : \Gm \to A$ is a
section of $A\to \Gm$ such that $\rho_G\circ \sigma \circ s_A = c$, then there
exists a section $s_G : \GL_2 \to G$ that extends $\sigma\circ s_A$ in the sense
that the diagram
$$
\xymatrix{
&& \Gm \ar@/_/[dll]_{s_G\circ c} \ar[dl]^c \ar@/^/[ddl]^{\id} \cr
G & \GL_2 \ar[l]_{s_G} \ar[d]_\det \cr
A \ar[u]^\sigma & \Gm \ar[l]^{s_A}
}
$$
commutes. Moreover, any two such sections $s_G$ are conjugate by an element of
$U$ that centralizes $\sigma$.
\end{lemma}

\begin{proof}
Denote $\rho_G^{-1}\im c$ by $G_c$. This is an extension of $\Gm$ by $U$. Choose
a splitting $s : \GL_2 \to G$ of $\rho_G$. Then $s$ induces a splitting $s' :
\Gm \to G_c$. Since $\rho_G\circ \sigma \circ s_A = c$, the image of $\sigma$
lies in $G_c$. Levi's theorem implies that there is an element $u$ of $U$ such
that $\sigma = us' u^{-1}$. Now define $s_G = u s u^{-1}$. Levi's Theorem
applied to $G$ implies that any two such sections $s_B$ are conjugate by an
element of $U$ that centralizes with the image of $\sigma$.
\end{proof}

\begin{proof}[Proof of Proposition~\ref{prop:splittings}]
Recall that we have fixed an ordered basis $\bb,\aa$ of $\Hdual$. This gives an
identification of $\GL_2$ with $\GL(\Hdual)$. Denote the diagonal torus by $T$:
$$
\diag(t_1,t_2) : \bb \mapsto t_1 \bb \text{ and }
\diag(t_1,t_2) : \aa \mapsto t_2 \aa.
$$
Now take $A = \pi_1(\MTM,\w)$, $G=\pi_1(\MEM_\ast,\w)$, $s_A$ to be any
splitting (such as the $\Q$ de~Rham splitting), and $\sigma = s_\vv$. The
summand $V_{2m,n}$ of an object of $\MEM_\ast$ is the subspace of $V_\Q$ on
which $(t_1,t_2)\in T$ acts by $t_1^{m-n}t_2^{-m}$.
\end{proof}

Such bigraded splittings of $M_\dot$ and $W_\dot$ correspond to homomorphisms
$\Gm\times\Gm \to \pi_1(\MEM_\ast)$ whose composition with the natural
surjection $\pi_1(\MEM_\ast)\to \GL(H)$ is the inclusion of the maximal torus
associated to the basis $\aa,\bb$ of $H$. Since $V_{m,n} \cong \Gr^M_m\Gr^W_n
V$, we have:

\begin{corollary}
\label{cor:exactness}
Each mixed elliptic motive $V$ is naturally isomorphic to its associated
bigraded
$$
V\cong \bigoplus_{m,n} \Gr^M_m\Gr^W_n V
$$
in the category of $\Q$-vector spaces. Consequently, the functors $\Gr^M_\dot$,
$\Gr^W_\dot$ and $\Gr^M_\dot\Gr^W_\dot$ are exact functors from $\MEM_\ast$ to
the category of (bi)graded $\Q$-vector spaces and are compatible with
$\bigotimes$ and $\Hom$.
\end{corollary}

The Lie algebra $\g^\MEM_\ast$ of $\pi_1(\MEM_\ast)$ is a pro-object of
$\MEM_\ast$. Its weight filtration $W_\dot$ satisfies
$$
\g^\MEM_\ast = W_0 \g^\MEM_\ast,\ W_{-1}\g^\MEM_\ast \text{ is prounipotent, and
} \Gr^W_0 \g^\MEM_\ast \cong \gl(H).
$$
It follows that if $V$ is an object of $\MEM_\ast$, then each $\Gr^W_\dot V$ is
a $\gl(H)$-module. The spaces can be split according to their $\sl(H)$-weight.
To make this precise, we will view $\sl(H)$ as a subalgebra of $\End(\Hdual)$.
We fix a Cartan by  specifying that $\aa$ has $\sl(H)$-weight 1 and $\bb$ has
$\sl(H)$-weight $-1$.

\begin{lemma}
If $V$ is an object of $\MEM_\ast$, then, under the identification of $V$ with
$\bigoplus\Gr^M_\dot\Gr^W_\dot V$, the $\sl(H)$-weight of $V_{m,n}$ is $m-n$.
\end{lemma}

In other words, the three notions of weight in $\Gr^M_\dot\Gr^W_\dot V$ are
related by
$$
M\text{-weight} = \sl(H)\text{-weight} + W\text{-weight}.
$$

\section{Index of Principal Notation}

\begin{tabbing}
loooongnotation \= an explanation ---- quite looooong xxxxxxxxxxxxxxxxxxxx
then
\= page reference \kill


$E$ \> an elliptic curve, often the fiber of the Tate curve over $\tate$ \> \\

$E'$ \> $E-\{0\}$ \> \\

$e_o$ \> the unique cusp $q=0$ of $\Mbar_{1,1}$ \> p.~\pageref{def:cusp}\\

$\tate$ \> the integral tangent vector $\partial/\partial q$ of $\Mbar_{1,1}$ at
the cusp $e_o$ \>  p.~\pageref{sec:tangent} \\

$\ww_o$ \> the integral tangent vector $\partial/\partial w$ of $1$ in $\P^1$
\> p.~\pageref{sec:tangent}\\

$\vv_o$ \> the generic notation for the integral tangent vector of
$\Mbar_{1,\ast}$\> p.~\pageref{sec:tangent}\\

$\H$ \> the local system $R^1f_\ast\Q$ associated to $f : \E \to \M_{1,1}^\an$
\> p.~\pageref{sec:localsys_H} \\

$\cH$ \> the canonical extension of $\H\otimes \O_{\M_{1,1}^\an}$ to
$\Mbar_{1,1}^\an$ \> p.~\pageref{sec:connection} \\

$\nabla_0$ \> the canonical connection on $S^m \cH$ \>
p.~\pageref{sec:connection} \\

$H$ \> the fiber of $\cH$ over $\vv_o$, an object of $\MTM$
\> p.~\pageref{def:H} \\ 

\\

$\MTM$ \> the category of mixed Tate motives over $\Z$ \>
p.~\pageref{def:MTM} \\

$\MEM_\ast$ \> the category of universal mixed elliptic motives of type $\ast$
\> p.~\pageref{def:MEM} \\

$\MEM^\ss_\ast$ \> the category of semi-simple universal mixed elliptic motives
\> p.~\pageref{def:MEMss} \\


$\MHS(\M_{1,\ast},\H)$ \> the category of admissible VMHS over $\M_{1,\ast}$
generated by $\H$ \> p.~\pageref{def:MHS_H} \\

\\ 

$\K$ \> the prounipotent radical of $\pi_1(\MTM)$ \> p.~\pageref{def:K} \\

$\k$ \> the Lie algebra of $\K$ \> p.~\pageref{def:K} \\

$\pi_1^\geom(\MEM_\ast)$ \>
the kernel of $\pi_1(\MEM_\ast) \to \pi_1(\MTM)$ \> p.~\pageref{def:pigeom} \\


$\cG_\ast^\rel$ \> the relative completion of $\pi_1(\M_{1,\ast}^\an,\vv_o)$
\> p.~\pageref{def:Grel} \\

$\U_\ast^\rel$ \> its prounipotent radical \> p.~\pageref{def:Grel} \\

$\U^\MEM_\ast$ \> prounipotent radical of $\pi_1(\MEM_\ast)$ \>
p.~\pageref{def:umem}\\

$\U^\geom_\ast$ \> the prounipotent radical of $\pi_1^\geom(\MEM_\ast)$
 \> p.~\pageref{def:ugeom}\\
 
$\u^\geom_\ast$ \> the Lie algebra of $\U^\geom_\ast$ \> p.~\pageref{def:ugeom}
\\

$\cGhat_\ast$ \> $\pi_1(\MHS)\ltimes \cG^\rel_\ast$ \> p.~\pageref{def:Ghat} \\

$\cG^{\cris,\ell}_\ast$ \> the $\ell$-adic crystalline completion of
$\pi_1(\M_{1,\ast/\Z[1/\ell]},\vv_o)$ \> p.~\pageref{def:Gcris} \\

$\cG_\ast^\MEM$ \> notation for $\pi_1(\MEM_\ast)$\> p.~\pageref{def:GMEM} \\


\\  

$\psi_{2n}$ \> differential form associated to the Eisenstein series $G_{2n}$
\> p.~\pageref{def:psi} \\

$V_f$ \> real 2-dimensional Hodge structure associated to eigenform $f$ \>
p.~\pageref{def:V_f} \\

$M_f$ \> the simple $\Q$-Hodge structure associated to the eignform $f$ \>
p.~\pageref{def:M_f}\\

$\Fr_\infty$ \>
involution induced by complex conjugation on $\M_{1,\ast}^\an$ and $\H$\>
p.~\pageref{lem:real_frob}\\

$\Frbar_\infty$ \>
$\F_\infty$ composed with complex conjugation on $\H_\C$ \>
p.~\pageref{lem:real_frob} \\

$V^\pm$ \> the eigenspaces of the involution $\Frbar_\infty: V \to V$ \>
p.~\pageref{lem:real_frob} \\

$\sr_f$ \> modular symbol of the cusp form $f$ \> p.~\pageref{def:mod_symb}\\

$\sr_f^\pm$ \> the even and odd degree parts of $\sr_f$
\> p.~\pageref{def:mod_symb}\\

\\ 

$\aa,\bb$ \> $\Q$-Betti basis of $H$ \>  p.~\pageref{sec:orbifolds} \\

$\aa,\bw$ \> $\Q$-de~Rham basis of $H$ \> p.~\pageref{sec:connection} \\

$\Hdual$ \> $H(1)$ \> p.~\pageref{def:Hdual} \\

$A, T$ \> $\Q$-de~Rham basis of $\Hdual$ \> p.~\pageref{def:A+T}\\

\\ 

$\L(V)$ \> free Lie algebra generated by a vector space $V$ \> 
p.~\pageref{def:LV} \\

$\ee_0$ \> the weight lowering nilpotent in $\sl(H)$ \> 
p.~\pageref{eqn:e_0} \\

$\ee_{2n}$ \> the generator of $\u^\MEM_\uu$ dual to $G_{2n}$ when $n>0$
\> p.~\pageref{def:e_2n}\\

$\zz_{2m-1}$ \> a generator of $\k$ and also certain of its lifts to
$\u^\MEM_\ast$ \> p.~\pageref{def:zz} \\

$\f_\ast$ \> the free Lie algebra generated by $H_1(\u_\ast^\MEM)$
\> p.~\pageref{def:f}\\

$\f_\ast^\geom$ \> the free Lie algebra generated by $H_1(\u_\ast^\geom)$
\> p.~\pageref{def:f}\\

$L^r\g$ \> the $r$th terms of the lower central series of the Lie algebra
$\g$ \> p.~\pageref{def:lcs} \\

$\Der^0\L(A,T)$ \> the derivations of $\L(A,T)$ that annihilate $[T,A]$ \> 
p.~\pageref{prop:der0}\\

$\e_{2n}(\v_1,\v_2)$ \> a derivation of $\L(H)$ depending on $\v_1,\v_2\in H$
\> p.~\pageref{eqn:deltas} \\

$\e_{2n}$ \> $\e_{2n}(T,A)$ \> p.~\pageref{eqn:deltas} \\

$\eop_{2n}$ \> $\e_{2n}(A,T)$ \> p.~\pageref{def:eop} \\

\\

$\p(E')$ \> the Lie algebra of $\pi_1^\un(E',\ww_o)$ \>
p.~\pageref{sec:depth} \\

$\p(U)$ \> the Lie algebra of $\pi_1^\un(\Pminus,\ww_o)$ \>
p.~\pageref{sec:depth} \\

$\Der^N\p(E')$ \> derivations that commute with $N$ and kill
cuspidal loop \> p.~\pageref{def:derN}\\

$\Der^E\p(U)$ \> the derivations of $\p(U)$ that extend to
$\p(E')$, etc.\> p.~\pageref{def:derE} \\

\end{tabbing}

\end{document}